\documentclass{article}
\usepackage[utf8]{inputenc}
\usepackage{graphicx}
\usepackage[normalem]{ulem}
\usepackage{mathrsfs}
\usepackage{amssymb}
\usepackage{amsfonts, mathtools}
\usepackage{amssymb,amsmath,amsthm,tikz-cd, tikz}
\usepackage{comment}
\usepackage[shortlabels]{enumitem}
\usepackage{xcolor}
\usepackage{hyperref}
\hypersetup{
    colorlinks,
    linkcolor={blue!80!black},
    citecolor={blue!80!black},
    urlcolor={blue!80!black}
}
\usepackage{caption}
\usepackage{subcaption}
\usetikzlibrary{shapes.geometric}
\usepackage{wasysym}
\usepackage{pifont}
\usetikzlibrary{calc}
\usetikzlibrary{decorations.markings}

\usepackage{tikz-3dplot}

\tdplotsetmaincoords{60}{120} 

\title{Quiver Hecke algebras from Floer homology
in Couloumb branches}
\author{Mina Aganagic${}^{1,2}$, Ivan Danilenko${}^{1}$, Yixuan Li${}^{1}$, Vivek Shende${}^{1,3}$, Peng Zhou${}^{1}$}

\def\MCB{\mcal^\times}
\def\ACB{\mcal^+}

\def\T{\tcal}

\def\pt{\text{pt}}
\def\GK{G(\K)}
\def\GO{G(\O)}

\def\cal{\mathcal}
 
\def\acal{\mathcal A}
\def\bcal{\mathcal B}
\def\ccal{\mathcal C}

\def\fcal{\mathcal F}
\def\hcal{\mathcal H}
\def\mcal{\mathcal M}

\def\ocal{\mathcal O}

\def\pcal{\mathcal P}

\def\tcal{\mathcal T}

\def\ycal{\mathcal Y}
\def\wcal{\mathcal W}

\def\tfrak{\mathfrak t}

\def\K{\mathcal{K}}
\def\O{\mathcal{O}}

\def\B{\mathbb{B}}

\def\R{\mathbb{R}}
\def\C{\mathbb{C}}
\def\Z{\mathbb{Z}}
\def\N{\mathbb{N}}
\def\P{\mathbb{P}}

\def\CS{\C^*}

\def\S{\mathcal{S}}
\def\D{\mathcal{D}}
\def\bD{\mathbb{D}}

\def\In{\subset}
\def\la{\langle}
\def\d{\partial}
\def\ra{\rangle}
\def\wt{\widetilde}
\def\wb{\overline}

\def\RM{\backslash}
\def\into{\hookrightarrow}
\def\xto{\xrightarrow}

\def\congto{\xrightarrow{\sim}}
\def\bfT{{\bf T}}

\def\bfL{{\bf L}}
\def\bfW{{\bf W}}
\def\bfR{{\bf R}}
\def\bfTh{{\bf \Theta}}

\DeclareMathOperator{\Bun}{Bun}
\DeclareMathOperator{\Sym}{Sym}
\DeclareMathOperator{\End}{End}
\DeclareMathOperator{\Coh}{Coh}

\DeclareMathOperator{\Hom}{Hom}

\DeclareMathOperator{\Fuk}{Fuk}

\DeclareMathOperator{\Spec}{Spec}

\DeclareMathOperator{\colim}{colim}

\usepackage{color}

\newtheorem{theorem}{Theorem}[section]
\newtheorem{lemma}[theorem]{Lemma}
\newtheorem{corollary}[theorem]{Corollary}
\newtheorem{proposition}[theorem]{Proposition}
\newtheorem{definition}[theorem]{Definition}

\newtheorem{example}[theorem]{Example}
\newtheorem{remark}[theorem]{Remark}
\newtheorem{conjecture}[theorem]{Conjecture}

\def\rij{{[i,j]}}
\def\bij{{(i,j)}}
\def\LRA{\Leftrightarrow}
\def\cp{{\C\pcal}}
\def\colim{\text{colim}}
\begin{document}

\maketitle
\begin{center}
{\it
$^{1}$ Departments of Mathematics, University of California, Berkeley, USA\\
$^{2}$Department of Physics,  University of California, Berkeley, USA\\
$^{3}$ Center for Quantum Mathematics, Syddansk University, Odense, Denmark}
\end{center}
 
\begin{abstract}
Homology theories categorifying quantum group link invariants are 
    known to be governed by the representation theory
    of quiver Hecke algebras, also called KLRW algebras. 
    Here we show that certain
    cylindrical 
    KLRW algebras, relevant in particular for cylindrical generalizations of link homology theories, can be realized 
    by Lagrangian Floer homology in
    multiplicative
    Coulomb branches. This confirms a homological mirror symmetry
    prediction of the first author. 
\end{abstract}

\tableofcontents

\newpage

\section{Introduction}

Let $\Gamma$ be a directed graph with  no multiple edges, or loops. Pick a natural number $d_i \ge 0$ for each vertex $i$ of $\Gamma$, and a collection $F$ of points on a line, labeled by vertices of $\Gamma$ and possibly empty.  In diagrams, we depict the points of $F$ as red. The KLRW category 
$\mathcal{C}_{\Gamma, \vec{d}, F}$ is defined as follows 
: 
\begin{itemize}
    \item Objects are collections of 
    points on a line $\R$, all distinct
    from the points of $F$, with
    $d_i$ points labeled by the vertex $i$
    of $\Gamma$.  In diagrams, we depict
    these points as black. 
    \item Morphism spaces
    are generated by strand diagrams in
    the plane $\R \times [0,1]$, with no either horizontal or vertical
    tangencies, or non-generic intersections. The strands
    may be decorated by 
    dots. 
    \item 
    Composition $D_1 \circ D_2$ is given by stacking the diagram $D_1$ on top of $D_2$.
    \item 
    Diagrams are considered
    up to isotopy and satisfy relations
    in Figure \ref{fig:KLRW}. 
\end{itemize}
This category is relevant for categorification of $U_{\mathfrak{q}}({\mathfrak{g}})$, where $\mathfrak{g}$ has Dynkin diagram $\Gamma$ \cite{Khovanov-Lauda-diagrammatics-1, Khovanov-Lauda-diagrammatics-2, Rouquier-2kac, Webster-weighted,Kang-Kashiwara}. In particular, when $\Gamma$ is of ADE type, it leads to corresponding invariants of links in ${\mathbb R}^3$ \cite{Webster}.

We will be interested in a cylindrical variant 
$\mathcal{C}^{cyl}_{\Gamma, \vec d, F}$, which 
is the analogous structure with 
the line $\R$ replaced by a circle, and
the plane correspondingly replaced 
by a cylinder.  $\mathcal{C}^{cyl}_{\Gamma, \vec d, F}$ leads to homological invariants of links in  ${\mathbb R}^2\times S^1$ \cite{webster2022coherent}.

We will (re)discover 
the cylindrical KLRW categories 
as Fukaya-Seidel categories
of multiplicative Coulomb branches 
of quiver gauge theories whose quiver is $\Gamma$.
This identification, along with the specific form of the superpotential, was proposed in \cite{aganagic-knot-1, aganagic-knot-2, aganagic-icm}.

\begin{figure}
    \centering
     \begin{subfigure}[b]{0.1\textwidth}
     \centering     
     \begin{tikzpicture}[very thick]
        \def\br{40}
        \draw (0,0) to [out=\br, in=-\br] node[pos=0, below ]{$(i)$} (0,2);
        \draw (0.5,0) to [out=180-\br, in=\br - 180] node[pos=0, below ]{$(i)$} (0.5,2);
        \node at (1,1) {$=0$};
        \end{tikzpicture}
        \caption{bigon}\label{fig:KLRW-bigon}
     \end{subfigure} 
     \hspace{2cm}
      \begin{subfigure}[b]{0.3\linewidth}
     \centering     
     \begin{tikzpicture}[very thick]
     \begin{scope}
        \def\br{40}
        \draw (0,0) to [out=\br, in=-\br] node[pos=0, below ]{$(j)$} (0,2);
        \draw (0.5,0) to [out=180-\br, in=\br - 180] node[pos=0, below ]{$(i)$} (0.5,2);
        \end{scope}
        \node at (1,1) {$= \eta $};
        \begin{scope}[shift=({1.5,0})]
        \def\br{40}
        \draw (0,0) to   node[pos=0, below ] {$(j)$} (0,2);
        \draw (0.5,0) to node[pos=0.5, fill, circle, inner sep=2pt]{} node[pos=0, below ]{$(i)$} (0.5,2);
        \end{scope}
        \node at (2.5,1) {$- \eta $};
        \begin{scope}[shift=({3,0})]
        \def\br{40}
        \draw (0,0) to  node[pos=0.5, fill, circle, inner sep=2pt]{} node[pos=0, below ] {$(j)$} (0,2);
        \draw (0.5,0) to   node[pos=0, below ]{$(i)$} (0.5,2);
        \end{scope}
        \end{tikzpicture}
        \caption{bigon with neighbor $(j) \to (i)$ }
     \label{fig:KLRW-ij}\end{subfigure} 
     \hspace{2cm}
    \begin{subfigure}[b]{0.2\textwidth}
     \centering     
     \begin{tikzpicture}[very thick]
        \def\br{20}
        \begin{scope}
        \draw[red] (0,0) to node[pos=0, below ]{$[i]$} (0,2);
        \draw (0.4,0) to [out=90, in=-90] node[pos=0, below ]{$(i)$} (-0.4,1);
        \draw (-0.4, 1)  to [out=90, in=-90]  (0.4, 2);
        \end{scope}
        \node at (0.7,1) {$=\; \eta$};
        \begin{scope}[shift =({1.5,0})]
        \draw[red] (0,0) to node[pos=0, below ]{$[i]$} (0,2);
        \draw (0.5,0) to node[pos=0.5, fill, circle, inner sep=2pt]{}  node[pos=0, below ]{$(i)$} (0.5,2);
        \end{scope}
        \end{tikzpicture}
        \caption{bigon with red}\label{fig:KLRW-bigon-red}
     \end{subfigure}
     \vskip 1cm
          \begin{subfigure}[b]{0.4\linewidth}
     \centering     
       \begin{tikzpicture}[very thick]
        \def\br{40}
        \begin{scope}
        \draw (0,0) -- node[pos=0, below ]{$(i)$}  (1,2);
        \draw (1,0) -- node[pos=0, below ]{$(i)$}  (0,2);
        \draw (0.5,0) to [bend left] node[pos=0, below ]{$(j)$} (0.5,2);
        \end{scope}
        \node at (1.5,1) {$-$};
        \begin{scope}[shift=({2,0})]
        \draw (0,0) -- node[pos=0, below ]{$(i)$}  (1,2);
        \draw (1,0) -- node[pos=0, below ]{$(i)$}  (0,2);
         \draw (0.5,0) to [bend right] node[pos=0, below ]{$(j)$} (0.5,2);
        \end{scope}
 \node at (3.5,1) {$= \eta \hbar $};
 \begin{scope}[shift=({4,0})]
        \draw (0.1,0) -- node[pos=0, below ]{$(i)$}  (0.1,2);
        \draw (1.1,0) -- node[pos=0, below ]{$(i)$}  (1.1,2);
         \draw (0.6,0) to node[pos=0, below ]{$(j)$} (0.6,2);
        \end{scope}
        \end{tikzpicture}
        \caption{braid with neighbour $(j) \to (i)$.}
     \label{fig:KLRW-jij}\end{subfigure} 
     \hspace{2cm}
         \begin{subfigure}[b]{0.3\textwidth}
     \centering     
       \begin{tikzpicture}[very thick]
        \def\br{40}
        \begin{scope}
        \draw (0,0) to [bend right]  node[pos=0, below ]{$(i)$}  (1,2);
        \draw (1,0) to [bend right] node[pos=0, below ]{$(i)$}  (0,2);
        \draw[red] (0.5,0) to node[pos=0, below ]{$[i]$} (0.5,2);
        \end{scope}
        \node at (1.5,1) {$-$};
        \begin{scope}[shift=({2,0})]
        \draw (0,0) to [bend left] node[pos=0, below ]{$(i)$}  (1,2);
        \draw (1,0) to [bend left] node[pos=0, below ]{$(i)$}  (0,2);
        \draw[red] (0.5,0) to node[pos=0, below ]{$[i]$} (0.5,2);
        \end{scope}
 \node at (3.5,1) {$=\eta \hbar $};
 \begin{scope}[shift=({4.1,0})]
        \draw (0,0) -- node[pos=0, below ]{$(i)$}  (0,2);
        \draw (1,0) -- node[pos=0, below ]{$(i)$}  (1,2);
        \draw[red] (0.5,0) to node[pos=0, below ]{$[i]$} (0.5,2);
        \end{scope}
        \end{tikzpicture}
        \caption{braid with red}\label{fig:KLRW-braid-red}
     \end{subfigure} 
  \vskip 1cm
          \begin{subfigure}[b]{0.4\textwidth}
     \centering     
       \begin{tikzpicture}[very thick]
        \def\br{40}
        \begin{scope}
        \draw (0,0) -- node[pos=0, below ]{$(i)$}  (1,2);
        \draw (1,0) -- node[pos=0, below ]{$(i)$} node[pos=0.8, inner sep=2pt, circle, fill]{}  (0,2);
        \end{scope}
        \node at (1.5,1) {$-$};
        \begin{scope}[shift=({2,0})]
        \draw (0,0) -- node[pos=0, below ]{$(i)$}  (1,2);
        \draw (1,0) -- node[pos=0, below ]{$(i)$}  (0,2);
         \draw (1,0) -- node[pos=0, below ]{$(i)$} node[pos=0.2, inner sep=2pt, circle, fill]{}  (0,2);
        \end{scope}
        \node at (3.5,1) {$= \;\; \hbar $};
         \begin{scope}[shift=({4.3,0})]
        \draw (0,0) -- node[pos=0, below ]{$(i)$}  (0,2);
        \draw (1,0) -- node[pos=0, below ]{$(i)$}  (1,2);
        \end{scope}
        \end{tikzpicture}
        \caption{dot-pass-crossing}\label{fig:KLRW-skein-1}
     \end{subfigure} \hfill
         \begin{subfigure}[b]{0.4\textwidth}
     \centering     
       \begin{tikzpicture}[very thick]
        \def\br{40}
        \begin{scope}
        \draw (0,0) -- node[pos=0.2, inner sep=2pt, circle, fill]{}  node[pos=0, below ]{$(i)$}  (1,2);
        \draw (1,0) -- node[pos=0, below ]{$(i)$}  (0,2);
        \end{scope}
        \node at (1.5,1) {$-$};
        \begin{scope}[shift=({2,0})]
        \draw (0,0) -- node[pos=0.8, inner sep=2pt, circle, fill]{}  node[pos=0, below ]{$(i)$}  (1,2);
        \draw (1,0) -- node[pos=0, below ]{$(i)$}  (0,2);
         \draw (1,0) -- node[pos=0, below ]{$(i)$}  (0,2);
        \end{scope}
        \node at (3.5,1) {$= \;\; \hbar $};
         \begin{scope}[shift=({4.3,0})]
        \draw (0,0) -- node[pos=0, below ]{$(i)$}  (0,2);
        \draw (1,0) -- node[pos=0, below ]{$(i)$}  (1,2);
        \end{scope}
        \end{tikzpicture}
        \caption{another dot-pass-crossing}\label{fig:KLRW-skein-2}
     \end{subfigure} 
         
    \caption{Nontrivial KLRW relations. Exchanging $i$ and $j$ in diagrams (b) and (d), i.e. if we have an arrow $(j) \gets (i)$, the right-hand-side gets an extra (-1) factor.  
    }
    \label{fig:KLRW} 
\end{figure}
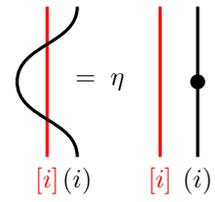
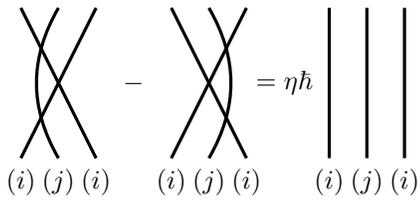
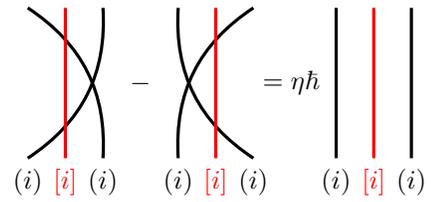
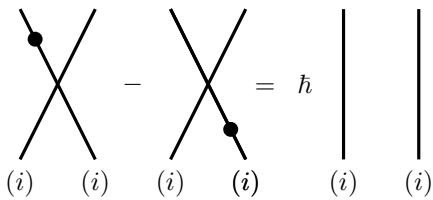
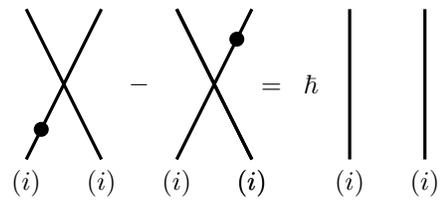

From the data $(\Gamma, {\vec d}, F)$, 
we consider the following group and vector spaces. 
\begin{equation} 
\label{group and rep}
\mathbf G := \prod_{i \in \Gamma} GL(\C^{d_i}), \quad \mathbf N := \bigoplus_{i \to j} \Hom(\C^{d_i}, \C^{d_j}), \quad \mathbf N_F := \bigoplus_{i } \Hom(\C^{m_i}, \C^{d_i}),  
\end{equation}
Here, the integers $m_i$ are the number of points in $F$ labeled by the $i$'th vertex of $\Gamma$. 

The space $\mathbf N$ parameterizes representations of the quiver $\Gamma$ with dimension vector $\vec d$; automorphisms of such representations are given by the action of $\mathbf G$.  
Meanwhile $\mathbf N \oplus \mathbf N_F$ parameterizes 
representations of a framed quiver $\overline{\Gamma}$, with dimension vector provided by $\vec d$ and $F$.

In general, from any  reductive Lie 
group $\mathbf{G}$ and representation $\mathbf N$, 
 the mathematical work  \cite{BFN}
 defines spaces 
 $\ACB(\mathbf G, \mathbf N)$
 and $\MCB(\mathbf G, \mathbf N)$, termed
 the 
 `additive' and `multiplicative' Coulomb branches, and expected
 to capture the appropriate
 physical moduli of vacua of 
 certain corresponding 3d $N=4$
 gauge theories. 
 The additive 
Coulomb branches  $\ACB(\mathbf G,\mathbf N)$ include 
many well studied spaces in representation theory, such as cotangent bundles of
flag varieties, slices in affine Grassmanians, and Hilbert schemes of points on $A_n$ surfaces.  These affine algebraic
varieties
carry holomorphic symplectic structures
on their smooth loci and maps $\ACB(\mathbf G,\mathbf N) \to \mathfrak{t}/W$,
and
 $\MCB(\mathbf G,\mathbf N) \to \mathbf T/W$ and 
where $\mathbf T \subset \mathbf G$ is a maximal torus, ${\bf t}$ is its Lie algebra, and $W$ is
the Weyl group of $\mathbf G$.

We will write $\MCB(\Gamma, \vec d) := \MCB(\mathbf G,\mathbf N)$.  
We will also use the Coulomb branch of the framed quiver
$\MCB(\mathbf G,\mathbf N \oplus \mathbf N_F)$ in order to describe a certain function 
${\cal W}_F: \MCB(\Gamma, \vec d) \rightarrow {\mathbb C}$, along with some special coordinates on  $\MCB(\Gamma, \vec d)$. 

The purpose of the present article is to construct an embedding of $\mathcal{C}^{cyl}_{\Gamma, \vec d, F}$ in  the Fukaya-Seidel category $Fuk(\MCB(\Gamma, \vec d), {\cal W}_F)$, as conjectured in  \cite{aganagic-knot-2}.

\subsection{Geometry of Coulomb branches for quiver gauge theory}

\subsubsection{Coordinates and superpotential}

We fix a maximal torus ${\bf T} \subset \mathbf G$ along
with an isomorphism 
$\mathbf T \cong \prod_i \prod_{\alpha = 1}^{d_i} \C^*$.  For $y \in \mathbf T$, we 
write its coordinates as
$y_{i, \alpha} \in \C^*$.  

\begin{definition}\label{def:1.1}
Fix 
$a \in \mathbf{T}_F = \prod_{i} \prod_{\alpha = 1}^{m_i} \C^*$
so that the coordinate entries 
$a_{i, \alpha} \in \C^*$ have distinct arguments. 
We write:  $\mathbf{T}_{O} \subset \mathbf{T}$ for
the complement of the following hyperplanes:
\begin{enumerate}
\item \label{intro root hyperplane} $H_{R, i}$ for the locus where some of the $y_{i, \alpha}$ coincide, and $H_R := \bigcup H_{R, i}$. 
\item $H_{I, j \to i}$, if there is an arrow $j \to i$, for the locus where some $y_{i, \alpha}$ coincides
with some $y_{j, \beta}$
\item $H_{F, i}$ for the locus where some $y_{i, \alpha}$ coincides with some $a_{i, \beta}$.  
\end{enumerate}
\end{definition}
The locus $\mathbf T_{O}$ is evidently $W$ invariant, and descends to the quotient.  We define by fiber products: 
$$
\begin{tikzcd}
\wt \MCB(\Gamma, \vec d) \ar[r] \ar[d] &  \MCB(\Gamma, \vec d)\ar[d] & \qquad & \wt \MCB(\Gamma, \vec d)_O \ar[r] \ar[d] & \MCB(\Gamma, \vec d)_O  \ar[d] \\
\mathbf T \ar[r] & \mathbf T / W & \qquad & \mathbf T_O \ar[r] & \mathbf T_O/W
\end{tikzcd}
$$

\noindent{}We will explicitly construct coordinates trivializing the fibration, over the torus $\mathbf T_O$.  These coordinates are chosen so that our desired superpotential function has a simple expression: 

\begin{theorem} \label{trivialization-intro} 
For  $a \in \mathbf T_F$ as in Definition \ref{def:1.1},
    there is a $W$-equivariant isomorphism  
    $$(u, y):  \wt \MCB(\Gamma, \vec d)_O \to \mathbf T^\vee \times \mathbf T_{O}$$
    trivializing the fibration 
    $y:  \wt \MCB(\Gamma, \vec d)_O \to \mathbf T_{O}$, 
    such that 
    $$\mathcal{W} := \sum_{i, \alpha} u_{i,\alpha}$$
    extends to a $W$-invariant regular function on $\wt \MCB(\Gamma, \vec d)$. 
\end{theorem}
\begin{proof}
 As we recall in  Section \ref{localization section}, 
 it follows immediately from the results of \cite{BFN-slice} that there is such a trivialization $(x, y)$ for coordinates `$x$' coming via a comparison to the Coulomb branch for the Cartan.  The $\mathcal{W}$ is a certain  `monopole operator' (in particular a global function) whose expression in the $x$ variables 
 was calculated in \cite{Finkelberg-Tsymbaliuk} to be a sum of terms, each of which is a product of one of the `$x$' variables and a term invertible away from the loci in Def. \ref{def:1.1}; see  Prop. \ref{p:Pn-class} below.  The `$u$' variables are defined (Def. \ref{d:u-variable}) to be the terms in this sum.  The fact that $(u, y)$ gives an isomorphism then follows immediately; it is stated in Proposition \ref{removing framing} and Proposition \ref{p:u-coord} below. 
\end{proof}

Note that $u$ and hence $\mathcal{W}$ depend on $a$,
as does the subset $\mathbf T_O \subset \mathbf T$, and thus also 
$\MCB(\Gamma, \vec d)_O \subset \MCB(\Gamma, \vec d)$.  However, 
 $\MCB(\Gamma, \vec d)$ itself does not depend on $a$.

The function
$\mathcal{W}$ is the single valued part of the `upstairs superpotential' introduced in 
\cite[Appendix B.2.3]{aganagic-knot-2}, as one can see by 
comparing Def. \ref{d:u-variable} with that reference. The potential, and the space it is defined on, are distinct relatives of that in \cite{gaiottowitten2011, braverman2014}.

\subsubsection{Holomorphic curves and cylindrical model}

We will study holomorphic curves $C \to \MCB(\Gamma, \vec d)$ by first considering their composition with the projection to $\mathbf T/W = \prod_i Sym^{d_i} \C^*_y$.  

We borrow the `cylindrical model' from the Heegaard Floer literature.  This is the following assertion: 

\begin{lemma}
\cite{lipshitz2006cylindrical}
    For any target curve $X$, there is a bijection (a) between maps $C \to Sym^d X$ transverse to the double point locus and (b) maps $S \to C \times X$, where $S \to C$ is a $d:1$ cover with simple ramification.
\end{lemma}

We correspondingly identify: (a) maps $C \to \prod_i Sym^{d_i} \C^*_y$ which are transverse to the 'big diagonal divisor' with (b) tuples of maps $S_i \to C \times \C^*_y$ such that $S_i \to C$ has simple ramification.  We write $S = \coprod S_i$, but always understand the components of $S$ to be labelled.  

Now will classify lifts of a given $C \to \prod_i Sym^{d_i} \C^*_y$ (always assumed transverse to the double point locus) to $C \to \MCB(\Gamma, \vec d)$, in terms of the corresponding maps  $S \to C \times \C^*_y$.  

Let us write $C_O = C \times_{\mathbf T / W} \mathbf T_O / W$, and correspondingly $S_{O} = S \times_C C_O$. 
Since in  $\mathbf T_O$ we have in particular removed the double point locus, the map $S_O \to C_O$ is \'etale. 
Since the $W$ action on 
$\mathbf T_{O}$ is free, the following is immediate: 

\begin{lemma}
Fix a map $C_O \to \mathbf{T}_O/W$ and consider
the corresponding 
$S_{O} \to C_O \times \C^*_y$. 
Then there is a canonical bijection between: (a) 
    lifts to 
$C_O \to \MCB(\Gamma, \vec d)_O= (\mathbf T^\vee \times \mathbf T_{O})/ W$ and
(b) tuples of maps $S_{O} \to \C^*_u$, where the subscript $u$ in $\mathbf T_u^\vee$
\end{lemma}

Consider now the case of some  $C \to \mathbf{T}/W = \prod_i Sym^{d_i} \C^*_y$ transverse to $(\mathbf{T} \setminus \mathbf{T}_O) / W$. 
We write $C_{(i)}$, $C_{[i] \to (i)}$, and $C_{(i) \to (j)}$ for preimages in $C$ of the corresponding hyperplanes.  (By assumption of transversality, these sets of points have multiplicity one and are all distinct.)  

Points in the first two classes each correspond to a single witness in $S$, we write the sets of such points as  $S_{(i)}$ and  $S_{[i] \to (i)}$.  Points in 
$C_{(i) \to (j)}$ correspond to certain pairs of points $(s_i, s_j)$ with the same image in $C$; we denote the set of such pairs $S_{(i) \to (j)} \subset S \times_C S$.

We establish the following result. 

\begin{theorem} \label{cylindrical model intro} (\ref{cylindrical model})
    Fix a map $C \to \mathbf{T}/W = \prod_i Sym^{d_i} \C^*_y$ transverse to $(\mathbf{T} \setminus \mathbf{T}_O) / W$, and consider
    the corresponding 
    $S \to C \times \C^*_y$.  
    Lifts to $C \to \MCB(\Gamma, \vec d)$ are in bijection with maps $\Phi_u: S \to \P^1_u$ satisfying the following conditions: 
\begin{enumerate}
\setcounter{enumi}{-1}
    \item $\Phi_u(S_O) \subset \C^*_u$. 
    \item \label{cylindrical model root condition} $\Phi_u$ has simple poles over points in $S_{(i)}$.
    \item \label{cylindrical model external matter condition} $\Phi_u$ has simple zeros over $S_{[i] \to (i)}$.
    \item \label{cylindrical model internal matter condition} For $(s_i, s_j) \in S \times_C S$, the function $\Phi_u(s_i) \Phi_u(s_j)$ has a simple zero at $(s_i,s_j) \in S_{(i) \to(j)}$. 
\end{enumerate}
\end{theorem}

The basic point in the proof is that, according to \cite[Section 6]{BFN} the geometry of the Coulomb branch over the various hyperplanes are given by a trivial factor times a Coulomb branch   with total dimension $2$.  The `$u$' variables transform conveniently along this reduction, and so the theorem reduces to the case of such quivers (there are three), where it is elementary. 

\begin{remark}
    Our approach is informed and inspired 
    by some recent uses of cylindrical model approaches to compute (or define) Fukaya categories of symmetric powers and related spaces of higher dimensional symplectic varieties \cite{Colin-Honda-Tian, Honda-Tian-Yuan, Mak-Smith}.
\end{remark}

\subsection{Fukaya category calculations}
\label{s:intro-Fuk}

The trivialization \ref{trivialization-intro} lets us explicitly describe certain special Lagrangians ${\cal T}_{\theta}$ in $\MCB(\Gamma, \vec d)$. 
We write $\mathbf{T} =  \mathbf{R} \times  \mathbf{\Theta}$,  for the coordinate-wise splitting of the corresponding complex numbers into modulus and argument, $y = e^{\rho} e^{i \theta}$, and similarly for the dual torus.  

We define
$$ \bfTh_O := \{ \theta \in \bfTh \mid \bfR \times \theta \In \bfT_O \}.$$
Then $\bfW$ acts on $\bfTh_O$ freely. 

For any $\theta \in \bfTh_O$, we define
$$\widetilde T_\theta := \mathbf{R} \times \theta  \subset \mathbf{T}_{O} \qquad \qquad  \widetilde{\mathcal{T}}_\theta :=  (\mathbf{R}^\vee \times 0)  \times  (\mathbf{R} \times \theta)  \subset \mathbf{T}^\vee \times \mathbf{T}_{O}$$
Let $T_\theta \subset \bfT_O/W$ and $\mathcal{T}_\theta \subset (\bfT^\vee \times \bfT_O)/W$ to be the
images of these Lagrangians under the free $\bfW$ quotient. 
We use the same name of the Lagrangians in the (unpunctured) upstairs and downstairs spaces $T_\theta \in \bfT/\bfW$ and $\T_\theta \in \MCB(\Gamma, \vec d)$; these
remain closed submanifolds. 
Note that the $\tcal_\theta$  are Lagrangians for the natural product K\"ahler form on $\mathbf{T}^\vee 
\times \mathbf{T}_{O}$, not for the holomorphic symplectic form.

We can record all the data relevant to ${\cal T}_{\theta}$ in terms of configurations of red and black points on a circle.
Write $[[a]]$ for a circle decorated by red points, with one red point labelled $i$ at each $\arg(a_{i, \alpha})$.
Write $((\theta))$
for a circle which is decorated by black points in addition to red, with 
one black point labeled $i$ at each $\theta_{i, \alpha}$.  



There is an evident correspondence of objects: 

$$\mathcal{T}_\theta \in Fuk(\MCB(\Gamma, \vec d), \mathcal{W}) \longleftrightarrow ((\theta)) \in \mathcal{C}. $$
We will prove: 
\begin{theorem} \label{main theorem}
    Assume $\MCB(\Gamma, \vec d) $ is smooth.  (Smoothness is known when $\Gamma$ is of ADE type \cite{BFN-slice}.) 
    There is an embedding
    \begin{eqnarray*}
    KLRW & \hookrightarrow &
    \Fuk(\MCB(\Gamma, \vec d) , \mathcal{W}) \\
    ((\theta)) & \mapsto & \mathcal{T}_\theta
    \end{eqnarray*}
\end{theorem}
The proof proceeds by using the cylindrical model of Theorem \ref{cylindrical model intro} to reduce all calculations to disk counts in a cylinder, which we then carry out explicitly.  The proof occupies the entirety of the final four sections of this article.

\begin{conjecture} \label{generation}
    The collection $\mathcal{T}_\theta$ generate $\Fuk(\MCB(\Gamma, \vec d) , \mathcal{W})$.
\end{conjecture}

\begin{remark}
The reader might wonder whether the problem of generation might be approached by using the well known fact that, for a Lefschetz fibration in the sense of Seidel, the thimbles are known to generate the Fukaya-Seidel category (see e.g. \cite{Seidel-book, Biran-Cornea-cone, GPS2}).  We caution that this result does not apply here because $\mathcal{W}$ is not a `fibration at infinity' in the fibers.  In fact, our $\mathcal{W}$ has no critical points whatsoever, similar to the situation $\mathcal{W}: \C^* \to \C$ given by the evident embedding.  (The analogue of $\mathcal{T}_\theta$ in that example is $\R e^{i \theta}$, which does generate.)

There is a (multi-valued)  `equivariant' potential $\mathcal{W}_{eq}$ for $\MCB(\Gamma, \vec d)$ (see \eqref{equivariant superpotential} below), which does have critical points, and
whose thimbles include some, but not generally all, of the 
$\mathcal{T}_\theta$. However, at present, there are no available results guaranteeing generation by thimbles of such multi-valued superpotentials. 
\end{remark}

\subsection{Gradings} \label{intro gradings}

There is a long history of attempts to find a Floer theoretic realization of categorified quantum group link invariants, pioneered by Seidel and Smith \cite{SS} and pursued in \cite{abouzaid-smith:khovanov, manolescu, manolescu2, Mak-Smith} and elsewhere.  One virtue of such descriptions was that the categorical braid group action had a simple geometric description.  However, the corresponding categories were invariably singly graded: they had the homological grading, but lacked the grading associated to the `$q$' variable of the Jones polynomial. 

In general, if $Y$ is affine and $H$ is a divisor, then the function cutting out $H$ determines a map $Y \setminus H \to \C^*$.  Such a map in particular determines a $H^1$ class (by pulling back from $\C^*$).  We may use this $H^1$-class to define a $\Z$-gradings on the hom space $End(L)$, where $L$ is any contractible Lagrangian in $Y \RM H$. 

In our setup, we can obtain the desired gradings in this manner.  The relevant divisor arises as follows.    Our space $\MCB(\mathbf{G}, \mathbf{N})$ arises from the multiplicative Coulomb branch $\MCB(\mathbf{G}, \mathbf{N} \oplus \mathbf{N}_F)$ by removing the divisor $\{r_{\vec 1, \cdots, \vec 1}=0\}$.  The resulting $\Z$-grading matches the KLRW $\Z$-grading, where crossing of black strands of the same label contribute $-1$ and dot contribute $1$. See Section \ref{s:grading}.

Let us note that the target in \cite{SS} was shown in  \cite{manolescu} to be a certain quiver variety, now known to be isomorphic to $\ACB(\mathbf{G}, \mathbf{N} \oplus \mathbf{N}_F)$.  Our space differs both by being the multiplicative, instead of additive, version (this is related to the fact that we find the {\em cylindrical} KLRW) and insofar as we have deleted a divisor.

We refer to \cite{ aganagic-knot-1, aganagic-knot-2, aganagic-icm} and especially \cite{ALR} for the realization in terms of $Fuk(\MCB(\Gamma, \vec d) , \mathcal{W})$ of the relation between the KLRW category and the homological knot invariants.  In \cite{ALR}  it is explained how to explicitly compute the braid group action on in terms of ${\cal T}_{\theta}$ branes as well as link homology groups, as spaces of morphisms between objects of the cylindrical KLRW category.  We note in particular that the cylindrical KLRW category gives rise to the (richer) invariants of knots in $S^1  \times \R^2$, of which knots in $\R^3$ are a special case.

\subsection{Homological mirror symmetry}

It was conjectured in \cite{aganagic-knot-2} that $Fuk(\MCB(\Gamma, \vec d), \mathcal{W})$ is equvialent to the category of coherent sheaves on the (resolved) additive Coulomb branch $\mcal^+_C(\mathbf{G}, \mathbf{N} \oplus \mathbf{N}_F)_\chi$ of a 3d $N=4$ theory based on the same quiver, where $\chi$ is a resolution parameter.

For a general class of conical holomorphic symplectic varieties, including in particular the additive Coulomb branches, Bezrukavnikov and Kaledin \cite{Bezrukavnikov-Kaledin} constructed a tilting vector bundle ${\cal T}$ \cite{kaledin2008derived}.  For the case of interest here, Webster computed the endomorphism algebra of this bundle in \cite{webster2019coherent, webster2022coherent}, and showed that it coincides with the cylindrical KLRW algebra.

Combining this result of Webster with Thm. \ref{main theorem}, we have: 

\begin{equation}\label{HMS} Coh(\mcal^+_C(\mathbf{G}, \mathbf{N} \oplus \mathbf{N}_F)_\chi) = \text{KLRW-Perf} \hookrightarrow  Fuk(\MCB(\Gamma, \vec d) , \mathcal{W})
\end{equation}

Establishing surjectivity would amount to verifying Conjecture \ref{generation}. 

Let us turn to the question of how the gradings  discussed in Section \ref{intro gradings} interact with mirror symmetry.  An `equivariant' extension of $\mathcal{W}$ was proposed in 
 \cite[Appendix B.2.3]{aganagic-knot-2}, it took the form: 
\begin{equation} \label{equivariant superpotential} \wcal_{eq} = \wcal + \sum_{i \in \Gamma} \Lambda_i \log f_i + \Lambda_0 \log f_0. 
\end{equation}

Let us recall the relation of equivariance and log terms in the superpotential.  It has long been known that considering torus-equivariant Gromov-Witten invariants on one side of mirror symmetry corresponds to adding log terms to the mirror superpotential on the other \cite{Givental-mirror}.
However, here we are concerned instead with the other direction of mirror symmetry: how a torus action in the B-model is reflected the A-model.

The prescription we follow is the following.  First, delete the locus where the argument of the logarithms in the equivariant superpotential $\mathcal{W}_{eq}$ vanish. Then, form the pull-back square
$$ \begin{tikzcd}
  \wt \ycal \ar[r]\ar[d] & \ycal \ar[d]  \\
  \C^r \ar[r, "exp"] & (\C^*)^r
\end{tikzcd}
$$
where $r$ is the number of log terms. Now $\wcal_{eq}$ is a well-defined function on $\wt \ycal$, and the action by deck transformation induces a $\Z^r$ action on the Fukaya-Seidel category of $(\wt \ycal, \wcal)$.
The general expectation is that this $\Z^r$ action should be mirror to the action by tensor product of the characters of the mirror $(\C^*)^r$ on the  equivariant coherent sheaves category of the mirror \cite[Rem. 4.6]{teleman2014gauge}.  (The case of toric varieties is discussed in \cite{shende2021toric}.) 

In the present case, tracking the $\Z^r$ action across the mirror symmetry \eqref{HMS} amounts to checking that the grading on KLRW  is recovered on the $Coh$ side via a torus action on  $\mcal^+_C(\mathbf{G}, \mathbf{N} \oplus \mathbf{N}_F)_\chi$.  This was established in  \cite{aganagic-knot-1, webster2022coherent}.

\begin{remark}
    In the abelian case (all $d_i=1$), 
a related mirror symmetry result (after removing the superpotential on the A-model and replacing the B-model with the multiplicative version) was established by Gammage, McBreen, and Webster \cite{mcbreen2018homological, gammage2019homological} by entirely different methods;  namely, calculating the skeleton and microsheaves in the sense of \cite{nadler-shende} and then appealing to \cite{GPS3}.
\end{remark}

\begin{remark}
In the abelian setting, there is a Hamiltonian torus action on the target on the $A$ side.  Using the results of the aforementioned \cite{gammage2019homological}, it was shown in \cite{McBreen-Shende-Zhou} that `reduction commutes with Fukaya category', which in turn amounted to a mirror symmetry between the A-model on $\mathbf{T}_O$, and the B-model on the core of a certain conical Lagrangian inside the original mirror dual multiplicative Coulomb branch.  We expect a similar phenomenon in general, although on the A-side, rather than simply $\mathbf{T}_O$, we expect a more general ``Fukaya category with coefficients in a schober on $\mathbf{T}$''.
\end{remark}

\begin{remark}
Another mirror symmetry result in the context of Coulomb branches is Jin's study  \cite{jin2022homological}, of a Fukaya category of $\ACB(G, 0)$, where the symplectic form is taken to be the real part of the holomorphic symplectic form $\Omega$; she showed an equivalence with  $\Coh(T^\vee / W)$.  At present, we do not entirely understand the relationship with the work in this article.  
\end{remark}

\begin{remark}
Cautis and Kamnitzer \cite{cautis-kamnitzer-1, cautis-kamnitzer-2} previously studied coherent sheaves on  moduli spaces of Hecke modifications, to obtain invariants of links in ${\mathbb R}^3$. The spaces of Hecke modifications are (compact) convolution products ${\rm Gr}^{\vec \mu}$ of orbits in the affine Grassmannian ${\rm Gr} = {\rm Gr}_{G_{\Gamma}}$
 of the quiver group $G_{\Gamma}$. In this setting, the desired $q$ grading also arises from a torus action. 
 Our space is a transversal slices inside ${\rm Gr}^{\vec \mu}$ to another orbit ${\rm Gr}_{\nu}$, whenever a ${\vec \mu}$ and $\nu$ satisfy a certain dominance condition \cite{BFN}.
 \end{remark}

\subsection{Acknowledgements}
We would like to thank Mohammed Abouzaid, Ivan Smith, Ben Webster, Xin Jin for many insightful discussion.  We would like to especially thank Cheuk-Yu Mak, Yin Tian and Tianyu Yuan for explaining the work in \cite{Mak-Smith} and \cite{Honda-Tian-Yuan}. The main disk counting are done with the generous help of Tianyu Yuan. 

M.A. is supported, in part, by the NSF foundation grant PHY2112880 and by the Simons Investigator Award.
V.S. is supported by Novo Nordisk Foundation grant NNF20OC0066298, Villum Fonden Villum Investigator grant 37814, and Danish National Research Foundation grant DNRF157.

\section{Review of 3d N=4 Coulomb branches} \label{s:general-CB}
Let $G$ be a complex reductive group and $N$ a finite dimensional representation.  To this data, physicists associate a 3d N=4 gauge theory, which has certain associated moduli spaces called the ``Higgs branch'' and the ``Coulomb branch''. 

A direct mathematical definition of  Coulomb branches was given in \cite{BFN}, who describe them as algebraic varieties.  There are two versions: additive ($\ACB(G,N)$) and multiplicative ($\MCB(G,N)$).  Each carries a map: 
$$ \mu^+: \ACB(G,N) \to \tfrak/W, \quad  \mu^\times: \MCB(G,N) \to T/W. $$

The main results of the present article concern Floer theory in multiplicative Coulomb branches.  The purpose of this section is to recall from the literature various general results regarding the geometric properties of these spaces.  We focus on the multiplicative version.

\subsection{Definition}

Fix a connected complex reductive group $G$ and a finite dimensional representation $N$ of $G$. We assume $\pi_1(G)$ has no torsion. Let $T\In B \In G$,  $B$ a Borel subgroup and $T$ a maximal torus. Let $W$ be the Weyl group. Let $X_*(T) = \Hom(\C^*, T)$ be the coweight lattice. Let $\K = \C((t)), \O=\C[[t]]$. 
Let $Gr_G = G(\K) / G(\O)$ be the affine Grassmannian. For any $\lambda \in X_*(T)$, we have the $G(\O)$ orbit $Gr_G^\lambda = G(\O) t^\lambda G(\O)/G(\O) \In Gr_G$.

Let $\bD = \Spec \O$ denote the formal disk, $\bD^o = \Spec \K$ the punctured formal disk, and $\B = \bD \cup_{\bD^o} \bD$ the `bubble' or the `raviolli'.

We have the following moduli stacks of $G$-bundles
$$ \Bun_G(\bD) = \pt /  \GO, \quad \Bun_G(\bD^o) = \pt /  \GK, \quad \Bun_G(\B) = \Bun_G(\bD) \times_{\Bun_G(\bD^o)} \Bun_G(\bD). $$
where the map $\Bun_G(\bD) \to\Bun_G(\bD^o)$ is the restriction of $G$-bundles. 
Let $\Bun_{G,N}$ denote the moduli stack of $G$-bundles together with a section in the associated $N$-bundle, then
$$ \Bun_{G,N}(\bD) = N(\O) /  \GO, \quad \Bun_{G,N}(\bD^o) = N(\K) /  \GK, \quad \Bun_{G,N}(\B) = \Bun_{G,N}(\bD) \times_{\Bun_{G,N}(\bD^o)} \Bun_{G,N}(\bD). $$

Define $\B_{123}:= \bD_1 \cup_{\bD^o} \bD_2 \cup_{\bD^o} \bD_3$ and $\B_{ij} = \bD_i \cup_{\bD^o} \bD_j$ for $1=i<j=3$. We have restriction maps $p_{ij}: \Bun_{G,N}(\B_{123}) \to \Bun_{G,N}(\B_{ij})$.  In terms of these, one would expect a convolution product on 
$$\acal^\times (G,N) \,\, \mbox{`:='} \,\,  K(\Bun_{G,N}(\B)) = K \left( \frac{N(\O)}{G(\O)} \bigtimes_{\frac{N(\K)}{G(\K)}} \frac{N(\O)}{G(\O)} \right)$$

We put the quotes because the infinite nature of the spaces involved (and the state of present technology) lead to some problems in defining the RHS directly.  Instead, \cite{BFN}  rewrite
$$ \frac{N(\O)}{G(\O)} \bigtimes_{\frac{N(\K)}{G(\K)}} \frac{N(\O)}{G(\O)} 
\cong 
\frac{\frac{G(\K)\times N(\O)}{G(\O)} \underset{N(\K)}{\times}  
\frac{G(\K)\times N(\O)}{G(\O)}}{G(\K)} 
\cong 
\frac{N(\O) \underset{N(\K)}{\times}  \frac{G(\K)\times N(\O)}{G(\O)}}
{G(\O)} $$
and introduce 
\begin{equation}
    \label{e:RGN-TGN}
    T_{G,N} :=  \frac{G(\K)\times N(\O)}{G(\O)}, \quad R_{G,N} = 
N(\O) \underset{N(\K)}{\times}  T_{G,N}
\end{equation} 

A main result of \cite{BFN} is that 
$K^{G(\O)}(R_{G,N})$ indeed admits a natural
structure of convolution algebra, which
is moreover commutative.  They then define:

\begin{definition}[\cite{BFN}] \label{d:coulomb}
The affine multiplicative Coulomb branch algebra and spaces are defined as
    $$ \acal^\times(G,N) := K^{G(\O)}(R_{G,N}), \quad \MCB(G, N) := \Spec  K^{G(\O)}(R_{G,N})$$
\end{definition}

\begin{remark}
For the discussion of equivariant K-theory and BM homology, we can replace $G(\O)$-equivariance with $G$-equivariance. See \cite{BFM} Lemma 6.2 and the comment above. 
\end{remark}

Using $R_{G,N} \to pt$, we get a pullback map of algebra $K^{\GO}(pt) \to K^{\GO}(R_{G,N})$. Since $\Spec K^{\GO}(pt) = T/W$, we get a map of spaces
$$ \mu: \MCB(G,N) \to T/W. $$


Let $\pi_{R}: R_{G, N} \to Gr_{G}$ be the projection map. If $\lambda$ is a minuscule coweight of $G$, then the $G(\ocal)$ orbit $Gr_G^\lambda$ is closed, and the preimage $R^\lambda_{G,N}:=\pi_R^{-1}(Gr_G^\lambda)$ is closed. 

\begin{definition} \label{monopole operators}
    For any  minuscule coweight $\lambda$ of $G$,
    then we write $r_\lambda \in  \acal^\times(G,N)$ for the equivariant K-theory class of the structure sheaf on  $R^\lambda_{G,N}$.   
    More generally, if $\cal V$ is a $G(\O)$-equivariant vector bundle over $Gr^\lambda$, then we write $r_{\lambda, \cal V} \in \acal^\times(G,N)$ for the K-theory class of the pullback bundle $\pi_R^* \cal V$ on $R^\lambda_{G,N}$. 
\end{definition}

\subsection{Abelian case and ``pure gauge theory'' case}

The simplest example is when $G = T = (\C^*)^k$ and $N = 0$. In this case, we get
$$ \MCB(T, 0) = T^\vee \times T $$
and the map $\mu$ is simply the projection $\mu: T^\vee \times T \to T$. 

\begin{proposition} \label{prop abelian coulomb branch coordinates} \cite[Prop 4.1]{BFN}
Let $T$ be a torus, and $N$ be its representation with weights $\xi_1, \cdots, \xi_n$. For any two integers $k,l$, let 
$$ d(k,l) = \begin{cases}
  0 & \text{if $k$ and $l$ have the same sign} \\
  \min(|k|, |l|) & \text{otherwise}
\end{cases}
$$
and for any $\xi \in X^*(T)$, and $\lambda, \mu \in X_*(T)$, let
$$ d_\xi(\lambda, \mu) := d(\la \xi, \lambda \ra, \la \xi, \lambda \ra). $$ 
The multiplicative Coulomb branch algebra $\acal^\times (T,N)$ is generated as an algebra over $\O(T)$ with generators $\{r_\lambda\}_{\lambda \in X_*(T)}$, satisfying relations
$$ r_\lambda r_\mu = r_{\lambda+\mu} \cdot \prod_{i=1}^n (1 - y^{-\xi_i})^{d_{\xi_i}(\lambda, \mu)} $$
where for $\xi \in X^*(T)$ we denote $y^{\xi}: T \to \C^*$ as the character function on $T$.  
\end{proposition}

In general, when $N = 0$ (`pure gauge theory'), 
there is an explicit description of the Coulomb branch: 
\begin{proposition} \label{pure gauge theory coordinates} 
\cite[Theorem 2.15]{BFM}
Let $G$ be a complex reductive group, $R \In X^*(T)$ the roots of $G$, $W$ the Weyl group, then: 
$$ K^{G(\O)}_*(Gr_G) \cong \C[T^\vee \times T, \{\frac{e^{\alpha^\vee} -1 }{e^\alpha - 1}\}_{\alpha \in R} ]^W.$$
\end{proposition}

\begin{example}
If $G = GL(d)$, we have 
$$  K^{G(\O)}(Gr_G) \cong \C[x_1^\pm, \cdots, x_d^\pm, y_1^\pm, \cdots, y_d^\pm, \{\frac{x_i/x_j -1 }{y_i/y_j -1 }\}_{i \neq j} ]^{S_d}$$
where $x_i^\pm$ means $x_i$ is a $\C^*$ variable. 
\end{example}

\subsection{Localization, abelianization, and formulas for monopole operators} \label{localization section}

We recall from \cite{BFN-slice} an embedding 
of the coordinate ring of $\MCB(G,N)$ into
a localization of the 
coordinate ring of $\MCB(T,0)$. 

Consider the closed embeddings (recall the definition of $R_{G,N}, T_{G,N}$ from Eq. \eqref{e:RGN-TGN}), 
$$ \iota: R_{T, N} \into R_{G,N} $$ 
$$ h: R_{T, N} \into T_{T, N}, $$
and the inclusion of the zero-section $$z: Gr_T \into T_{T,N}.$$
Consider the pushforward in K-theory
$$ h_*: K^T(R_{T,N}) \to K^T(T_{T,N}) , \quad \iota_*: K^T(R_{T,N}) \to K^T(R_{G,N}).$$
An explicit calculation in the abelian gauge theory case (see section 4(vi) of \cite{BFN}) shows $h_*$ is injective. 
Since $T_{T,N}$ is a vector bundle over $Gr_T$, we have an isomorphism 
$$ z^*: K^T(T_{T,N}) \congto K^T(Gr_T) $$

Note $K^T(R_{T,N})$ and $K^T(R_{G,N})$ are modules over $K^T(pt) = O(T)$, and for any root $\alpha$ of $G$ we have root hyperplane $H_\alpha = \{e^\alpha = 1\}$ in $T$. Let $O(T)_{loc} = O(T \RM \cup_\alpha H_\alpha)$ denote the coordinate ring of $T$ removing all the root hyperplanes, and let $K^T(R_{T,N})_{loc}$ and $K^T(R_{G,N})_{loc}$ denote the corresponding localization as modules.  
After such localization 
$\iota_*$ is an isomorphism: 
$$ (\iota_*)_{loc}: K^T(R_{T,N})_{loc} \congto K^T(R_{G,N})_{loc}.  $$
Combined, this determines the desired morphism
\begin{equation} 
z^* h_* (\iota_*)_{loc}^{-1} :  K^T(R_{G,N})_{loc} \into  K^T(Gr_T)_{loc} \end{equation}
which restricts to the subring $\O(\MCB(G,N)) \cong K^T(R_{G,N})^W \into  K^T(R_{G,N})_{loc}$ to get the `abelianization' map
$$ ab:=z^* h_* (\iota_*)_{loc}^{-1} \quad :  \O(\MCB(G,N))  \into  \O(\MCB(T,0))_{loc}.$$ 

We introduce the following notation for writing in abelianized coordinates:

\begin{definition} \label{ab coordinates xs and ys}
For any $\lambda \in X_*(T)$, let $x^\lambda$ denote the corresponding element in $\C[X_*(T)] = \O(T^\vee)$. For any $\xi \in X^*(T)$, let $y^\xi$ denote the element in $\C[X^*(T)] = \O(T)$.  \end{definition}
Then we have 
$$ \ocal(\MCB(T,0)) = \ocal(T \times T^\vee), \quad 
\ocal(\MCB(T,0))_{loc} = \ocal(T \times T^\vee)[\{(y^\alpha - 1)^{-1} \mid \alpha \in R\}]$$
where $R \In X^*(T)$ is the set of roots in $G$.

We turn to recall some formulas for the monopole operators (Def. \ref{monopole operators}). 

\begin{proposition}\cite[Proposition A.2]{BFN-slice} \label{p:abelianize}
If $\lambda$ is a minuscule coweight of $G$, and $r_\lambda \in \O(\MCB(G,N))$ is the K-theory class of the structure sheaf on $R^\lambda_{G,N}$, then
\begin{equation} \label{abelianized coordinates} ab (r_\lambda) = \sum_{\lambda' \in W \lambda} x^{\lambda'} \cdot \frac{\chi_T(t^{\lambda'} N(\O)/ t^{\lambda'} N(\O)\cap N(\O)) }{\chi_T(T_{\lambda'} Gr^\lambda )}  
\end{equation}
where $W \lambda$ is the $W$-orbit of $\lambda$, $T_{\lambda'} Gr^\lambda$ is the tangent space at the $T$-fixed point $t^{\lambda'}$ in $Gr^\lambda$, and $\chi_T$ is the map
$$ \chi_T: Rep(T) \to K(Rep(T)), \quad V \mapsto [\wedge^*(V^\vee[1])], \text{ i.e.} \quad \bigoplus_{i=1}^m \C_{\xi_i} \mapsto \prod_{i=1}^m (1-y^{-\xi_i}). $$ 
More generally, if $\cal V$ is a $G(\O)$-equivariant vector bundle over $Gr^\lambda$, and $r_{\lambda, \cal V}$ is the K-theory class of the pullback of $\cal V$ to $R^\lambda_{G,N}$, then 
\begin{equation} \label{more abelianized coordinates} ab (r_{\lambda, \cal V}) = \sum_{\lambda' \in W \lambda} x^{\lambda'} \cdot  [\cal V|_{t^{\lambda'}}]\cdot \frac{ \chi_T(t^{\lambda'} N(\O)/ t^{\lambda'} N(\O)\cap N(\O)) }{\chi_T(T_{\lambda'} Gr^\lambda )}. 
\end{equation}
\end{proposition}

Consider the pure gauge theory of $\MCB(GL_d, 0)$. \cite{Finkelberg-Tsymbaliuk} gives the abelianized coordinate for some monopole operators as following.

For $r=1, \cdots, d-1$, let $\varpi_r$ be the $r$-th fundamental coweight, i.e. $\varpi_r=(1,\cdots, 1, 0, \cdots, 0)$ where $1$ appears $r$ times. 
Let $w_0 \in S_d$ be the total order reversing element, i.e., the longest word in the Weyl group $W$. 
Then $-w_0 \varpi_r = (0, \cdots, 0, -1, \cdots, -1)$ where $-1$ appears $r$ times. Let $L_0 = (\C[[t]])^d$ be the standard lattice, then we have
$$ Gr^{\varpi_r}=\{t L_0 \In L \In L_0 \mid \dim (L_0/L) = r. \} $$
and 
$$ Gr^{-w_0 \varpi_r}=\{ L_0 \In L \In t^{-1} L_0 \mid \dim (L/L_0) = r. \} $$ 
Then $Gr^{\varpi_r}$ parametrizes a codimension $r$ subspace in an $d$-dimensional space $L_0/ tL_0$, and $Gr^{-w_0 \varpi_r}$ parametrizes a dimension $r$ subspace in an $d$-dimensional space $t^{-1} L_0/ L_0$.

\begin{proposition} \label{p:gln-monopole} \cite{Finkelberg-Tsymbaliuk}
Let $Q_r$ denote the tautological rank $r$ quotient vector bundle $L_0/L$ on $Gr^{\varpi_r}$, and $S_r$ denote the tautological rank $r$ vector sub-bundle $L/L_0$ on $Gr^{-w_0\varpi_r}$. 
For $p=0, \cdots, r$, let $\wedge^p(Q_r)$ be the $p$-th exterior power of $Q_r$ (rank $r$ vector bundle), and let $e_p(x_1, \cdots, x_r)$ be the $p$-th elementary symmmetric function in $r$ variables.  Then, we have abelianization of K-theory classes: 
    \begin{equation} \label{FT eq1}
    ab ([\wedge^p(Q_r) \otimes \O_{Gr^{\varpi_r}}]) = \sum_{J \In \{1, \cdots, d\}, \#J = r } \frac{e_p(\{y_j\}_{j \in J}) }{\prod_{\alpha \in J, \beta \notin J} (1 - y_\beta/y_\alpha) } \cdot \prod_{\alpha \in J} x_\alpha 
\end{equation} 
\begin{equation} \label{FT eq2}
   ab([\wedge^p(S_r) \otimes \O_{Gr^{-w_0 \varpi_r}}]) = \sum_{J \In \{1, \cdots, d\}, \#J = r } \frac{e_p(\{y_j\}_{j \in J}) }{\prod_{\alpha \in J, \beta \notin J} (1 - y_\alpha/y_\beta) } \cdot \prod_{\alpha \in J} x^{-1}_\alpha.
\end{equation} 
\end{proposition}

We also recall a result about changing the representation $N$ to a subrepresentation.

\begin{proposition}\label{p:abelian-remove-matter}
\cite[Section 4(vi)]{BFN} If $G=T$ is a torus and $N' \In N$ is a sub-representation where $N/N' = \oplus_i \C_{\xi_i}$,  then we have canonical morphism (not necessarily an embedding)
$$ q: \MCB(T,N') \to \MCB(T,N). $$
If we use $r_\lambda$ for the monopole operator on $\MCB(T,N)$ and $r_\lambda'$ for $\MCB(T,N')$, then 
\begin{equation}\label{eq:remove-matter}  q^* (r_\lambda) = r'_\lambda  \prod_{i=1}^m (1 - y^{-\xi_i}))^{\max(0, - \langle \xi_i, \lambda \rangle)}.
\end{equation} 
\end{proposition}

    


\begin{proposition}\cite[Theorem 2.9]{KWWY22} \label{p:invert}
Let $\acal(G,N)$ denote the (additive or multiplicative) Coulomb branch algebra. Let $\xi: \C^* \to T$ be a central cocharacter $G$. Assume that the representation $N$ restricted to this $\C^*$ has only non-negative weights, and let $N^\xi$ denote the fixed subspace of the $\C^*$-action. 
Then the natural inclusion $\acal(G,N) \into \acal(G, N^\xi)$ induces an isomorphism
$$ \acal(G,N)[r_\xi^{-1}] \congto \acal(G,N^\xi). $$
\end{proposition}

\subsection{Root and Matter divisors}

To study the geometry of $\MCB(G,N)$, we can study the fibers of $\mu: \MCB(G,N) \to T/W$. First, we shall unfold the base $T/W$ back to $T$ and base change $\MCB(G,N)$. 
\begin{lemma}\cite[Lemma 5.3]{BFN}. 
Let    $ \wt \MCB(G, N) := \Spec K^T(R_{G,N}). $  Then there is
a pullback square: 
\begin{equation} \label{e:W-unfold} 
    \begin{tikzcd}
    \wt \MCB(G,N)  \ar[r] \ar[d, "\wt \mu"] & \MCB(G,N) \ar[d, "\mu"]   \\
    T \ar[r] &  T/W
\end{tikzcd}
\end{equation}
\end{lemma}

\begin{definition} \label{d:hyperplanes}
There are two collections of hyperplanes (i.e. codimension-1 subtori) in $T$, 
\begin{enumerate}
    \item If $\alpha \in X^*(T)$ is a root of $G$, we define $$ \wt H_\alpha = \{y^{\alpha}=1\} \In T$$ as {\bf the downstairs root hyperplane for root $\alpha$}. 
    \item If $w \in X^*(T)$ is a weight of $N$ under action of $T$, then we define $$ \wt H_w= \{y^{w} =1\} \In T $$ as {\bf the downstairs matter hyperplane for weight $w$}.      
\end{enumerate}

For $\bullet = \alpha$ or $w$, let $\wt \mu^{-1}(\wt  H_\bullet)$ be the unfolded upstairs (root or matter) hyperplane. Let the downstairs root or matter divisor $H_{\bullet} \In T/W$ be the $W$ quotient image of $\wt H_\bullet$, and let the upstairs root or matter divisors be $\mu^{-1}(H_\bullet)$. 
\end{definition}

Now we discuss the geometry near certain nice part of the divisor. We first introduce the following notation. 

\begin{definition}
Let $X$ be a smooth variety and let $D_1, \cdots, D_n$ be a collection of  divisor that pairwise intersects transversely in $X$. 
Let $D_i^{sing}$ be the singularity loci of $D_i$, and $D_i^{sm} = D_i \RM D_i^{sing}$. We define the 'punctured' divisor $D_i^o$ and its open neighborhood $U(D_i^o)$
$$ D_i^o = D_i \RM (D_i^{sing} \cup \cup_{j \neq i} D_j), \quad U(D_i^o) = X \RM (D_i^{sing} \cup \cup_{j \neq i} D_j). $$
We also define two open subsets in $X$, $X_O \In X_o \In X$, 
$$ X_o := \cup_{i=1}^n U(D_i^o) , \quad X_O:= \cap_{i=1}^n  U(D_i^o).$$
Note that $X_O= X\RM \cup_i D_i$. 
\end{definition}

\begin{proposition} \cite[Section 6]{BFN} \label{p:local-geom}
Consider the collection of all root hyperplanes $\{\wt H_\alpha\}$ and matter hyperplanes $\{\wt H_w\}$ in $T$, assuming that no root hyperplane coincides with any weight hyperplane. 
Let $t \in T_o$, then we have the following three cases: 
\begin{enumerate}
    \item \label{trivialization on O locus} If $t \in T_O$, i.e. $t$ is not on any matter or root hyperplanes, then the fiber $\wt \MCB(G,N)|_t= T^\vee$. More precisely, we have 
    $$ \wt \MCB(G,N)|_{T_O} \cong \wt \MCB(T,0)|_{T_O} = T^\vee \times T_O. $$
    \item If $t \in \wt H_\alpha^o$ for a root hyperplane,  then over the neighborhood $U(\wt H_\alpha^o)$ of $t$,  we have 
    $$ \wt \MCB(G, N)|_{U(\wt H_\alpha^o)} \cong \wt \MCB(G_\alpha, 0)|_{U(\wt H_\alpha^o)} $$ 
    where $T \In G_\alpha \In G$ is closed reductive subgroup that only contains root $\pm \alpha$.  
    Note that explicit coordinates on the RHS are given
    by Prop. \ref{pure gauge theory coordinates}. 
    \item If $t \in \wt H_w^o$ for a matter hyperplane, then over the neighborhood $U(\wt H_w^o)$ of $t$,  we have  
    $$\wt \MCB(G, N)|_{U(\wt H_w^o)} \cong \wt \MCB(T, N_{w})|_{U(\wt H_w^o)} $$
    where $N_w$ is the $w$-weight subspace of $N$ under the action of $T$. 
    Note that explicit coordinates on the RHS are given by 
    Prop. \ref{prop abelian coulomb branch coordinates}.
\end{enumerate}
\end{proposition}

\subsection{Deformation}

\begin{definition} \cite[Section 3(viii)]{BFN}
If we have a torus $T_F$ that acts on $N$, and the action commutes with the $G$ action. Then, we have
the $T_F$-family Coulomb branch $$ \MCB(G,N)_{T_F}: = \Spec K^{G(\O) \times T_F}(R_{G,N}). $$
with a map 
$$ \mu_F: \MCB(G,N)_{T_F} \to   T_F. $$
For $\beta \in T_F$, we define the deformed multiplicative Coulomb branch as
$$ \MCB(G,N)_{\beta} := \mu_F^{-1}(\beta). $$
\end{definition}


In the abelian gauge theory with matter, we have a deformed version of Prop \ref{prop abelian coulomb branch coordinates}. 

\begin{proposition} \label{prop deformed abelian coulomb branch coordinates} 
Let $T$ be a torus, $N = \oplus_{i=1}^n \C_{\xi_i}$ be its representation with weights $\xi_1, \cdots, \xi_n$. Let $T_F$ acts on $\C_{\xi_i}$ with weight $\eta_i \in X^*(T_F)$.  
The, the $T_F$-deformed multiplicative Coulomb branch algebra $\acal_C^\times (T,N)_{T_F}$ is generated as an algebra over $\O(T) \otimes \O(T_F)$ with generators $\{r_\lambda\}_{\lambda \in X_*(T)}$, satisfying relations
$$ r_\lambda r_\mu = r_{\lambda+\mu} \cdot \prod_{i=1}^n (1 - a^{-\eta_i} y^{-\xi_i})^{d(\xi_i(\lambda), \xi_i(\mu))} $$
where we use $y^\xi$ and $a^\eta$ as character functions on $T$ and $T_F$, respectively. 
\end{proposition}
\begin{proof}
Comparing with Prop \ref{prop abelian coulomb branch coordinates}, we only enlarged the equivariant group from $T$ to $T \times T_F$. A one dimensional representation of $T \times T_F$ of weight $(\xi, \eta) \in X^*(T) \times X^*(T_F)$ have Euler characteristic $1 - a^{-\eta} y^{-\xi}$. Hence we only need to replace  $1 - y^{-\xi_i}$ by $1-y^{-\xi_i} a^{-\eta_i}$ in the formulas in Prop \ref{prop abelian coulomb branch coordinates}. See see BFN section 4(iii) for more details.  
\end{proof}

\begin{example}
Let $G = \C^*$, $V=\C^2$ with $\C^*$ acting diagonally.  The undeformed Coulomb branch is given by 
$$  \MCB(G,N) = \{uv=(1-1/y)^2\} \In \C_u \times \C_v \times \C^*_y. $$
Let $T_F= (\C^*)^2$ acts on $\C^2$ canonically. Then, 
$$ \MCB(G,N)_{T_F} = \{uv = (1-a_1/y)(1-a_2/y)\}\In \C_u \times \C_v \times \C^*_y \times (\C^*)^2_{a_1,a_2}. $$ 
\end{example}

\section{Cylindrical model for disks in  Coulomb branches} \label{s:MCB-quiver}

\subsection{Quiver Gauge Theory}
Let $\Gamma$ be a quiver (i.e. directed graph), with (abusing notation) $\Gamma = \{(1), \cdots, (r)\}$ the set of nodes. Let $\bar \Gamma$  be the framed quiver associated with $\Gamma$, that has an additional `framing node' $[i]$ and additional arrow $[i] \to (i)$ for each node $(i) \in \Gamma$. Here we use `box' $[\,]$ and `circle' $(\,)$ to distinguish the added framing nodes and the original nodes in $\Gamma$. 

A representation of the framed quiver $\bar Q$ is given by a collection of vector spaces for each node and a linear map for each arrow. Concretely, for a `circle' node $(i)$ we attach vector space $V_{(i)} \cong \C^{d_i}$, for a framing node $[i]$ we attach $W_{[i]} \cong \C^{m_i}$. The tuple $\vec d=(d_i), \vec m=(m_i)$ are called dimension vectors.  The moduli stack of representations is $N / G$ where 
$$ N = N_B \oplus N_F, \quad N_B=\bigoplus_{e: (i) \to (j)} \Hom(\C^{d_i}, \C^{d_j}), \quad N_F = \bigoplus_{[i] \to (i)} \Hom(\C^{m_i}, \C^{d_i})$$ 
and 
$$ G = \prod_{(i)} GL(d_i), \quad T= \prod_{(i)} (\C^*)^{d_i}, \quad T_F = \prod_{[i]} (\C^*)^{m_i}.$$
An element $y \in T$ will have coordinates $(y_{i,\alpha})$ where $i=1,\cdots, r$, and $1 \leq \alpha \leq d_i$. An element $\beta \in T_F$ will have 
 coordinates $(a_{i,\alpha})$ where $i=1,\cdots, r$, and $1 \leq \alpha \leq m_i$.

\subsection{Examples}\label{ss:example}
Here we give some simple examples to illustrate the abelianization procedure. We draw gauge nodes as circles, framing nodes as boxes, and the number inside the quiver are the dimensions of the vector spaces. 

\begin{example}\label{ex:pure-abelian}
$$    \begin{tikzpicture}[square/.style={regular polygon, regular polygon sides=4}]
  \node[draw,circle] (c1) at (0,0) {1};
\end{tikzpicture}
 $$
    Let $G = GL_1$ and $N=0$. In this case,  $T_{G,N} = R_{G,N} = Gr_{T} \cong \Z$, and the $G(\O)$ action on $R_{G,N}$ is trivial. Hence we have 
    the multiplicative Coulomb branch algebra
    $$  \acal^\times (T,0) = K_*^{T(\O)}(Gr_{T}) = K_T^*(pt) \otimes \C[X_*(T)] = \ocal(T \times T^\vee).
    $$
    The same is true for $G=T = (\C^*)^n$. 
\end{example}

\begin{example}\label{ex:(1)-(1)}
$$
    \begin{tikzpicture}[square/.style={regular polygon, regular polygon sides=4}]
  \node[draw,circle] (c1) at (0,0) {1};
  \node[draw,circle] (c2) at (1.5,0) {1};
   \draw[->] (c2) -- (c1);
\end{tikzpicture}
$$
$G=GL_1 \times GL_1, N=C$ with weight $1$ under the first $GL_1$ and $-1$ under the second $GL_1$. 
We have monopole generators $r_{i,j}$, for $i, j \in \Z$, they satisfy relation as in Proposition \ref{prop abelian coulomb branch coordinates}. In particular, we have 
$$ r_{1,1} r_{-1,-1} = 1, \quad r_{1,0} r_{-1,0} = 1 - y_2 / y_1, \quad r_{0, 1} r_{0, -1} = 1 - y_2 / y_1. \quad r_{1,0} r_{0,1} = r_{1,1} (1 - y_2 / y_1).$$
One can see that $r_{1,0}, r_{0,1}, r_{1,1}, r_{-1,-1}$ are generators of the coordinate ring over the base ring $\C[y_1^\pm, y_2^\pm]$. If we write $u = r_{1,0}, v = r_{0,1}, z = r_{-1,-1}$, then we have 
$$ \MCB(G,N) = \{ uvz = 1-y_2/y_1 \} \In \C_u \times \C_v \times \C^*_z \times (\C^*)^2_{y_1,y_2} $$
The map $\mu: \MCB(G,N) \to T$ is taking the $(y_1, y_2)$ coordinate. 

It is still worthwhile to consider the abelianization process. We have the abelianization map 
$$ \O(\MCB(G, N)) \to \O(\MCB(T, 0))_{loc}, $$
where we use $x_i^\pm$ as coordinates on $T^\vee \cong (\C^*)^2$ and $y_i^\pm$ as coordinates on $T \cong (\C^*)^2$. The abelianization map does
$$ r_{1,0} \mapsto x_1, \quad  r_{-1,0} \mapsto x_1^{-1}( 1- y_2/y_1),$$ 
$$ r_{0,1} = x_2 (1 - y_2 / y_1) \quad r_{0, -1} \mapsto x_2^{-1}, $$
$$ r_{1,1} \mapsto x_1 x_2, \quad  r_{-1,-1} \mapsto x_1^{-1} x_2^{-1}. $$
\end{example}

\begin{example}\label{ex:[1]-(1)}
$$    \begin{tikzpicture}[square/.style={regular polygon, regular polygon sides=4}]
  \node[draw,square] (b1) at (0,1) {1};
  \node[draw,circle] (c1) at (0,0) {1};
  \draw[->] (b1) -- (c1);
\end{tikzpicture}
$$
Let $G = GL_1$ and $N=\C$. 
We have 
$$ \acal^\times(G, N) =\oplus_{\lambda \in \Z} \C[y, y^{-1}] r_\lambda, $$
such that 
$$ r_\lambda r_\mu = r_{\lambda+\mu} \cdot \begin{cases}
  1 & \text{if $\lambda, \mu$ have the same sign} \\
  (1 - 1/y)^{\min(|\lambda|, |\mu|)}& \text{otherwise}. 
\end{cases}. 
$$
We write $r_0$ as $1$, and $r_1=u,r_{-1}=v$, then we have
$$ \acal^\times(G, N) = \C[y, y^{-1}, u, v] / (uv = 1-1/y). $$
We have the map
$$ \mu: \MCB(GL_1, \C) \to \C^*, \quad (u,v,y) \mapsto y. $$
then the fiber over $y\neq 1$ is isomorphic to $\C^*_u$, and the fiber over $y=1$ is a cone $uv=0$. 
\end{example}

\begin{example}\label{ex:gl2}
$$
    \begin{tikzpicture}[square/.style={regular polygon, regular polygon sides=4}]
  \node[draw,circle] (c1) at (0,0) {2};
\end{tikzpicture}
$$
$G=GL_2, N=0$. See \cite[Section 3.8]{BFM}. The coordinate ring of $\MCB(G,N)$ is
$$ \C\left[x_1^\pm, x_2^\pm, y_1^\pm, y_2^\pm, \frac{x_1 - x_2 }{y_1 - y_2}\right]^{S_2}.  $$
where the symmetric group $S_2$ permute the indices of $x_i$ and $y_i$. 

Using Proposition \ref{p:abelianize}, we write 
miniscule monopole operators in abelianized coordinates. 
Let $\lambda = [n_1, n_2] \in \Z^2 / S_2 = X_*(T)/S_2$. 
We have: 
$$ r_{[1,1]} \mapsto (x_1 x_2), \quad r_{[1,0]} \mapsto \left(\frac{x_1}{1-y_2/y_1} +\frac{x_2}{1-y_1/y_2} \right). $$
$$ r_{[-1,-1]} \mapsto (x_1 x_2)^{-1}, \quad r_{[-1,0]} \mapsto \left(\frac{x_1^{-1}}{1-y_1/y_2} + \frac{x_2^{-1}}{1-y_2/y_1} \right). $$
\end{example}

\begin{example}\label{ex:(2)-(1)}
$$
    \begin{tikzpicture}[square/.style={regular polygon, regular polygon sides=4}]
  \node[draw,circle] (c1) at (0,0) {2};
  \node[draw,circle] (c2) at (1.5,0) {1};
   \draw[->] (c2) -- (c1);
\end{tikzpicture}
$$
$G=GL_2 \times GL_1, N=\C^2$, where $N$ is the standard representation under the first factor $GL_2$, and of weight $(-1,-1)$ under $GL_1$. 

Using Proposition \ref{p:abelianize} and Eqs. \eqref{FT eq1}, \eqref{FT eq2}, we can compute the abelianization of some minuscule operators, 
$$ r_{[1,0], 0} \mapsto \frac{x_{1,1}}{1 - y_{1,2} / y_{1,1}} +  \frac{x_{1,2}  }{1 - y_{1,1} / y_{1,2}}  $$
$$ r_{[1,1], 0} \mapsto x_{1,1} x_{1,2} , \quad  r_{[-1,-1], 0} \mapsto (x_{1,1} x_{1,2})^{-1} (1 - y_2 / y_{1,1})(1 - y_2 / y_{1,2}). $$
$$ r_{[-1,0], 0} \mapsto \frac{x_{1,1} (1 - y_2 / y_{1,1}) }{1 - y_{1,1} / y_{1,2}} +  \frac{x_{1,2} (1 - y_2 / y_{1,2}) }{1 - y_{1,2} / y_{1,1}} $$
$$ r_{[0,0], 1} \mapsto x_2 (1 - y_2 / y_{1,1})(1 - y_2 / y_{1,2}) , \quad  r_{[0,0], -1} \mapsto x_2^{-1}.$$

\end{example}

\subsection{Deformed Coulomb branch and formulas for monopole operators } \label{ss:quiver-local-geom}

Let $Q, \vec d, \vec m$ be the fixed quiver data; we consider the corresponding group $G$ and representation $N$.  We write $\beta \in T_F$ for a deformation parameter.  Let us observe: 

\begin{proposition} \label{removing framing}
There is a canonical isomorphism
$$\MCB(G, N)_\beta \RM \{r_{(\vec 1, \cdots, \vec 1)} = 0 \}  = \MCB(G, N_B). $$
In particular, the left hand side is independent of the parameter $\beta$. 
\end{proposition}
\begin{proof}
We can apply \cite[Theorem 2.9]{KWWY22}(recalled here as Proposition \ref{p:invert}) to the central cocharacter $\xi=(\vec 1, \cdots, \vec 1)$, and observe that if $N=N_B \oplus N_F$, then $N_F$ is a direct sum of representations with weight $1$, and $N_B$ is a direct sum  representations with weight $0$. Note that $T_F$ acts on $N_B$ trivially, hence the space is independent of the deformation parameters. 
\end{proof}

In particular, from Proposition \ref{removing framing} we obtain from  (a $T_F$ family of) restrictions of the monopole operators on $\MCB(G, N)$ to functions on $\MCB(G, N_B)$.

For any $\lambda$ minuscule coweight of $G$, we have monopole operator $r_\lambda$ as a global function on $\MCB(G,N)_\beta$.
Using Proposition \ref{p:abelianize} and \ref{p:gln-monopole}, with the adaptation to the deformed Coulomb branch as in Proposition \ref{prop deformed abelian coulomb branch coordinates},  we can express its restriction to $(T^\vee \times T_O)/W$ using abelianized coordinates $(y_{i,\alpha}, x_{i,\alpha})$. 

For $GL(d, \C)$, $l=1, \cdots, d-1$, let 
$$ \omega_l^\vee = (\underbrace{1, \cdots, 1}_{\text{$l$ many $1$s}}, 0, \cdots 0) \In \Z^{d} $$
be $l$-th fundamental coweight, where $1$ appear $l$ times.  Let $\vec 1 = (1,\cdots, 1) \in \Z^d$. We also note that 
$$ -w_0 \omega_l^\vee = (0, \cdots, 0, \underbrace{-1, \cdots, -1}_{\text{$l$ many $-1$s}}).$$ 

For $G = \prod_{i=1}^r GL(d_i)$, let $\omega_{i,j}^\vee$ be the coweight of $G$ that equals $\omega_j^\vee$ on the $i$-th factor. 
$$ \omega_{i,j}^\vee = (\vec 0, \cdots, \vec 0, \underbrace{(1, \cdots, 1, 0, \cdots 0)}_{\text{the $i$-th factor, with $j$ many $1$s.}}, \vec 0, \cdots, \vec 0) \in \Z^{d_1} \times \cdots \times \Z^{d_n}.  $$

In general, we will use  $(\vec a_1, \cdots, \vec a_r) \in \Z^{d_1} \times \cdots \Z^{d_r}$ for a cocharacter of $G$.

\begin{proposition} \label{p:Pn-class}
For $i \in 1, \cdots, r$, the monopole operator $r_{-w_0  \omega_{i,1}^\vee}$ restricted to $(T^\vee \times T_O)/W$ can be written in the abelianized coordinates as
    $$ ab(r_{-w_0  \omega_{i,1}^\vee}) = \sum_{(i,\alpha_i)} x_{i,\alpha_i}^{-1} \frac{\prod_{(j, \alpha_j) \to (i, \alpha_i)}\left(1 - \frac{y_{j, \alpha_j}}{y_{i, \alpha_i}} \right) \prod_{[i,\beta_i] \to (i,\alpha_i)}\left(1 - \frac{a_{i,\beta_i}}{y_{i, \alpha}} \right)}{\prod_{(i,\beta_i) \neq (i,\alpha_i)} (1-\frac{y_{i,\alpha_i}}{ y_{i,\beta_i}})} $$
where the product is over all possible indices  with fixed $i$. 
\end{proposition}
\begin{proof}
This follows from Propositions \ref{p:abelianize}, \ref{p:gln-monopole} (Eq \eqref{FT eq2}), and \ref{p:abelian-remove-matter}. The effect of deformation is to associate the framing edge $[i,\beta_i] \to (i,\alpha_i)$ the Euler characteristic of $T \times T_F$ instead of $T$ only, hence we get the factor $1 - \frac{a_{i,\beta_i}}{y_{i, \alpha}}$ instead of $ 1 - \frac{1}{y_{i, \alpha}}$ (see Prop \ref{prop deformed abelian coulomb branch coordinates}). 
\end{proof}

\begin{definition} \label{d:u-variable}
    For $y \in T_O$, we define fiberwise coordinates for $T^\vee$, where $i=1,\cdots, r$, and $\alpha_i=1, \cdots, d_i$. 
    \begin{equation} \label{def u coord}
         u_{i,\alpha_i} := x_{i,\alpha_i}^{-1} \frac{\prod_{(j, \alpha_j) \to (i, \alpha_i)}\left(1 - \frac{y_{j, \alpha_j}}{y_{i, \alpha_i}} \right) \prod_{[i,\beta_i] \to (i,\alpha_i)}\left(1 - \frac{a_{i,\beta_i}}{y_{i, \alpha}} \right)}{\prod_{(i,\beta_i) \neq (i,\alpha_i)} (1-\frac{y_{i,\alpha_i}}{ y_{i,\beta_i}})}
    \end{equation}
\end{definition}

The definition is arranged to ensure 
\begin{equation}\label{W from us}
    ab(r_{-w_0  \omega_{i,1}^\vee}) = \sum_{(i,\alpha_i)}  u_{i,\alpha_i}
\end{equation}

\begin{proposition} \label{p:111}
 The monopole operator $r_{(\vec 1, \cdots, \vec 1)}$ restricted to $(T^\vee \times T_O)/W$ can be written in the abelianized coordinates as
$$ 
ab(r_{(\vec 1, \cdots, \vec 1)}) = \prod_{i,\alpha} x_{i,\alpha}
$$
\end{proposition}
\begin{proof}  We use Proposition \ref{p:abelianize}. 
The $G(\O)$ orbit $Gr_G^\lambda$ for the coweight $\lambda =  (\vec 1, \cdots, \vec 1)$ is a point. 
\end{proof}

\subsection{Some geometry of deformed Coulomb branch} 

We now study the geometry of $ \MCB(G,N)_\beta$.
As in Eq \eqref{e:W-unfold}, 
we consider the $W$-fold cover, where $W$ is the Weyl group of $G$. 

\begin{equation} \label{e: deformed W-unfold}
 \begin{tikzcd}
    \wt \MCB(G,N)_\beta  \ar[r] \ar[d, "\wt \mu"] & \MCB(G,N)_\beta \ar[d, "\mu"]   \\
    T \ar[r] &  T/W
\end{tikzcd}
\end{equation}

Over $T_O$, we have fiber $T^\vee$ with fiber coordinates $x_{i,\alpha_i}$, where $i=1,\cdots,r$ and $1 \leq \alpha_i \leq d_i$.  

There are three types of hyperplanes  in $T$, over which the fiber of $\wt \mu$ degenerate. For each such hyperplane $H$, we will describe the fiber behavior over the neighborhood $U(\wt H^o)$. 
\begin{enumerate}
    \item (Root Hyperplane) For each circle node $(i)$, we have downstairs root hyperplanes $H_{i, \alpha_i, \alpha_i'}$ labelled by $1 \leq \alpha_i < \alpha_i' \leq d_i$:
    $$ \wt H_{i, \alpha_i, \alpha_i'} := \{y_{i,\alpha_i} = y_{i,\alpha'_i}\}.$$    
    We write the function cutting out the union of all root hyperplanes as
    \begin{equation}
        \Delta_R := \prod_{(i,\alpha)\neq (i, \beta)}
        (y_{i, \alpha} - y_{i, \beta})
    \end{equation}

    \item (Internal matter hyperplane) For each internal edge $(i) \to (j)$, we have downstairs weight hyperplanes $H_{(i, \alpha_i) \to (j, \alpha_j)}$  labelled by $1 \leq \alpha_i \leq d_i$, $1 \leq \alpha_j \leq d_j$ 
    $$ \wt H_{(i, \alpha_i) \to (j, \alpha_j)} := \{y_{i,\alpha_i} = y_{j,\alpha_j}\}.$$
    We write the function cutting out the union of all the internal matter hyperplanes as 
    \begin{equation} \Delta_I := \prod_{(i, \alpha_i) \to (j, \alpha_j)} (y_{i, \alpha_i} - y_{j, \alpha_j})  \end{equation}

    \item (Framing matter hyperplane) For each framing edge $(i) \gets [i]$, we have downstairs  weight hyperplanes $H_{(i, \alpha_i) \gets [i, \beta_i]}$ (resp.  $\wt \hcal_{(i, \alpha_i) \gets [i, \beta_i]}$) labelled by $1 \leq \alpha_i \leq d_i$, $1 \leq \beta_i \leq m_i$ 
    $$ \wt H_{(i, \alpha_i) \gets [i, \beta_i]} = \{y_{i,\alpha_i} = a_{i,\beta_i}\}.$$ 

    We write the function cutting out the union of all the framing matter hyperplanes as 
    \begin{equation} \Delta_F := \prod_{(i, \alpha_i) \gets [i, \beta_i]} (y_{i, \alpha_i} - a_{i, \beta_i})  \end{equation}

\end{enumerate}

The geometry with the three kinds of hyperplane divisors removed is given by the following Proposition.

\begin{proposition}\label{p:u-coord}
The affine variety $\wt \MCB(G,N)_{\beta}^O \setminus \{r_{(\vec 1, \cdots, \vec 1)} = 0\}$ has ring of global
functions given by
    $$\wt \acal(G, N)_\beta^O := 
    \C[\{ x_{i,\alpha}^{\pm 1}, y_{i, \alpha}^{\pm 1}\}_{i,\alpha}, \Delta_R^{-1}, \Delta_F^{-1}, \Delta_I^{-1}] =
        \C[\{ u_{i,\alpha}^{\pm 1}, y_{i, \alpha}^{\pm 1}\}_{i,\alpha}, \Delta_R^{-1}, \Delta_F^{-1}, \Delta_I^{-1}].
    $$ 
\end{proposition}
\begin{proof}
The variety is affine, having been
obtained from the definitionally affine multiplicative Coulomb branches by removing divisors.  The equality of the two purported expressions for the ring of global functions is apparent from the definition of the $u$ variables (Definition \ref{d:u-variable}). 

To compute the ring of global functions, we begin by noting that the global functions on $\MCB(G, N_B)^O$ are, by 
Proposition \ref{p:local-geom}, item \ref{trivialization on O locus}, 
given by 
   $$\wt \acal(G, N_B)^O := 
    \C[\{ x_{i,\alpha}^{\pm 1}, y_{i, \alpha}^{\pm 1}\}_{i,\alpha}, \Delta_R^{-1},  \Delta_I^{-1}]
    $$ 
Under the isomorphism of Proposition \ref{removing framing}, we obtain 
$\wt \MCB(G,N)_{\beta}^O$ by further
removing the preimages of the framing matter hyperplanes
(i.e. inverting $\Delta_F$).
\end{proof}

\subsection{Proof of Theorem \ref{cylindrical model intro}}

We restate the result: 

\begin{theorem} \label{cylindrical model}
    Fix a map $C \to \mathbf{T}/W = \prod_i Sym^{d_i} \C^*_y$ transverse to $(\mathbf{T} \setminus \mathbf{T}_O) / W$, and consider
    the corresponding 
    $S \to C \times \C^*_y$.

    Lifts to $C \to \MCB(\Gamma, \vec d)$ are in bijection with maps $\Phi_u: S \to \P^1_u$ satisfying the following conditions: 
\begin{enumerate}
\setcounter{enumi}{-1}
    \item $\Phi_u(S_O) \subset \C^*_u$. 
    \item $\Phi_u$ has simple poles over points in $S_{(i)}$.
    \item $\Phi_u$ has simple zeros over $S_{[i] \to (i)}$.
    \item For $(s_i, s_j) \in S_{(i) \to (j)}$, we require that either $\Phi_u$ has a simple zero at $s_i$ and $\Phi_u(s_j) \in \C^*$,  or $\Phi_u$ has a simple zero at $s_j$ and $\Phi_u(s_i) \in \C^*$. 
\end{enumerate}
\end{theorem}
\begin{proof}
    Such sections form a sheaf of sets over $C$.  
    By hypothesis, $C \to \mathbf{T}/W$ meets all hyperplanes distinctly and transversely; we are (by the sheaf property of sections) therefore reduced to the case where $C$ is a small disk meeting at most one of the hyperplanes.  In the complement of the hyperplanes, the result is immediate from the identification of Proposition \ref{removing framing} and the trivialization of Proposition \ref{p:u-coord}. 
    
    Thus, we may compute in one of the charts described in Proposition \ref{p:local-geom}.  Observe that, in the quiver gauge theory setting, the chart itself is again the Coulomb branch of a quiver gauge theory, which moreover is either abelian (all entries of dimension vector are $1$) and has a single arrow, or has no arrows and only a single non-abelian node, which has dimension $2$.  
    
    Under the isomorphism of Proposition \ref{p:local-geom}, we may compare the $u$ coordinates of the more complicated original quiver and of the simplified quiver.  From the explicit defining formula of Definition \ref{d:u-variable}, we see that the ratio is a regular invertible function on $T_O$, and extends to a regular invertible function on $U(H^o)$  when restricted to the smaller locus.  

   More precisely, for $p \in C \cap (T \RM T_O)$, one can choose a lift $\wt p \in T$, and find an open neighborhood $\wt U$ of $\wt p$, and a simplified quiver $Q'$, such that the abelianized Coulomb branch $\wt \mcal_Q$ with rational functions $u_{i,a}$, and the simplified one $\wt \mcal_{Q'}$ with rational function $u'_{i,a}$ satisfies the condition that $\wt \mcal_Q \cong \wt \mcal_{Q'}$ over $\wt  U$, and there is an invertible function $R_{i,a}$ on $U_p$, such that $u_{i,a} = R_{i,a} u'_{i,a}$ on $\wt U$. That is, $u_{i,a}$ and $u'_{i,a}$ have the same pole and zero structure when restricted to $C$. 
    
    In particular, if we consider the assertion of the theorem on the smaller locus, it does not matter whether we use the $u$ coordinates of the original ambient quiver, or of the simplied quiver.  We are thus reduced to proving the theorem for quivers which have either no arrows and a single nonabelian node which moreover has dimension $2$, or abelian quivers with one arrow.  As the $u$ coordinates are compatible with disjoint union of quivers, we are reduced to checking for the situations of Examples \ref{ex:pure-abelian}, \ref{ex:(1)-(1)}, \ref{ex:[1]-(1)}, and \ref{ex:gl2}.   
    These can be checked explicitly by hand.  
\end{proof}

\section{KLRW category} \label{s:KLRW}

Here we fix some terminology for discussing the KLRW categories, and establish some recognition principles.

For any $n \in \N=\Z_{\geq 0}$, let $[n]=\{1,\cdots, n\}$. 

Let $\Gamma$ be a framed quiver with nodes labelled by $[n]$. We will also use $\Gamma$ to denote the set of non-framing nodes. The non-framing node is written as $(i)$ where $i \in [n]$, and the framing node is $[i]$. The arrows are written as $(i) \to (j)$ between non-framing nodes, and $[i] \to (i)$ from the framing node to the non-framing node. For each $i \in [n]$, we have dimension $d_i \in \N$ for the non-framing node and $m_i \in \N$ for the framing node. We use $\vec d, \vec m$ for the tuples of numbers.  We use $(i,j)$ for $i \in [n], j \in [d_i]$, and use $[i,j]$ for $i \in [n], j \in [m_i]$.

\begin{definition}
    A {\bf red points configuration on $S^1$} of type $\vec m$ is a finite subset $R \In S^1$, where $R = \sqcup_{[i]} R_i$ and $|R_i| = m_i$.  Given a red points configuration, a {\bf black points configuration on $S^1$} of type $\vec d$ is finite subset
    $B \In S^1 \RM R$, where $B = \sqcup_{(i)} B_i$ and $|B_i|=d_i$.  
\end{definition}

A red points configuration is denoted as $\{\theta_\rij\}$ where $\theta_\rij \in S^1$ and the ordering of the $\theta_{[i,j]}$ for a fixed $i$ is irrelevant. Similarly, and a black points configuration is denoted as $\{\theta_{(i,j)}\}$. 

Typically we will have fixed a red point configuration. We use $\ccal$ to denote the set of black points configurations. 

Recall that the KLRW category (for a fixed red points configuration) is the following. The set of objects is the set of black points configurations. The morphism space $Hom_{KLRW}(c_1, c_2)$ is the vector space generated by KLRW strand diagrams with $c_1$ at the bottom and $c_2$ at the top modulo relations, and composition is vertical concatenations. 

Recall we have introduced parameters $\hbar$ and $\eta$, and defined the KLRW category over $\C[\hbar, \eta]$. These parameters appear in the KLRW relations. 

We will be interested in certain degenerate cases, first $\hbar=0=\eta$, and also just $\hbar = 0$. In these cases, it is convenient to describe the morphisms in terms of the following basis. 

\begin{definition} \label{taut strands diagram}
    A {\bf taut strands diagram} on the cylinder from black points configuration $c_1$  to $c_2$, is a collection of strands that ends on $c_1$ at bottom and $c_2$ at top, such that the strands have no horizontal tangency, no triple intersection points,  and no bigon between any two strands (red or black). Two taut strands diagrams are equivalent, if they are isotopic modulo braid relation. 

    A {\bf weighted taut strands diagram} is a taut strands diagram where each strand is decorated by a non-negative number (called weight). 

    We write $\pcal(c_1, c_2)$ as the equivalence class of weighted taut strands diagram from black points configuration $c_1$  to $c_2$. 
\end{definition}

If  $D_{12} \in \pcal(c_1, c_2)$ is a weighted taut strand diagram, a {\bf taut decomposition} of $D_{12}$ are two diagrams $D_{1a} \in \pcal(c_1, c_a), D_{a2} \in \pcal(c_a,c_2)$ such that the concatenation $concat(D_{a2}, D_{1a})$ is taut and equal to $D_{12}$ and the sum of weights of each  strands matches with $D_{12}$. We write this as 
$$ D_{12} = concat(D_{a2}, D_{1a}). $$
and say $D_{12}$ is the {\bf taut concatenation} of $D_{a2}$ and $D_{1a}$. 

\begin{proposition} \label{taut strands at zero}
For any black points configurations $c_1, c_2 \in \ccal$, we have canonical isomorphism of vector spaces
$$ Hom_{KLRW}(c_1, c_2)_{\hbar=\eta=0} \cong \C \pcal(c_1, c_2). $$
\end{proposition}
\begin{proof}
We may use KLRW relations to straighten the strands in any KLRW diagrams and move dots to the top of the diagram, then by viewing the number of dots on a strand as the weight on that strand we have the bijection. 
\end{proof}

\begin{lemma} \label{lm: KLRW composition hbar eta all zero}
Let $D_{12} \in \pcal(c_1,c_2)$ and $D_{23} \in \pcal(c_2,c_3)$ be weighted taut strands diagrams, and let $D_{13}$ be the concatenation and straightening of $D_{23} \circ D_{12}$. If we view $D_{ij}$ as morphisms in the KLRW category with $\hbar=\eta=0$, then $D_{23} \circ D_{12}$ either equals zero or equals to $D_{13}$, and equals $D_{13}$ if and only if one of the following happens
  \begin{enumerate}[(1)]
        \item There is no bigon cancellation between black strands of same label. 
        \item There is no bigon cancellation between black strands with labels $(i)$ and $(i')$ if  $(i) \to (i')$. 
        \item There is no bigon cancellation between black strands and red strands of the same label (corresponding to a framing edge $[i] \to (i)$). 
    \end{enumerate}
\end{lemma}
\begin{proof}
    This is immediate following the KLRW relations. 
\end{proof}

\begin{proposition} \label{stage2:KRLW has basis}
When $\hbar=0$, we have a canonical isomorphism of vector spaces 
$$ Hom_{KLRW}(c_1, c_2)_{\hbar=0} \cong \C[\eta] \cdot \pcal(c_1, c_2). $$
\end{proposition}
\begin{proof}
First, any KLRW diagram can be expressed as a linear combination of taut KRLW diagrams with dots at the bottom. Next, since $\hbar=0$, for any 3 strands in a KLRW diagram we have braid relation, hence we have a bijection between taut KLRW diagram (modulo relations) and taut weighted strands diagram. 
\end{proof}

\begin{definition}
A weighted taut diagram (resp. KLRW diagram) is called {\em elementary} if it is one of the two types
\begin{enumerate}
    \item all strands vertical, and exactly one black strand has weight $1$ (resp. carries a dot). 
    \item there is at most one crossing of any type, and all strands have zero weights (resp. no dots). 
\end{enumerate}
\end{definition}
There is an obvious bijection between elementary KLRW diagram and elementary weighted taut diagram, by exchanging weight $1$ with one dot. For this reason, we do not distinguish elementary KLRW or weighted elementary diagram. Let $Elem(c_1, c_2) \In \pcal(c_1, c_2)$ be the subset of elementary diagrams from $c_1$ to $c_2$.

The morphisms in KLRW category can be generated by elementary diagrams modulo relations, which we formalize below. 
\begin{definition}
    Let $FreeElem$ be the category with objects $\ccal$ and hom space be 
$$ Hom_{FreeElem}(c, c') = \oplus_{k \geq 0} \oplus_{c=c_0, c_1, \cdots, c_k=c'} \C[\eta, \hbar] Elem(c_0 \to c_1 \cdots \to c_k) $$
where $Elem(c_0 \to \cdots \to c_k) = \prod_{i=1}^k Elem(c_{i-1}, c_i)$ is a chain of elementary weighted taut diagrams of length $k$, and composition is concatenation of chains.
\end{definition} 
There is an obvious functor from $FreeElem$ to $KLRW$ by composing the chain of elementary diagrams, and the map on hom space is surjective.

\begin{proposition} \label{p: KLRW h=0 recog}
Let $KLRW_{\hbar=0}$ be the KLRW category with $\hbar=0$, where we identify the morphism space by $Hom_{KLRW}(c_1, c_2)_{\hbar=0} \cong \C[\eta] \cdot \pcal(c_1, c_2)$ and we denote the composition of $c_1 \xto{D_{12}} c_2 \xto{D_{23}} c_3$ as $D_{23} \circ_K D_{12}$. 

Let $\acal$ be a category with the same set of objects and morphism space as $KLRW_{\hbar=0}$ and we denote the composition in $\acal$ of $c_1 \xto{D_{12}} c_2 \xto{D_{23}} c_3$ as $D_{23} \circ_\acal D_{12}$.

If the two compositions matches
\begin{equation} \label{char KLRW hbar=0}
D_{23} \circ_{K} D_{12} = D_{23} \circ_{\acal} D_{12} 
\end{equation}
for the following subset of diagrams
\begin{enumerate}
    \item $D_{12}, D_{23}$ are elementary strands diagram. 
    \item $D_{12}, D_{23}$ are such that $concat(D_{12}, D_{23})$ is taut. 
\end{enumerate}    
Then the two composition matches\eqref{char KLRW hbar=0} for all diagrams $D_{23}, D_{12}$. In particular, the category $\acal$ and $KLRW_{\hbar=0}$ are equivalent. 
\end{proposition}
\begin{proof}

We first prove that, for  $D_{23}$ elementary (resp. $D_{12}$ elementary), Eq \eqref{char KLRW hbar=0} holds. 
We only prove the case for $D_{23}$. 

If $concat(D_{23}, D_{12})$ is taut, then there is nothing to prove by the hypothesis. Hence we only need to consider the case $D_{23}$ has a crossing and $concat(D_{23}, D_{12})$ has a bigon. By braid isotopy, we may move the bigon to the top and write $D_{12}$ as 
$$D_{12} = concat(D_{42}, D_{14} ) = D_{42} \circ_K D_{14} = D_{42} \circ_\acal D_{14},$$
where $D_{42}$ is elementary and the concatenation is taut. Since we know 
$$ D_{23} \circ_{K} D_{42} = D_{23} \circ_{\acal} D_{42} $$ 
by hypothesis, and the result is $\sum_i c_i (D_{43})_i $ where  $(D_{43})_i$ are diagrams without crossing,  hence the concatenation of $(D_{43})_i$ with $D_{14}$ is taut, and we get
$$ (D_{43})_i \circ_K D_{14} =  (D_{43})_i \circ_\acal D_{14}.$$
Putting these together, we have 
$$  D_{23} \circ_{K} D_{12} = D_{23} \circ_{K} (D_{42} \circ_K D_{14}) =  (D_{23} \circ_{K} D_{42}) \circ_{K} D_{14} = [\sum_i c_i (D_{43})_i] \circ_{K} D_{14} = \sum_i c_i [(D_{43})_i \circ_{K} D_{14}]  $$
$$ =  \sum_i c_i [(D_{43})_i \circ_{\acal} D_{14}]= [ \sum_i c_i (D_{43})_i]  \circ_{\acal} D_{14} = (D_{23} \circ_{\acal} D_{42}) \circ_{\acal} D_{14} = D_{23} \circ_{\acal} (D_{42} \circ_{\acal} D_{14})= D_{23} \circ_{\acal} D_{12}.
$$

Finally, we prove for general $D_{23}$,  by induction on length of $D_{23}$ where length of a diagram is the number of total number of crossings (of any type). Assume the composition matches for all $D_{23}$ of length $\leq N-1$  and we have $length(D_{23})=N$.  We do a taut decomposition
$$ D_{23} = concat(D_{43},  D_{24}) = D_{43} \circ_K D_{24} = D_{43} \circ_\acal D_{24} $$
where $D_{24}$ is elementary, and $length(D_{43}) = N -1$. Then 
$$   D_{23} \circ_K D_{12} =   D_{43} \circ_K D_{24} \circ_K D_{12} =  D_{43} \circ_{K} ( D_{24} \circ_{K} D_{12}) $$
$$ =  D_{43} \circ_{\acal} ( D_{24} \circ_{\acal} D_{12}) = (D_{43} \circ_{\acal} D_{24}) \circ_{\acal} D_{12}  = D_{23} \circ_{\acal} D_{12}. $$ 
\end{proof}

Finally we consider KLRW category for $\hbar \neq 0$.

Let $D$ be a KLRW diagram, let $cross_i(D)$ be the number of crossings among black strands with label $(i)$, and let $cross(D) = \sum_{i \in \Gamma} cross_i(D)$. 
Note that the function $cross$ is defined on KLRW diagrams, not modulo equivalence relations. 

For any non-negative integer $k$, let
$ \fcal_k Hom_{KLRW}(c_1, c_2) \In  Hom_{KLRW}(c_1, c_2)$ be the subspace spanned by the equivalence classes of KRLW diagrams $D$ from $c_1$ to $c_2$ with $cross(D) \leq k$. Clearly $\fcal_k$ defines an increasing filtration on $Hom_{KLRW}(c_1, c_2)$, and we have
\begin{lemma}
  The composition in KLRW category preserves filtration, namely
  $$ - \circ_{KLRW} -: \fcal_{l} Hom_{KLRW}(c_2, c_3) \otimes \fcal_{k} Hom_{KLRW}(c_1, c_2) \to \fcal_{k+l} Hom_{KLRW}(c_1, c_3).$$
\end{lemma}
\begin{proof}
    Indeed, the composition in KLRW category is by concatenation, and crossing number is additive on concatenation of diagrams. 
\end{proof}

For weighted taut strand diagram $D \in \pcal(c_1, c_2)$, we may also define $cross_i(D)$ for and $cross(D)$ as above. Let $\pcal_k(c_1, c_2)$ be the subspace spanned by $D$ with $cross(D)=k$. Then we have direct sum decomposition
$$ \pcal(c_1, c_2) = \oplus_{k \geq 0} \pcal_k(c_1, c_2). $$

\begin{lemma} \label{lm:GrKLRW}
For any $c_1, c_2 \in \ccal$, we have canonical isomorphism 
$$ \Phi_k: Gr_k^\fcal Hom_{KLRW}(c_1, c_2) \cong \C[\eta, \hbar] \cdot \pcal_k(c_1, c_2). $$
\end{lemma}
\begin{proof}
Let $D$ be a KLRW diagram with $cross(D) \leq k$, we may use KRLW relations to express $D$ as a linear combination of taut KRLW diagrams (i.e. with out any bigon)
$ D = \sum_i c_i D_i. $
We define $\Phi_k(D) = \sum_i c_i \Phi_k(D_i)$, where 
$\Phi_k(D_i)=0$ if $cross(D_i) < k$ and otherwise $\Phi_k(D_i)$ is defined by viewing $D_i$ as taut weighted diagram where weight is the number of dots. One can check that different decomposition of $D$ as taut KLRW diagrams will give the same $\Phi_k$ image and $\Phi_k$ naturally descend to $Gr_k^\fcal$. 
\end{proof}

Let $Gr(KLRW)$ denote the associated graded category, where the set of objects is still $\ccal$, and the hom space is
$Hom_{Gr(KLRW)}(c_1, c_2) = \oplus_k Gr_k^\fcal(Hom_{KLRW}(c_1, c_2))$, and composition is naturally induced.

\begin{lemma} \label{lm: filtered isom}
    Let $\acal, \bcal$ be two $R$-linear categories.  Assume given an exhausting increasing filtration indexed by $\Z_{\geq 0}$ on the hom spaces and with compositions compatible with filtrations. If $\psi: \acal \to \bcal$ is a functor that respects the filtration, and $Gr\psi: Gr\acal \to Gr\bcal$ is an isomorphism, then $\psi$ is an equivalence of category
\end{lemma}
\begin{proof}
    If $F: V \to W$ are morphisms of filtered $R$-module with isomorphism on associated graded, then $V \congto W$. 
\end{proof}

\begin{lemma} \label{lm: short-cut}
Let $\acal$ be a category defined over $\C[\hbar, \eta]$ with objects $\ccal$, and let there be a functor
    $$ \psi: FreeElem \to \acal $$ 
such that $\psi$ satisfies the following condition
\begin{enumerate}
    \item annihilate all KRLW relations except possibly the braid relations (Figure \ref{fig:KLRW-jij}, \ref{fig:KLRW-braid-red})
    \item Let $D$ be the elementary diagram of a crossing of type $(i)-[i]$, i.e. a black strand of label $(i)$ and a red strand with label $[i]$. Then $\psi(D)$ is a morphism in $\acal$ that is not a zero-divisor, i.e., for any pre-composable morphism $E$, 
    $$ \psi(D) \circ E = 0 \LRA E = 0. $$
    and for any post-composable morphism $E$ for $\psi(D)$, 
    $$ E \circ \psi(D) = 0 \LRA E = 0. $$
    \item Let $D$ be the elementary diagram of a crossing of type $(i)-(j)$, where $(i), (j)$ are adjacent nodes. Then $\psi(D)$ is not a zero-divisor (in the above sense). 
\end{enumerate}
then $\psi$ in fact also annihilates the braid relations, and hence descends to a functor $\psi: KRLW \to \acal$. 
\end{lemma}
\begin{proof}
The two braid relations that we want to show are as in Figure \ref{fig:KLRW-jij} and \ref{fig:KLRW-braid-red}. 
We pre-compose the equation with a non-zero divisor and show the new equation holds. 

The braid with red-strand relation is equivalent to Figure \ref{fig:braid-with-red-proof}, which holds by dot pass crossing relation. 

\begin{figure}[h]
    \centering
    \includegraphics[width=1\linewidth]{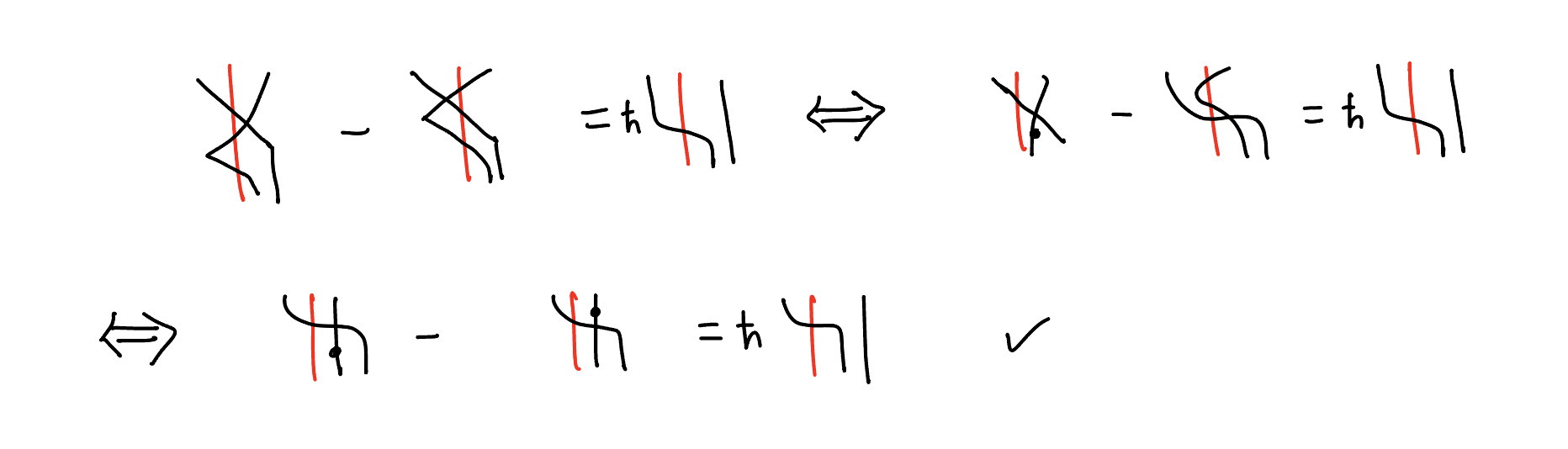}
    \caption{Proof for braid relations between two black strands and a red strands of the same label.}
    \label{fig:braid-with-red-proof}
\end{figure}

The braid relation between black strands of adjacent nodes is equivalent to the following equations of diagrams (Figure \ref{fig:braid-iji-proof}), which does hold. 
\begin{figure} [h]
    \centering
    \includegraphics[width=0.8\linewidth]{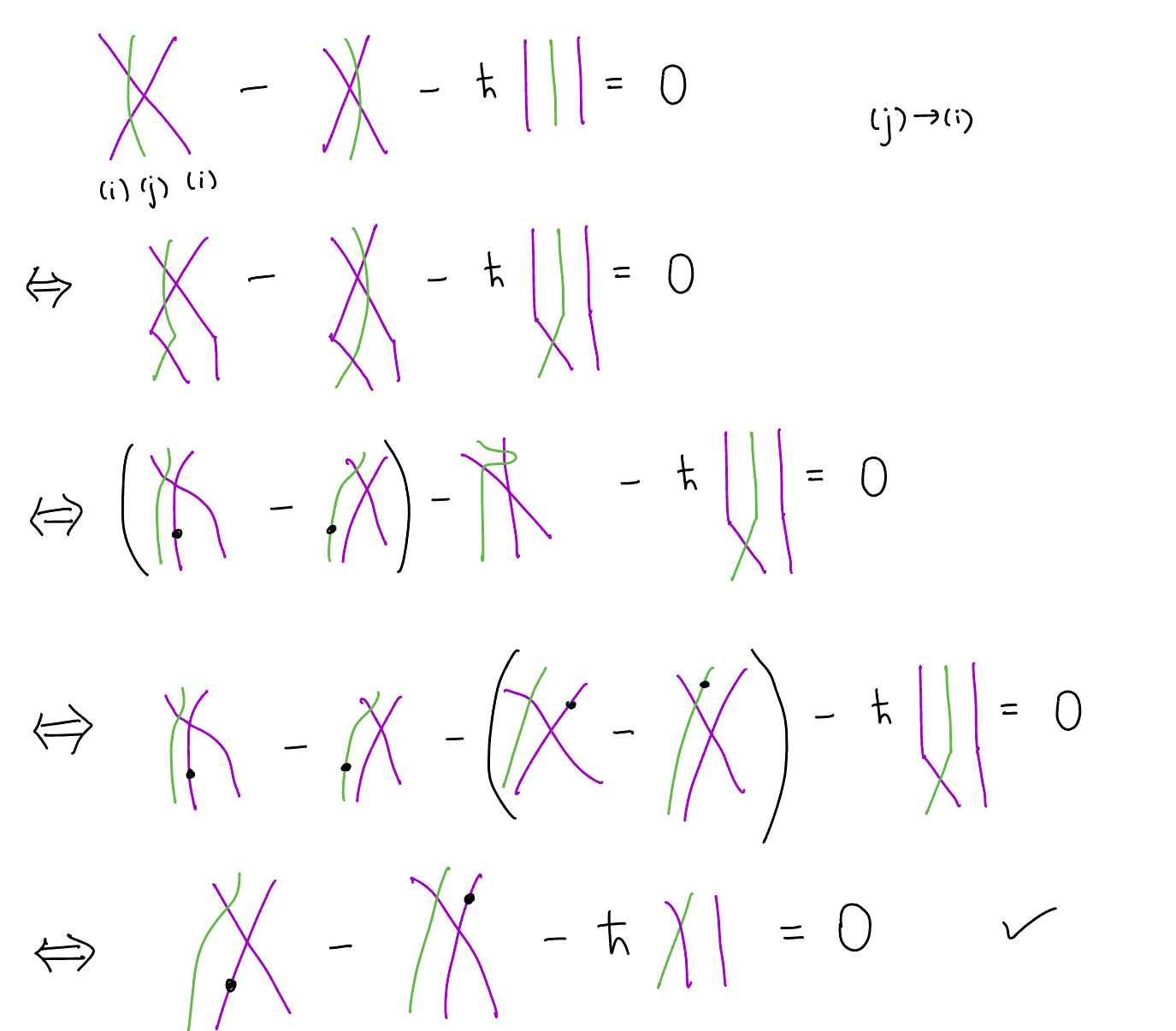}
    \caption{Proof for braid relations between black strands labelled by adjacent circle nodes. }
    \label{fig:braid-iji-proof}
\end{figure}


\end{proof}

\begin{proposition} \label{p: KLRW recog}
    Let $\acal$ be a category with objects $\ccal$ and with hom spaces $Hom_\acal(c_1, c_2)$ as a $\C[\eta, \hbar]$-free module. We assume that
\begin{enumerate}
    \item The hom spaces $Hom_\acal(c_1, c_2)$ has an increasing filtration  $\fcal_k(Hom_\acal(c_1, c_2))$ compatible with composition. 
    \item The associated graded has a canonical isomorphism
    $$ \Psi_k: Gr_k^\fcal(Hom_\acal(c_1, c_2)) \cong \C[\eta, \hbar] \cdot \pcal_k(c_1, c_2). $$
    \item $\Psi_k$ defines an isomorphism of graded category 
    $$ \Psi: Gr(\acal) \cong Gr(KLRW). $$ 
    \item There is a functor 
    $$ \psi: KLRW \to \acal $$ 
    that respects the filtration, and $Gr\psi = \Psi$. 
\end{enumerate}
Then we have $\psi$ is an equivalence of category. 
\end{proposition}
\begin{proof}
We may apply Lemma \ref{lm: filtered isom} to $\psi: KLRW \to \acal$. 
\end{proof}

\section{Floer homologies in $(\mathbf{T}_O \times \mathbf{T}^\vee)/\mathbf{W}$}
Our eventual goal is to compute Floer homology in the Coulomb branch for certain distinguished Lagrangians. It is possible to describe the Lagrangians, their positive wrapping and their intersections in an explicit way, and the only difficulty is to count disks that defines the $A_\infty$ structure. We will introduce two formal parameters $\eta$ and $\hbar$, where a factor of $\eta$ indicate a disk that intersects the matter divisor, and a factor of $\hbar$ indicate a disk intersects the root divisor. In this section, we set $\eta, \hbar=0$, which means the computation are in the complement of these divisors, which we call the punctured space. 

This section will describe the punctured space, certain distinguished Lagrangians in it and their hom spaces. In the next two sections, we will gradually add back the divisors. This section can be read without any knowledge of the Coulomb branches. We use the same notation as section \ref{s:KLRW}.

\subsection{Fukaya category setup}

We define
$$ \bfT_y = \prod_{(i,j)} \C^*_{y_{(i,j)}}, \quad \bfT^\vee_u = \prod_{(i,j)} \C^*_{u_{(i,j)}}, \quad \bfW = \prod_{(i)} S_{d_i}. $$
with coordinates as subscript. 
We also define a punctured version of $\bfT_y$
$$ \bfT_O = \left\{ (y_{(i,j)}) \in \bfT \mid 
\begin{cases}
    y_{(i,j)} \neq y_{(i,j')} & \forall (i,j) \neq (i,j') \\
    y_{(i,j)} \neq a_{[i,k]} & \forall (i,j), [i,k] \\
    y_{(i,j)} \neq y_{(i',j')} & \forall (i) \to (i'). 
\end{cases} \right \}. $$

We choose a product holomorphic volume form (used only to fix grading/orientation data for Floer homologies) on $\bfT^\vee \times \bfT$
\begin{equation} \label{hol vol T}
\Omega = \prod_{\bij} d \log u_{\bij} \wedge d \log y_{\bij}.
\end{equation}
Note that $\Omega$ is invariant under $\bf W$ hence descends to $(\mathbf{T}_O \times \mathbf{T}^\vee)/W$.

A {\bf product Lagrangian} is 
$$ {\bf L} = \prod_{(i,\alpha)} L_{u_{\bij}} \times L_{y_{\bij}} \In  \bfT^\vee \times \bfT_O, $$
where $L_{u_\bij} \In \CS_{u_\bij}$ and $L_{y_\bij} \In \C^*_{y_\bij}$ are Lagrangians (i.e. closed 1-dimensional submanifold without boundary). 

A {\bf symmetric product Lagrangian} is the quotient image of a product Lagrangian by $\bfW$
$$ {\bf L / W} \In (\bfT^\vee \times \bfT_O) / {\bf W}. $$

Given two product Lagrangians $\bf L_1,  L_2$, their intersections are product of intersections
$$ \bfL_1 \cap \bfL_2 = \prod_\bij (L_{1,u_\bij} \cap L_{2,u_\bij}) \times  (L_{1,y_\bij} \cap L_{2,y_\bij}).$$

Similarly, given two symmetric product Lagrangians, their intersections are 
$$ \bfL_1 /\bfW \cap \bfL_2/\bfW = \coprod_{\sigma \in \bfW} \prod_\bij (L_{1,u_\bij} \cap L_{2,u_\sigma\bij}) \times  (L_{1,y_\bij} \cap L_{2,y_{\sigma\bij}}).$$

\begin{definition}
Let $c = \{\theta_{(i,j)}\}$ be a black points configuration. We define the following symmetric product Lagrangian 
$$ \tcal_c = \{((u_\bij, y_\bij)) \in \bfT^\vee \times \bfT_O \mid \arg(y_{\bij})=\theta_\bij, \arg(u_\bij)=0\} / \bfW. $$
\end{definition}

\subsection{Floer theory in $\C^*, (\C^*)^2$ , $(\C^*)^2 \RM \Delta$ and $\Sym^2(\C^*)\RM \Delta$}
We study some model cases in this subsection.

\begin{definition}
Let $(\C^*_z, \omega =  i |z|^{-2} dz \wedge d\bar z,  \Omega = dz/z)$ be the symplectic cylinder.  
Let $T_1, T_2$ be two graded Lagrangians with conical ends that are Lagrangian section for the torus fibration $\log|z|: \C^* \to \R$.   We say {\em $T_1$ intersects $T_2$ positively and write $T_1 > T_2$},  if  $T_1, T_2$ intersects transversely and the chain complex $CF(T_1, T_2)$ is concentrated in degree $0$. 
In general, we write $T_1 > T_2 >\cdots > T_n$ if for any $i < j$ we have $T_i > T_j$.
\end{definition}

\begin{lemma} \label{lm: obvious-disk-1}
Given $T_1 > T_2 > T_3$ in $\C^*_z$, if $p_{12} \in T_1 \cap T_2$ and $p_{23} \in T_2 \cap T_3$, then there exists a unique $p_{13} \in T_1 \cap T_3$ as the composition output of $p_{23} \circ p_{12}$.     
\end{lemma}
\begin{proof} 
    We consider the universal cover $\C_u \to \C^*_z$, $z = e^u$. Choose a lift $\wt T_i$ of $T_i$ for $i=1,2,3$, such that the intersection points $p_{12}$ and $p_{23}$ lifts to $\wt p_{12} \in \wt T_1 \cap \wt T_2$ and $\wt p_{23} \in \wt T_2 \cap \wt T_3$, respectively. For $i<j$, we also have $CF(\wt T_i, \wt T_j)$ in degree $0$ (where $\Omega_{\C_u} = du$), hence $\wt T_i \cap \wt T_j$ at most at one point. We have seen that $\wt T_1 \cap \wt T_2 = \{\wt p_{12}\}$ and $\wt T_2 \cap \wt T_3 = \{\wt p_{23}\}$, it remains to show that $\wt T_1 \cap \wt T_3 \neq \emptyset$. Since $T_i$ are eventually conical, hence $Im(u)$ on $\wt T_i$ are eventually constant when $|Re(u)|$ are large enough. Let $R_i, L_i$ be the eventual value of $Im(u)|_{\wt T_i}$ on the far right and left ends, then we have 
    $$ R_1 > R_2 > R_3, \quad L_1 < L_2 < L_3. $$
    Since $R_1 > R_3, L_1 < L_3$, we see $\wt T_1 \cap \wt T_3 \neq \emptyset$, and denote the unique intersection point as $\wt p_{13}$. Then, there is an obvious disk in $\C_u$ that realizes $\wt p_{13} = \wt p_{23} \circ \wt p_{12}$. Let $p_{13} \in \C^*_z$ be the image of $\wt p_{13}$, then the image of the disk in $\C_u$ witnesses the composition $p_{13} = p_{23} \circ p_{12}$. 
\end{proof}

As a warm-up for using strand diagrams to represent morphisms in Fukaya category, we define the set of states to be
$ \ccal_1 = S^1$,  and for $c_1, c_2 \in \ccal_1$, the set of diagrams $\pcal_1(c_1, c_2)$ to be the taut strands in $S^1 \times [0,1]$ that connects $c_1$ at the bottom $S^1 \times 0$ to $c_2$ at the top $S^1 \times 1$ modulo isotopy. 
The composition of diagrams $D_{23} \circ D_{12}$ is concatenation, and is denoted as $concat(D_{23}, D_{12}) $. 

For any $c \in \ccal_1$, let $T_c = c \times \R \In S^1 \times \R \cong \C^*$ and let $T^w_c$ denote a wrapped version of $T_c$.

\begin{lemma} \label{lm:baby-comparison-1}
   (1) For any $c_1, c_2 \in \ccal_1$, if $T^w_{c_1} > T^w_{c_2}$, then there is a natural injection
    $$ \Phi: T^w_{c_1} \cap T^w_{c_2} \into \pcal_1(c_1, c_2). $$
   (2) For any $c_1, c_2, c_3 \in \ccal_1$, if $T^w_{c_1} > T^w_{c_2} > T^w_{c_3}$, and $p_{12}, p_{23}$ are intersection points corresponding to diagrams $D_{12}, D_{23}$ where $p_{ij} \in T^w_{c_i} \cap T^w_{c_j}$, then 
    $$ \Phi: p_{23} \circ p_{12} \mapsto concat(D_{23},  D_{12}). $$
\end{lemma}
\begin{proof}
(1) Given a point $p \in T^w_{c_1} \cap T^w_{c_2}$, let $c_i^w = T^w_{c_i} \cap U(1)$, where $U(1) \In \C^*$ is the unit circle. Then we have a natural sequence of paths, 
$$ c_1 \xto{U(1)} c_1^w \xto{T^w_{c_1}} p  \xto{T^w_{c_2}}  c^w_2 \xto{U(1)} c_2, $$
where the label on the arrow indicate where the path is along,  the first and last arrow are given by the isotopy from $T_{c_i}$ to $T^w_{c_i}$, and the middle two arrow are the unique path along $T^w_{c_i} \cong \R$ between two points. This path from $c_1$ to $c_2$ in $\C^*$ defines a path on $U(1)$ by taking argument, hence defines a (taut) strand diagram in $\pcal(c_1, c_2)$. Clearly, different point $p$ corresponds to differnt diagram $D(p)$ (we may isotope $T^w_{c_2}$ to be the unwrapped versino $T_{c_2}$, then different point $p$ corresponds to different winding number, hence different $D(p)$). 

(2) By Lemma \ref{lm: obvious-disk-1}, we see
$p_{23} \circ p_{12}$ is non-zero and is a unique intersection point $p_{13} \in T^w_{c_1} \cap T^w_{c_3}$. It is also easy to check that composition are compatible, since the holomorphic disk provides the homotopy of two paths from $c_1$ to $c_3$. 
\end{proof}

Next, we consider Fukaya category of $(\C^*)^2$. Let $\ccal_1^2 = (S^1)^2$. Then for $c_1,c_2 \in \ccal_1^2$, we write $c_i=(c_{i,1},c_{i,2})$. If $c_1, c_2 \in \ccal_1^2$, then let
$$ \pcal_1^2(c_1, c_2) := \pcal_1(c_{1,1}, c_{2,1}) \times \pcal_1(c_{1,2}, c_{2,2}).$$
The composition of $\C \pcal_1^2(c_1, c_2)$ is given by concatenation of each strand, denoted as $concat(D_{23}, D_{12})$. 
\begin{lemma} \label{lm:baby-comparison-1.5}
   (1) For any $c_1, c_2 \in \ccal_1^2$, if $T^w_{c_1} > T^w_{c_2}$, then there is a natural injection
    $$ \Phi: T^w_{c_1} \cap T^w_{c_2} \into \pcal_1^2(c_1, c_2). $$
   (2) For any $c_1, c_2, c_3 \in \ccal_1^2$, if $T^w_{c_1} > T^w_{c_2} > T^w_{c_3}$, and $p_{12}, p_{23}$ are intersection points corresponding to diagrams $D_{12}, D_{23}$ where $p_{ij} \in T^w_{c_i} \cap T^w_{c_j}$, then 
    $$ \Phi: p_{23} \circ p_{12} \mapsto concat(D_{23}, D_{12}). $$
\end{lemma}
\begin{proof}
    This follows by taking product of Lemma \ref{lm:baby-comparison-1}. 
\end{proof}

Finally, we consider Fukaya category of $(\C^*)^2 \RM \Delta$. Let 
$$ \ccal_2 = (S^1)^2 \RM \Delta $$
and for all $c_1, c_2 \in \ccal_2$, let $\pcal_2(c_1, c_2)$ be the set of taut strands diagrams (modulo isotopy) from $c_1$ to $c_2$, and we write $c_1 \xto{D} c_2$ for a diagram $D \in \pcal_2(c_1, c_2)$. Let
$$ c_1 \xto{D_{12}} c_2 \xto{D_{23}} c_3 $$
we define
\begin{equation} \label{nil-cat}
    D_{23} \circ D_{12} = \begin{cases}
  concat(D_{23}, D_{12}) & \text{if the concatenation is taut} \\
  0 & \text{if the concatenation has a bigon}
\end{cases}
\end{equation} 

For $c \in \ccal_2$, let 
$$ T_c = c \times \R^2 \In (\C^*)^2 \RM \Delta. $$

\begin{lemma} \label{lm:baby comparison 2}
   (1) For $c_1, c_2 \in \ccal_2$, if $T^w_{c_1} > T^w_{c_2}$, then there is a natural injection
    $$ \Phi: T^w_{c_1} \cap T^w_{c_2} \into \pcal_2(c_1, c_2). $$
   (2) For $c_1, c_2, c_3 \in \ccal_2$, and  $T^w_{c_1} > T^w_{c_2} > T^w_{c_3}$, if $p_{12}, p_{23}$ are intersection points ($p_{ij} \in T^w_{c_i} \cap T^w_{c_j}$), corresponding to diagrams $D_{12}, D_{23}$,  then composition in Fukaya category is compatible with composition in $\ccal_2$, 
    $$ \Phi: p_{23} \circ p_{12} \mapsto D_{23} \circ D_{12}. $$
\end{lemma}
\begin{proof}
(1) We note that $\ccal_2 \into \ccal_1^2$, and if $c_1, c_2 \in \ccal_2$, then $\pcal_2(c_1, c_2) \cong \pcal_1^2(c_1, c_2)$ where we choose representative of taut strands by straight strands. 

If $T^w_{c_1} > T^w_{c_2}$, we may view them as in $(\C^*)^2$ then apply Lemma \ref{lm:baby-comparison-1.5} to get 
$$ T^w_{c_1} \cap T^w_{c_2} \into \pcal_1^2(c_1, c_2) \cong \pcal_2(c_1, c_2)$$

(2) Let $T^w_{c_1} \xto{p_{12}} T^w_{c_2} \xto{p_{23}} T^w_{c_3}$. By Lemma \ref{lm:baby-comparison-1.5}(2), there exists an intersection point $T^w_{c_1} \xto{p_{13}}  T^w_{c_3}$ and a disk $\D \into (\C^*)^2$ that realizes the Floer composition $p_{13} = p_{23} \circ p_{12}$ in $(\C^*)^2$. If the disk does not intersects diagonal $\Delta$, then in $(\C^*)^2 \RM \Delta$, we still have $p_{13} = p_{23} \circ p_{12}$ and if the disk intersects $\Delta$, then $p_{23} \circ p_{12} = 0$. The the desired result follow from the following Lemma.
\end{proof}

\begin{figure}[hh]
    \centering
    \begin{subfigure}[b]{\linewidth}
        \centering
        \includegraphics[width=0.7\linewidth]{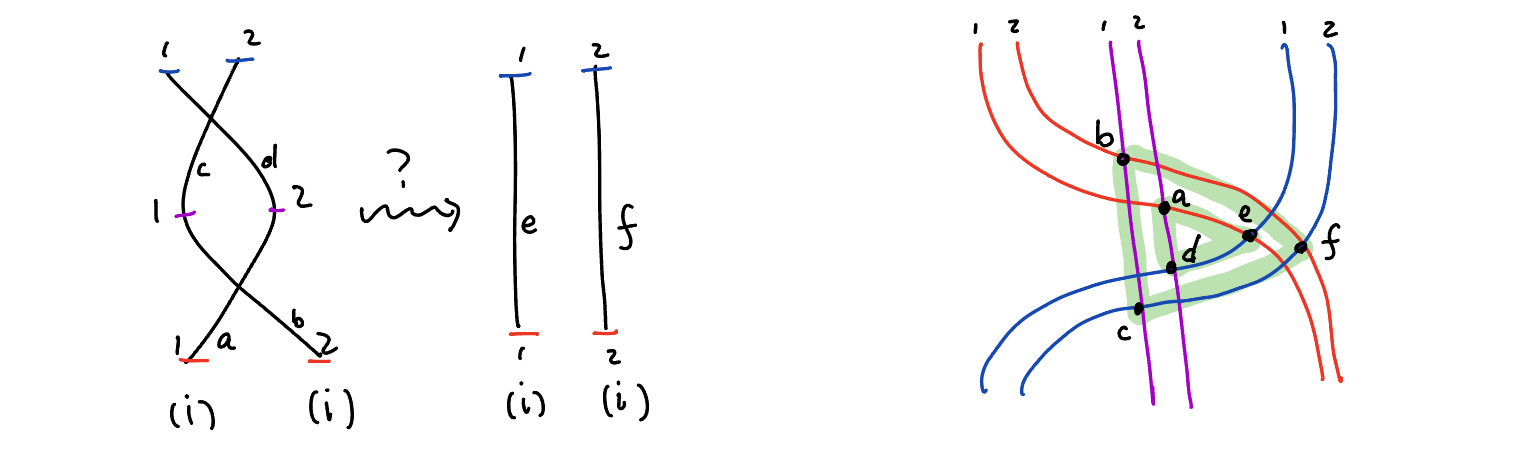}
        \caption{Bigon cancellation implies disk intersects with diagonal. }
        \label{fig:has-bigon}
    \end{subfigure} 
    \begin{subfigure}[b]{\linewidth}
    \centering
    \includegraphics[width=0.7\linewidth]{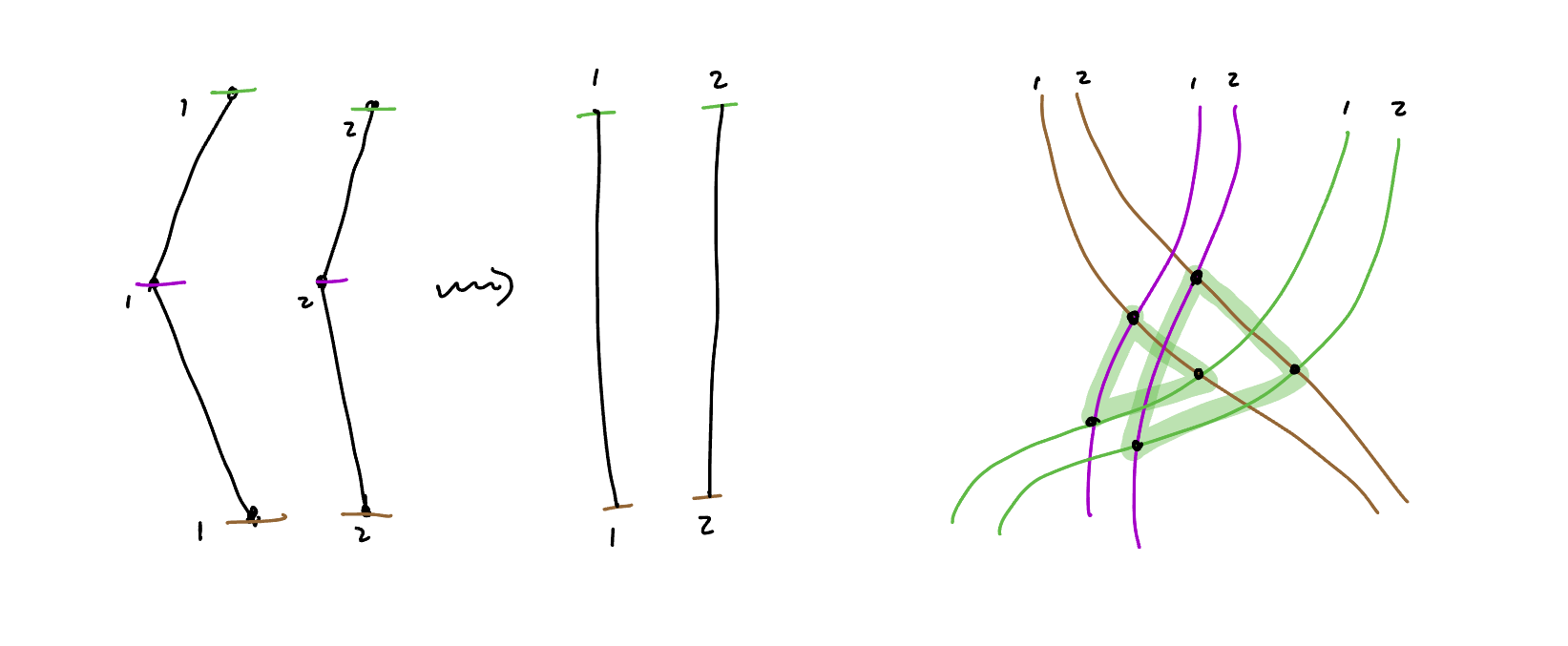}
    \caption{No bigon in diagram concatenation corresponds to disk avoiding diagonal. }
    \label{fig:no-bigon}
\end{subfigure}
    \caption{Caption}
    \label{fig:enter-label}
\end{figure}

\begin{lemma} \label{lm: bigon-and-diagonal}
Let $T^w_{c_1} > T^w_{c_2} > T^w_{c_3}$ in $(\C^*)^2 \RM \Delta$. Let $p_{12}, p_{23}, p_{13}$ be intersection points corresponding to diagrams $D_{12}, D_{23}, D_{13}$, and $\D \to (\C^*)^2$ realizes the composition $p_{13} = p_{23} \circ p_{12}$ in $(\C^*)^2$. 
$$
\begin{tikzcd}
 T^w_{c_1}\ar[r," p_{12}"]\ar[rr,out=-45,in=235,swap," p_{13}"] &  T^w_{c_2}\ar[r," p_{23}"] &  T^w_{c_3}
\end{tikzcd}
$$
Then 
 disk $\D \to (\C^*)^2$ intersects diagonal $\Delta$ if and only if the concatenation $concat(D_{23}, D_{12})$ has a bigon. 
\end{lemma}
\begin{proof}
We consider the universal cover $\C_Z$ of $\C^*_z$ and $(\C_Z)^2$ for $(\C^*_z)^2$. The Lagrangians $T^w_{c_i}$, intersection points $p_{ij}$ and disk $\D$ in $(\C^*_z)^2$ each admits $\Z^2$-many choices of lifts. We consider all compatible lifts of Lagrangians , intersections and disks
$$
\begin{tikzcd}
\wt T^w_{c_1}\ar[r,"\wt p_{12}"]\ar[rr,out=-45,in=235,swap,"\wt p_{13}"] & \wt T^w_{c_2}\ar[r,"\wt p_{23}"] & \wt T^w_{c_3}
\end{tikzcd}
$$
with a disk $\wt D \to (\C_Z)^2$ realizing the composition. We may also consider lifting of diagrams to $\R \times [0,1]$.

Let $\wt \Delta$ be the diagonal in $(\C_Z)^2$. 
We claim the following are equivalent. 
\begin{enumerate}[(1)]
 \item $\D \cap \Delta = \emptyset$. 
    \item For all such lift $\wt D$ we have $\wt D \cap \wt \Delta = \emptyset$
    \item For all lift of $concat(D_{23}, D_{12})$ to $\R \times [0,1]$, there is no bigon. 
    \item $concat(D_{23}, D_{12})$ has no bigon. 
\end{enumerate}

It is clear that $(1) \LRA (2)$ and $(3) \LRA (4)$, we only need to show $(2) \LRA (3)$. 

$(2) \Rightarrow (3)$. We only need to show that, if there is a lift of diagram with bigon, then the corresponding lifts of Lagrangians, intersection and disk satisfies $\wt D \cap \wt \Delta \neq\emptyset$. 
Consider  Figure \ref{fig:has-bigon}. We have map $f=(f_1,f_2): \wt D \into (\C_Z)^2$. 
We need to show that there exists $p \in \wt D$, such that $f_1(p) = f_2(p)$. Such $p$ exists by contraction mapping theorem.


$(3) \Rightarrow (2)$. We need to show that, if a lift of diagram with no bigon, then the corresponding lifts of Lagrangians, intersection and disk satisfies $\wt D \cap \wt \Delta  = \emptyset$. Consider  Figure \ref{fig:no-bigon}. We have map $f=(f_1,f_2): \wt D \into (\C_Z)^2$. 
We need to show that there exists no $p \in \wt D$, such that $f_1(p) = f_2(p)$. Indeed, as $p \in \d \D$, we have the winding number of $f_1(p) - f_2(p)$ being zero, hence $f_1(p) - f_2(p) \neq 0$ for all $p \in \D$ by argument principle. This finishes the proof of the Lemma. 
\end{proof}

Here is another useful Lemma with a similar proof.
\begin{lemma} \label{lm:bigon-with-red}
    Let $a \in \C^*$ with $\theta = \arg(a)$.  Let $c_1, c_2, c_3 \in \ccal_1 \RM \theta$, and $T^w_{c_1} > T^w_{c_2} > T^w_{c_3}$ be wrapping in $\C^* \RM a$. Let $p_{12}, p_{23}, p_{13}$ be intersection points of $T^w_{c_i}$ with diagram $D_{12}, D_{23}, D_{13}$, and  assume $\D \to \C^*$ realizes the composition of $p_{13} = p_{23} \circ p_{12}$ in $\C^*$. Then the image of $\D$ covers $a$ if and only if the concatenated diagram $concat(D_{23}, D_{12})$ forms a bigon with the 'red' line $\theta \times [0,1] \In S^1 \times [0,1]$. 
\end{lemma}
\begin{proof}
We may use the same argument as in Lemma \ref{lm: bigon-and-diagonal} to reduce the discussion to planar diagram in $\R \times [0,1]$ and Lagrangians in $\C$ (instead of $\C^*$).  If the composition of diagram has a bigon with the red line, then the corresponding Lagrangians will be as in Figure \ref{fig:bigon-with-red}, and the disk clearly covers the point $a$. Conversely, if there is no bigon with the red line, then the disk will not cover point $a$. 

    \begin{figure}[h]
        \centering
        \includegraphics[width=0.7 \linewidth]{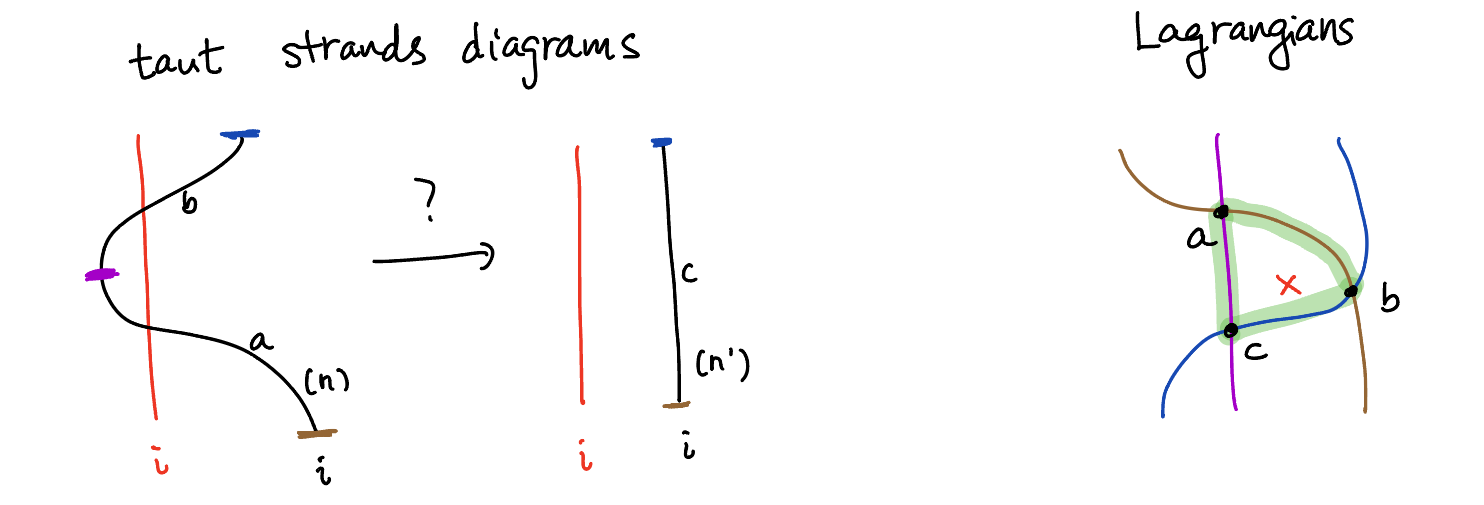}
        \caption{Bigon cancellation between a black and a red strand would involve a disk that pass through the red puncture, hence forbidden in the $\hbar=\eta=0$ setting.  }
        \label{fig:bigon-with-red}
    \end{figure}
    
\end{proof}

Finally, we consider Fukaya category of $\Sym^2(\C^*) / \Delta$. Let 
$$ \ccal_{[2]} = \Sym^2(S^1) \RM \Delta. $$
For $c_1, c_2 \in \ccal_{[2]}$, let $\pcal_{[2]}(c_1, c_2)$ be the set of taut diagrams. We define composition with the same formula as in Eq\ref{nil-cat}. For $c \in \ccal_{[2]}$, we may pick a lift $\wt c \in \ccal_2$, then we define
$$ T_c = T_{\wt c} / S_2. $$

\begin{lemma} \label{lm:baby comparison 2.5}
   (1) For $c_1, c_2 \in \ccal_{[2]}$, if $T^w_{c_1} > T^w_{c_2}$, then there is a natural injection
    $$ \Phi: T^w_{c_1} \cap T^w_{c_2} \into \pcal_{[2]}(c_1, c_2). $$
   (2) For $c_1, c_2, c_3 \in \ccal_2$, and  $T^w_{c_1} > T^w_{c_2} > T^w_{c_3}$, if $p_{12}, p_{23}$ are intersection points ($p_{ij} \in T^w_{c_i} \cap T^w_{c_j}$), corresponding to diagrams $D_{12}, D_{23}$,  then composition in Fukaya category is compatible with composition in $\ccal_{[2]}$, 
    $$ \Phi: p_{23} \circ p_{12} \mapsto D_{23} \circ D_{12}. $$
\end{lemma}
\begin{proof}
(1) Choose a lifting of $c_i$ to $\wt c_i \in \ccal_2$. Let $S_2 = \{e, \sigma\}$. Then we have natural isomorphism 
$$ \pcal_{[2]}(c_1, c_2) = \sqcup_{g \in S_2}  \pcal_2(\wt c_1, g \wt c_2) $$ 
and
$$ T^w_{c_1} \cap T^w_{c_2} = \sqcup_{g \in S_2}  T^w_{\wt c_1} \cap T^w_{g \wt c_2} $$
Hence the result follows from Lemma \ref{lm:baby comparison 2}. 

(2) Let $c_1 \xto{D_{12}} c_2 \xto{D_{23}} c_3$ be diagrams. Choose a lifting of $c_1$ to $\wt c_1 \in \ccal_2$. Then the diagrams determine a lifting of $c_2 \leadsto \wt c_2$ and  $c_3 \leadsto \wt c_3$. We have corresponding lift of $T^w_{c_i} \In \Sym^2(\C^*)/\Delta$ to $T^w_{\wt c_i} \In (\C^*)^2/\Delta$, and the disk $\D$ bounded by $T^w_{c_i}$, if exists, would lift to $\wt D$ bounded by $\wt T^w_{c_i}$. Hence the result of Lemma \ref{lm:baby comparison 2} implies the desired result here. 
\end{proof}

\subsection{Wrapping and Intersections}
We will consider the partial wrapping associated to the potential
$$ \wcal = \sum_{\bij} u_{\bij}: (\bfT^\vee \times \bfT_O) / \bfW \to \C.$$
A cofinal sequence of wrappings of $\tcal_c$ may be realized by 
a family which remain symmetric products (we will explain briefly why at the end of the section; alternatively, the reader could choose simply to view this as our definition of wrapping in this space).  

Because we have a superpotential $\sum_{\bij} u_{\bij}$, the wrapping in the $\C^*_u$ factors is partially stopped as in the model case $(\C^*_u, W=u)$. \footnote{We will actually choose the superpotential to be $W = e^{-i \epsilon} u$, for $0<\epsilon\ll 1$, so that the stop will be at $u = \infty e^{i \epsilon} \in \d_\infty \C^*$. We make such a perturbation so that the Lagrangian $\R_{+} \In \C^*_u$ will have its ideal boundary just before the stop with respect to the Reeb flow.} The wrapping in the $\C^*_y$ factor is unstopped. See Figure \ref{fig:wrap-T}.

Fix black points configurations $c_1,c_2$, and correspondingly consider $\T_{c_1}, \T_{c_2}$.  Let $\T_{c_1}^w$, $\T_{c_2}^w$ be certain positived wrapped versions of them.

\begin{enumerate}
\item We write $\T_{c_1}^w >  \T_{c_2}^w$ if for each node $(i) \in \Gamma$, and for each $j_1, j_2 \in \{1,\cdots, d_i\}$, we have $\T_{c_1, y_{(i,j_1)}}^w > \T_{c_2, y_{(i,j_2)}}^w$ and $\T_{c_1, u_{(i,j_1)}}^w > \T_{c_2, u_{(i,j_2)}}^w$. See Figure \ref{fig:wrap-T} for an illustration. 
    
    \item We write $\T_{c_1}^w >  \T_{c_2}^w > \cdots > \T_{c_n}^w$ if for any $i<j$ we have $\T^w_{c_i} > \T^w_{c_j}$. 
\end{enumerate}

\begin{figure}
    \centering
\begin{subfigure}[b]{0.4\linewidth}
    \centering
    \includegraphics[width=\linewidth]{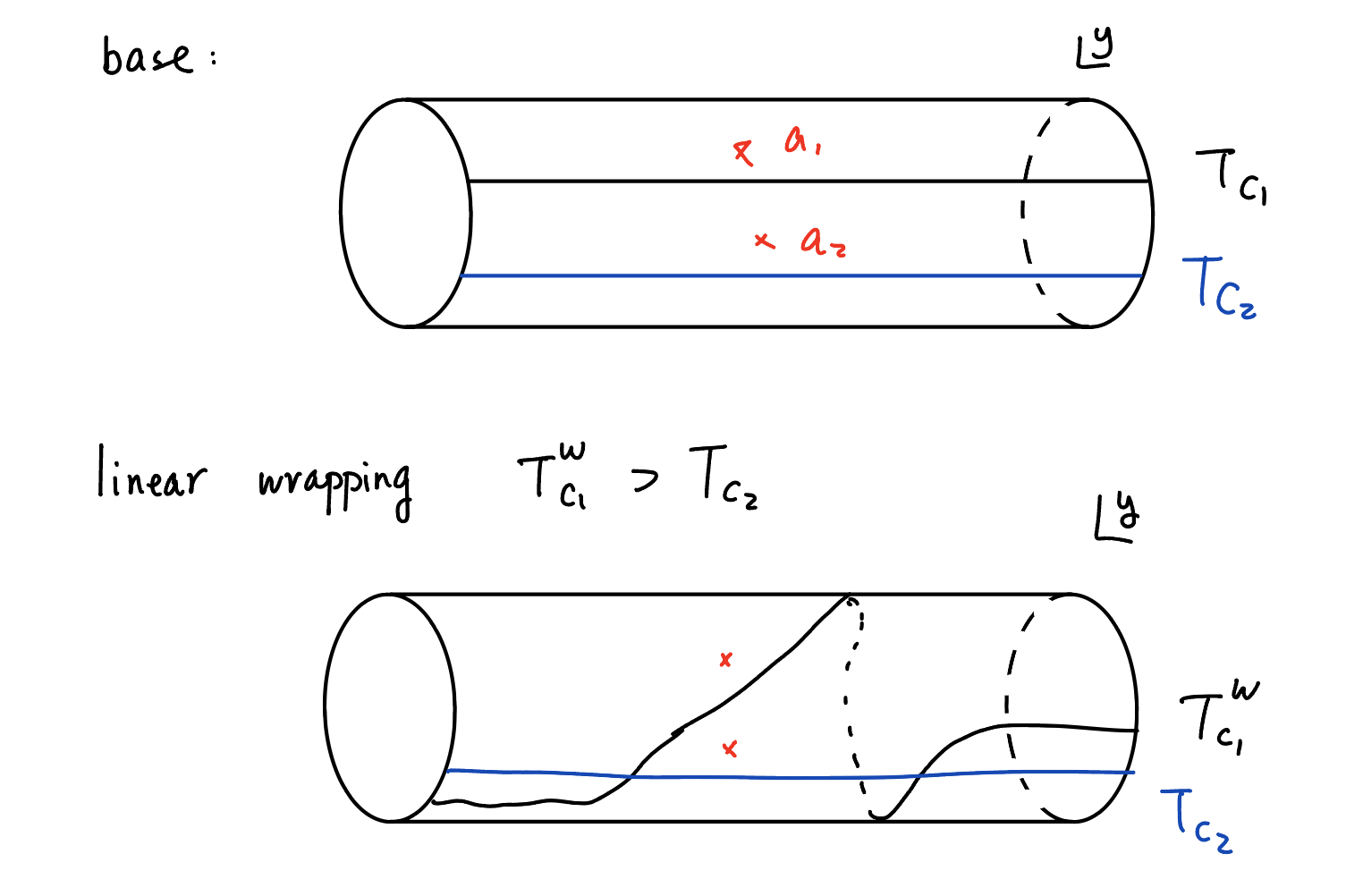}
    \caption{Base factor $\C^*_y$ wrapping}
\end{subfigure}
\begin{subfigure}[b]{0.4\linewidth}
    \centering
    \includegraphics[width=\linewidth]{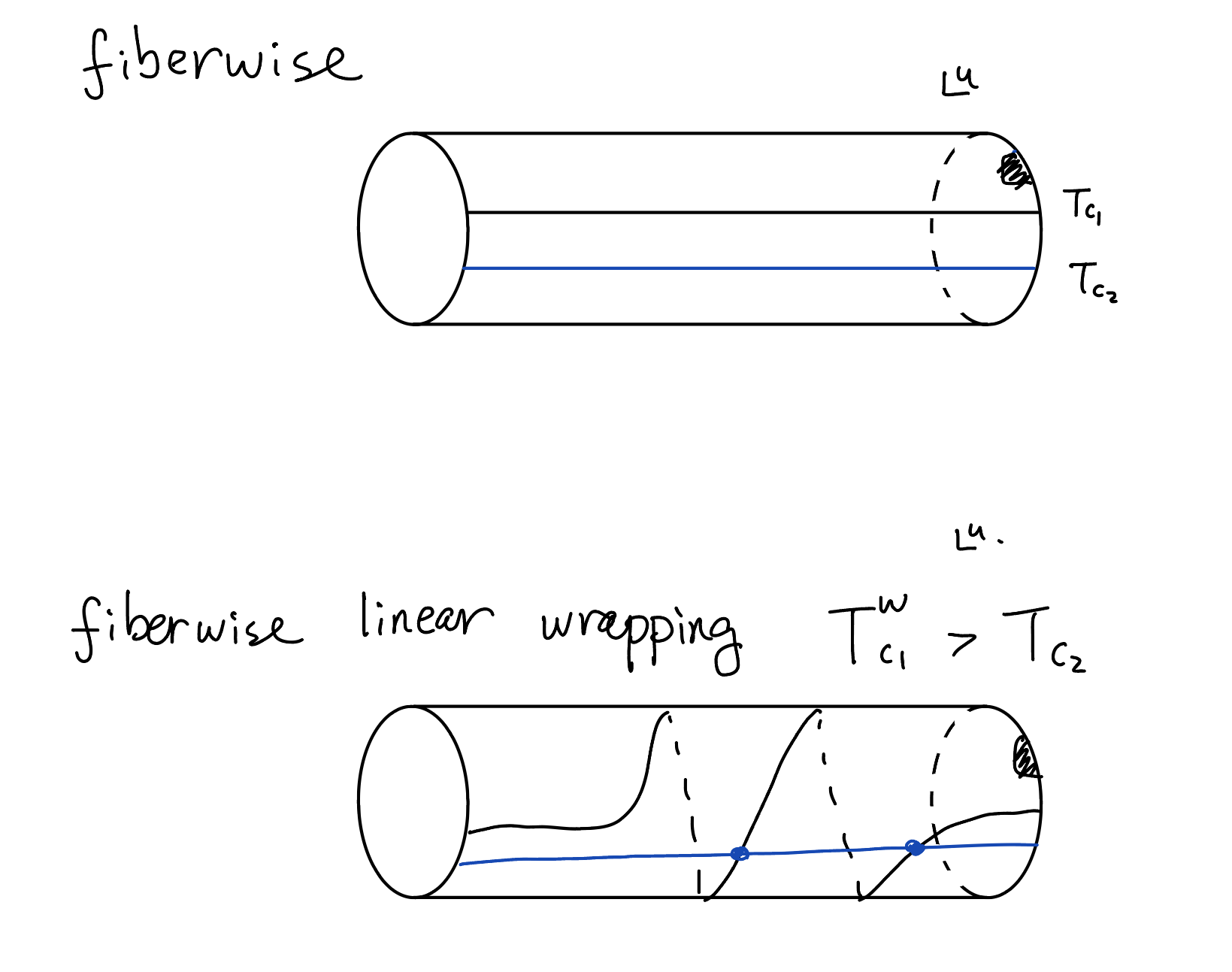}
    \caption{Fiberwise $\C^*_u$ wrapping}
\end{subfigure}
\caption{Wrapping of $\mathcal{T}$ in the base and in the fiber}
\label{fig:wrap-T}
\end{figure}

Next, we describe how to use weighted taut strands diagram to describe intersections of wrapped $\mathcal{T}_c$. This follows from the model cases in the previous subsection, we still give a complete description. 
Let $c_1, c_2$ be two black strands configurations, and consider the corresponding Lagrangians $\T_{c_i}$. Let $\T^w_{c_i}$ be their wrapped versions, such that $\T^w_{c_1} > \T^w_{c_2}$. Let $p \in \T^w_{c_1} \cap \T^w_{c_2}$ be an intersection point. We now describe how to assign a weighted taut strand diagram $D(p) \in \pcal(c_1, c_2)$ for the intersection point $p$.

We make a choice and label the points in $c_1$ and $c_2$ as $(c_1)_{\bij}$ and $(c_2)_\bij$, then we may write factors of the symmetric product Lagrangian as $\T_{c_k, y_\bij}$ etc.  The intersection point is
$$ p = \prod_{(i,j)} (p_{y,\bij}, p_{u, \bij}), \quad \text{where } p_{y,\bij} \in \wt \T_{c_1, y_{\bij}}^w \cap \wt \T_{c_2, y_{\sigma \bij}}^w, \text{ and } p_{u,\bij} \in \wt \T_{c_1, u_{\bij}}^w \cap \wt \T_{c_2, u_{\sigma \bij}}^w, $$
for some $\sigma \in \bfW$. 
The taut strands diagram lives in $U(1)_y \times [0,1]$, where the bottom $U(1)_y \times \{0\}$ is decorated by $c_1$ and the top is decorated by $c_2$. Take a black point $(c_1)_{\bij}$ in $c_1$, find the corresponding factor  $\T_{c_1,y_{\bij}}$. Write $(c_1)_\bij = \T_{c_1,y_{\bij}} \cap U(1)$,  $(c_1)^w_\bij = \T^w_{c_1,y_{\bij}} \cap U(1)$, and similar for $c_2$. Then we get a path in $\C^*_y$ as following
$$ (c_1)_\bij \leadsto (c_1)^w_\bij \leadsto p_{y,\bij} \leadsto (c_2)^w_{\bij} \to (c_2)_{\bij}, $$
where the first and last segments of path are along $U(1)_y$ and are induced by Lagrangian isotopy, and the second and third are along $\T^w_{c_1, y_\bij}$ and $\T^w_{c_2, y_{\sigma \bij}}$ respectively. 
Take the projection $\C^*_y \to U(1)_y$, we then get a path from $(c_1)_{\bij}$ to $(c_2)_{\sigma \bij}$ in $U(1)_y$ which determines a straight-line strand in $S^1 \times [0,1]$ from $(c_1)_\bij$ to $(c_2)_{\sigma \bij}$. Similar consideration in the $u$-factors gives a winding number (winding counter-clockwise around $u=0$ contributes to $+1$) which we label as weight on the strand. Repeat this construction for all $(i,j$ determines the weighted taut strands diagram $D(p)$.

Thus, we have shown that
\begin{proposition} \label{p:intersection as diagram}
If $\T_{c_1}^w >  \T_{c_2}^w$, then we have a natural injection
$$ \T_{c_1}^w \cap \T_{c_2}^w \into \pcal(c_1 ,  c_2); $$ 
Every element in $\pcal(c_1 ,  c_2)$ is  realized by some wrapping. 
\end{proposition}


\begin{figure}[hh]
    \centering
\begin{subfigure}[b]{\linewidth}
    \centering
    \includegraphics[width=0.7\linewidth]{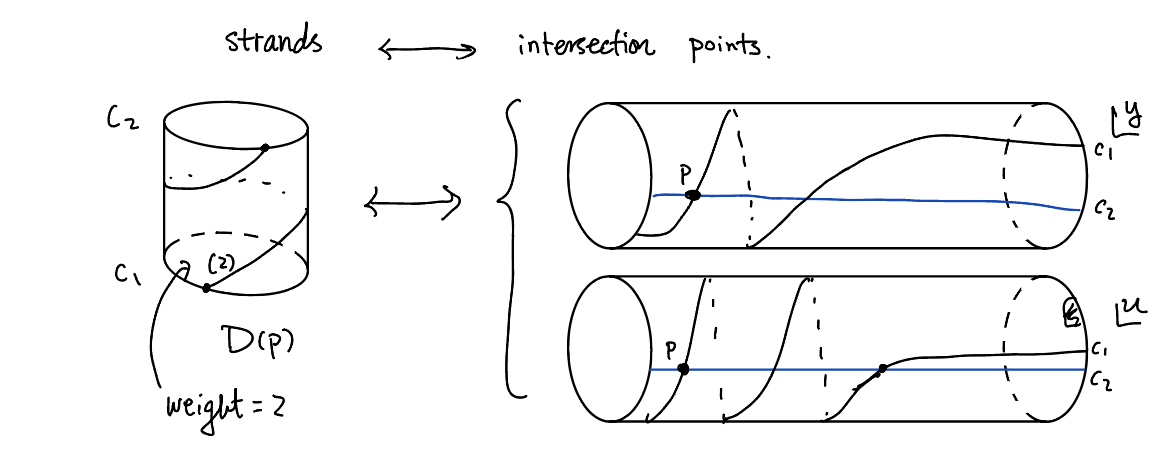}
     \caption{An example of one strands winding around in the base and in the fiber.}
\end{subfigure}
\begin{subfigure}[b]{\linewidth}
    \centering
    \includegraphics[width=0.7\linewidth]{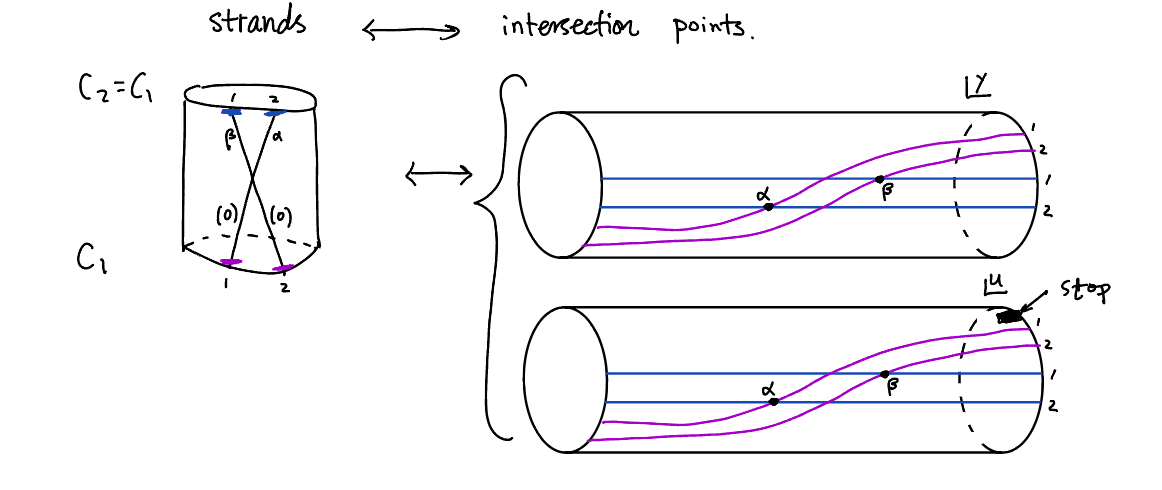}
    \caption{An example of two strands crossing}
\end{subfigure}
\caption{Two examples of intersection points as weighted taut strands diagram}
\label{fig:intersection-strands}
\end{figure}

\subsection{Composition}

Recall that a black points configuration on $S^1$, $c = \{\theta_\bij\}$, defines an element of the appropriate KLRW category, with the individual $\theta_\bij$ giving the positions of the black dots.  Recall also the `taut strands' basis of Def. \ref{taut strands diagram}, and the correspondence with intersection points from Prop \ref{p:intersection as diagram}. 

Consider the pull-back diagram
$$ \begin{tikzcd}
  \wt \ccal \ar[r] \ar[d]&  \prod_i (S^1)^{d_i} \RM \Delta  \ar[d] \\
  \ccal\ar[r] & (\prod_i (S^1)^{d_i} \RM \Delta)/\bfW
\end{tikzcd}
$$
For $c \in \ccal$, if $\wt c \in \wt \ccal$ is a lift of $c$, then we have a lift of $\T_c$ to 
$\wt \T_{\wt c} \In \bfT^\vee \times \bfT_O. $

We claim: 
\begin{proposition}  \label{disk calculation in maximally punctured case}
    Let $c_1, c_2, c_3 \in \ccal$, and let $\T_{c_i}$ be the corresponding Lagrangians. Choose any positive wrappings of $\T_{c_i}$ such that $\T^w_{c_1} > \T^w_{c_2} > \T^w_{c_3}$.  Choose any intersection points $p_{ij} \in \T^w_{c_i} \cap \T^w_{c_j}$ for $ij=12,23, 13$, and let $D_{ij}$ denote the corresponding diagram in $\pcal(c_i,c_j)$. Then $p_{13}$ appears in $p_{23} \circ p_{12}$ if and only if the following holds
    \begin{enumerate}[(1)]
        \item $D_{13}$ is the straightening out of $concat(D_{23}, D_{12})$ with  weights added on each strands.  
        \item There is no bigon cancellation between black strands of same label. 
        \item There is no bigon cancellation between black strands of adjacent labels $(i)$ and $(i')$, i.e.,  $(i) \to (i')$ has an arrow. 
        \item There is no bigon cancellation between black strands and red strands of the same label (corresponding to a framing edge $[i] \to (i)$). 
    \end{enumerate}
    (That is, bigon cancellations are allowed only for strands labelled by quiver nodes which are not connected by an edge.)
\end{proposition}
\begin{proof}
To determine compositions, we are counting 
disks $\D$ with three boundary marked points mapping to the target space with prescribed Lagrangian boundary condition.
We use the cylindrical model of Theorem \ref{cylindrical model}.  Note that because the Lagrangians are of symmetric product type, the boundary conditions may also be expressed in the cylindrical model.  

Since we deleted all root and matter divisors, the conditions from Theorem \ref{cylindrical model} on the map $\S \to \D \times \C^*_y \times \C^*_u$ simplify to the following: 
\begin{enumerate}[(i)]
    \item Each map $\pi_{i,z}: \S_i \to \D_z$ is a unramified cover of degree $d_i$, namely $\S_i$ is a disjoint union of $d_i$ disks, which we denote as $\S_i = \sqcup_{j=1}^{d_i} \S_{i,j}$. 
    \item Each map $(\pi_{i,z}, \pi_{i,y}): \S_i \to \D_z \times \C^*_y$ is an embedding.
    \item For each arrow $(i) \to (i')$, the embedding images of $\S_i$ and $S_j$ in $\D_z \times \C^*_y$ are disjoint.  Since otherwise the map $\D \to \bfT / \bfW$ would intersect with the loci where $y_{\bij} = y_{(i',j')}$. 
    \item For each arrow $[i] \to (i)$, the projection image $\S_i \to \C^*_y$ is disjoint from any $a_{\rij}$. 
\end{enumerate}

We only need to show that the four conditions (i) - (iv) implies condition (1) - (4) in the proposition. 

$(i) \LRA (1)$.  We first show that there exists a unique such disk $\S \to \D \times \C^*_y \times \C^*_u$ satisfing condition (i) above. Indeed, we may choose a lift of $\T^w_{c_k}$ to $\wt \T^w_{\wt c_k} \In \bfT^\vee \times \bfT$, for $k =1,2,3$, such that $p_{12}$ and $p_{23}$ lift as intersections of the lifted $\T^w_{c_k}$. For each index $(i,j)$, we have $\T^w_{c_1,y_\bij} > \T^w_{c_2,y_\bij}> \T^w_{c_3,y_\bij}$ in $\C^*_y$, and similarly for $\C^*_u$ factor. By Lemma \ref{lm: obvious-disk-1}, there exists a unique disk $\S_{i,j} \to \C^*_y$ realizing the composition of $p_{12,y_\bij}$ and $p_{23,y_\bij}$, similarly for $\C^*_u$ direction. Thus $D_{13}$ is necessarily is the concatenation of $D_{23}$ and $D_{12}$ with weights added on the corresponding strands. 

 $(ii) \LRA (2)$ and $(iii) \LRA (3)$ follows from Lemma \ref{lm: bigon-and-diagonal}. 

 $(iv) \LRA (4)$ follows from Lemma \ref{lm:bigon-with-red}. 
\end{proof}

\begin{theorem}\label{thm:stage1}

    There is a fully faithful functor 
\begin{eqnarray*}
    KLRW_Q|_{\hbar = \eta = 0} & \to & \Fuk((\mathbf{T}_O \times \mathbf{T}^\vee)/\mathbf{W}, \wcal) \\
    c & \mapsto & \tcal_c
\end{eqnarray*}
\end{theorem}
\begin{proof}
Fix any $c_1, c_2 \in \ccal.$
Recall we can choose wrappings where all the intersections have zero cohomological degree.  Using also Proposition \ref{p:intersection as diagram}, we have: 
$$ Hom_{Fuk}(\T_{c_1}, \T_{c_2}) = HW(\T_{c_1}, \T_{c_2}) = \colim_{\T^w_{c_1} > \T_{c_2}} CF(\T^w_{c_1}, \T_{c_2}) = \colim_{\T^w_{c_1} > \T_{c_2}} \C\pcal(\T^w_{c_1}, \T_{c_2})$$

Note that the above considerations alone does not imply that the maps in the colimit on the RHS are compatible with the natural inclusions into $\C\pcal(c_1, c_2)$.  However, said compatibility is established by Proposition \ref{disk calculation in maximally punctured case}.  Combining this with Proposition \ref{taut strands at zero}, we get a canonical vector space isomorphism: 
\begin{equation} \label{punctured functoriality via taut strands} Hom_{KLRW}(c_1, c_2)_{\hbar=\eta=0} \cong \C \pcal(c_1, c_2) \cong Hom_{Fuk}(\T_{c_1}, \T_{c_2}) \end{equation}

Also by Proposition \ref{disk calculation in maximally punctured case} and Lemma \ref{lm: KLRW composition hbar eta all zero}, this carries composition 
in the KLRW category to the composition in the Fukaya category.  
\end{proof}

\subsection{Comparison of Liouville structure}

Let us return briefly to clarify some issues about how our wrapping prescription above compares with the general setting of \cite{GPS1, GPS2}.  The space 
$(\mathbf{T}_O \times \mathbf{T}^\vee)/W$ acquires the structure of Weinstein symplectic manifold via a K\"ahler form arising from the structure of this space as an affine algebraic variety. 
We will want to compare this symplectic structure to both the symplectic structure descended from the product $\mathbf{T} \times \mathbf{T}^\vee$, and later to compare to $\MCB(\Gamma, \vec d)$ (of which $(\mathbf{T}_O \times \mathbf{T}^\vee)/W$ is an open subset by Theorem \ref{trivialization-intro}).  The basic point is the  following fact: 

\begin{lemma}
    Suppose $Z$ is a smooth affine algebraic variety, and $f: Z \to \C$ a function.   Let $\lambda_Z$ and $\lambda_{Z \setminus f^{-1}(0)}$ be K\"ahler Liouville structures on $Z$ and $Z \setminus f^{-1}(0)$.  

    Fix $\epsilon > 0$ sufficiently small.  Then there is a Liouville domain $\overline{Z} \subset Z$ completing to $Z$ and an isotopy of Liouville structures $\lambda_{Z \setminus f^{-1}(0)} \sim \widetilde{\lambda}_{Z \setminus f^{-1}(0)}$
    such that $\lambda_{Z} = \widetilde{\lambda}_{Z \setminus f^{-1}(0)}$ on $\overline{Z} \cap \{|f| \ge \epsilon\}$. 
\end{lemma}

\begin{corollary}
    Let $L_t \subset \bar{Z} \cap \{|f| \ge \epsilon\}$ be a family of Lagrangians with Legendrian boundary 
    $\partial L_t \subset \bar{Z} \setminus Z$.  Then $L_t$ satisfies the cofinality criterion of \cite{GPS2} with respect to $\lambda_{Z}$ iff it does with respect to $\widetilde{\lambda}_{Z \setminus f^{-1}(0)}$. 
\end{corollary}

\section{Gradings}\label{s:grading}

In general, if $L$ is a contractible Lagrangian in a Weinstein manifold $M$, and $\lambda \in H^1(M, \Z) \cong H^1(M, L; \Z)$ is an integral class, then given a path $\gamma$ in $M$ from $L$ to $L$, we can get a integer $\la \lambda, \gamma \ra$ using the pairing
$$ H_1(M, L; \Z) \times H^1(M, L;  \Z) \to \Z. $$
Then $\End(L)$ is a $\Z$-graded algebra.

If $f: M \to \C^*$ is a continuous function, then $(2\pi i)^{-1} d \log f \in H^1(M, \Z)$ give an integral class. We have two kinds of such functions
\begin{enumerate}
    \item For each circle node $(i)$, we have 
$$f_i(u,y)= \prod_{\bij} y_{\bij} $$ 
This is called the {\bf equivariant $q_i$-grading}. 
\item By Proposition \ref{p:111}, we have a monopole operator 
$$f_0 = r_{(\vec 1, \cdots, \vec 1)}$$ corresponding to the monopole operator for the diagonal embedding cocharacter $\C^* \to G =\prod_i GL(d_i)$. On $(\bfT_O \times \bfT^\vee)/\bfW$, we have an explicit expression
\begin{equation}
    \label{eq:f0} f_0(u,y) =  \prod_{\bij} u_{\bij}^{-1} \cdot \prod_{(i,j) \to (i',j')} (1 - y_{(i,j)} / y_{(i', j')}) \cdot \prod_{[i,k] \to (i,j)} (1 - a_{[i,k]} / y_\bij) \cdot \prod_{(i,j) \neq (i,j')} (1 - y_{(i,j)} / y_{(i,j')})^{-1} 
\end{equation} 
This is called the {\bf  $q$-grading}. 
\end{enumerate}


For each $c \in \ccal$, $\tcal_c \cong \R^{2d}$ is simply connected, hence $\End(\tcal_c)$ has well-defined $q$ and $q_i$ grading. 

For $c \neq c' \in \ccal$, the space $\Hom(\tcal_c, \tcal_{c'})$ is not naturally graded but instead a torsor for the grading space for either $c$ or $c'$.  One choice of trivialization of the torsor is: 
\begin{enumerate}
    \item a weight $w$ on any strand has $q$-grading $+w$.
    \item a crossing between strands of the same type has $q$-grading $-1$.
    \item if there is an arrow $[i] \to (i)$, then a strand of type $(i)$ crossing the red strand $[i]$ has $q$-grading $+1/2$. 
    \item if there is an arrow $(i) \to (j) $, then a strand of type $(i)$ crossing strand $(j)$ has $q$-grading $+1/2$. 
    \item all other crossings has $q$-grading $0$.
    \item Fix a representative in each connected component of $\ccal$ (there are finitely many). Fix a generic $\theta \in S^1$, away from any red points and any black points in the representatives. If a black strand of type $(i)$ crosses this line from right to left ($\arg$ is increasing), then $q_i$ grading is $+1$; if the crossing is from left to right, then $q_i$ grading is $-1$. 
\end{enumerate}
From Eq \eqref{eq:f0}, rule 1 comes from the factor $\prod_{\bij} u_\bij^{-1}$; rule 2 comes from $\prod_{(i,j) \neq (i,j')} (1 - y_{(i,j)} / y_{(i,j')})^{-1}$; rule 3 comes from factor $\prod_{[i,k] \to (i,j)} (1 - a_{[i,k]} / y_\bij)$; rule 4 comes from $\prod_{(i,j) \to (i',j')} (1 - y_{(i,j)} / y_{(i', j')}) $.


\section{Floer homologies with only root divisor removed}

Consider the following (less) punctured torus
$$ \bfT_o = \{ (y_\bij) \in \bfT \mid y_{(i,j)} \neq y_{(i,j')} \;\; \forall (i,j) \text{ and } (i,j')\}, $$
and the corresponding base change:
$$
\begin{tikzcd}
    \ycal_o \ar[d] \ar[r] & \MCB(\Gamma, \vec d) \ar[d] \\
    \bfT_o/\bfW \ar[r] & \bfT/\bfW
\end{tikzcd}
$$

Let $\wb H_{matter} = (\bfT_o - \bfT_O)/\bfW$, and $H_{matter}$ be the preimage of $\wb H_{matter}$ in $\ycal_o$. Let $\Fuk_\eta(\ycal_o, \wcal|_{\ycal_o}; H_{matter})$ be the Fukaya category defined over $\C[\eta]$, where we weight the count of a disk $\mathcal{D}$ by $\eta^{\mathcal{D} \cdot H_{matter}}$.

Our main result in this section is the following embedding. 
\begin{theorem} \label{thm:stage2}
    There is a fully faithful functor 
\begin{eqnarray*}
    KLRW_Q|_{\hbar = 0} & \to & \Fuk_{\eta} (\ycal_o, \wcal|_{\ycal_o}; H_{matter}) \\
    c & \mapsto & \tcal_c
\end{eqnarray*}
\end{theorem}

\subsection{Hom spaces and taut diagrams}

\begin{proposition} \label{p:stage2, taut composition}
Let $c_1,c_2, c_3 \in \ccal$, and consider wrappings $\T^w_{c_1} > \T^w_{c_2} > \T^w_{c_3}$. Let $D_{12}, D_{23}$ be weighted taut strands diagrams that represent intersection points (hence morphisms)
$$ \T^w_{c_1} \xto{D_{12}} \T^w_{c_2} \xto {D_{23}} \T^w_{c_3}.$$ 
If the concatenation $concat(D_{23} , D_{12})$ is taut, denoted as $D_{13}$ then $D_{13}$ represent an intersection point in $CF(\T^w_{c_1}, \T^w_{c_3})$ and the output of the Fukaya category composition is
$$ \mu_2(D_{23}, D_{12})= D_{13}. $$
\end{proposition}
\begin{proof}
From Prop \ref{disk calculation in maximally punctured case}, we see there is a disk $\D \to (\bfT^\vee \times \bfT_O)/\bfW$ that realizes the composition $\mu_2(D_{23}, D_{12})= D_{13}$. The argument in \ref{lm: obvious-disk-1} shows that such disk $\S \to \D \times \C^*_y$ is unique, hence this is the only disk contributing to $\mu_2(D_{23}, D_{12})$. 
\end{proof}

\begin{corollary} \label{stage2:continuation map}
Given $\T^{w}_{c_1} > \T^{w'}_{c_1} > \T^{w}_{c_2}$, and $e_{11'} \in \T^{w}_{c_1} \cap \T^{w'}_{c_1}$ corresponds to the identity diagram, and $p_{1'2} \in \T^{w'}_{c_1} \cap \T^{w}_{c_2}$ corresponds to some diagram $D_{12}$, then there is an intersection $p_{12} \in \T^{w}_{c_1} \cap \T^{w}_{c_2}$ corresponding to $D_{12}$, and
$$ \mu_2(p_{1'2}, e_{11'}) = p_{12}. $$
\end{corollary}
\begin{proof}
    The identity diagram composed with any taut diagram is still taut. 
\end{proof}

\begin{corollary} \label{c: stage2, Fuk has basis}
For any $c_1, c_2 \in \ccal$, we have canonical isomorphism
$$ HW(\T_{c_1}, \T_{c_2}) \cong \C\pcal(c_1, c_2). $$
\end{corollary}
\begin{proof}
Since for each wrapped versions $\T^w_{c_1} > \T_{c_2}$ we have natural basis (intersection points) in $CF(\T^w_{c_1}, \T_{c_2})$ identified with diagrams in $\pcal(c_1, c_2)$ (Proposition \ref{p:intersection as diagram}), and the continuation map identifies these basis (Corollary \ref{stage2:continuation map}), hence $HW(\T_{c_1}, \T_{c_2})$ has a natural basis given by diagrams in $\C\pcal(c_1, c_2)$. Further more, any diagrams in $\C\pcal(c_1, c_2)$ can be realized in wrapped Floer cohomology by sufficient wrapping, hence we have the desired isomorphism. 
\end{proof}

\subsection{KLRW relations from disk counts, I: matter divisor} \label{s: bigon relation}

Here we will count certain disks, which will ultimately give rise to the bigon-dot relations in the KLRW category. 
 
\subsubsection{Bigon of the same node}
$$
 \begin{tikzpicture}[very thick]
        \def\br{40}
        \draw (0,0) to [out=\br, in=-\br] node[pos=0, below ]{$(i)$} (0,2);
        \draw (0.5,0) to [out=180-\br, in=\br - 180] node[pos=0, below ]{$(i)$} (0.5,2);
        \node at (1,1) {$=0$};
        \end{tikzpicture}
        $$
This follows from Lemma \ref{lm: bigon-and-diagonal} and we removed root divisor, ie, diagonal divisor. 

\subsubsection{Bigon between red and black}

\begin{figure}[h]
        \centering
    \includegraphics[width=1\linewidth]{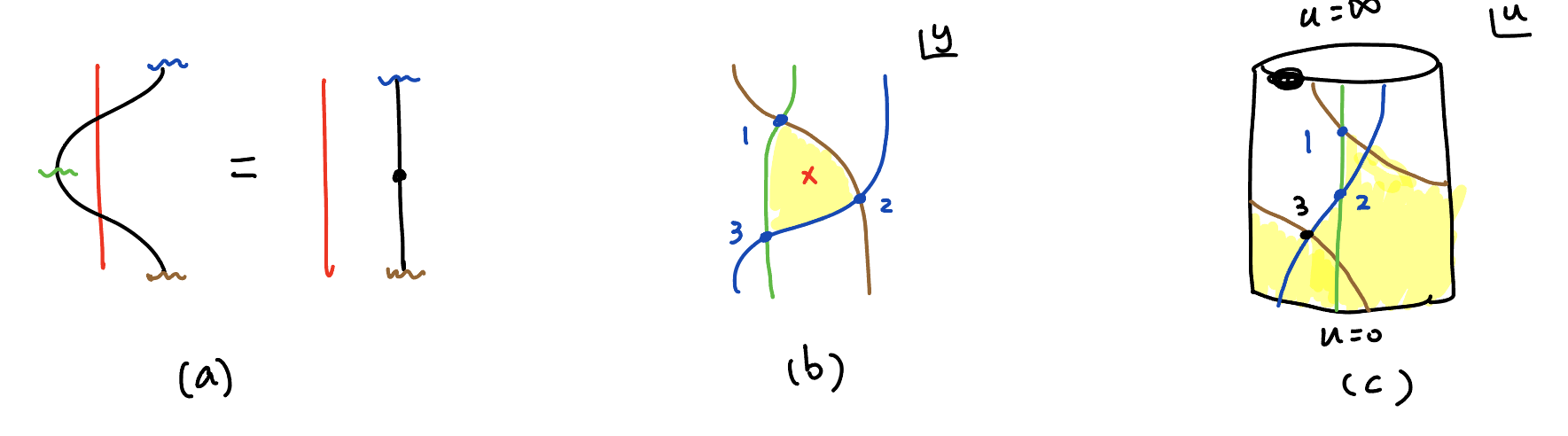}
   \caption{KLRW relation for bigon between a black strand and a red strand}
 \label{fig:bigon-red-black-proof}
\end{figure}

Consider three Lagrangans $\T_1 > \T_2 > \T_3$, drawn with color brown, green and blue in Figure \ref{fig:bigon-red-black-proof}. On the base $\C^*_y$ direction, we have $\T_1, \T_3$ are on the same side of the marked point, and $\T_2$ on the other side. On the fiber $\C^*_u$ direction, we pick the intersection points in $\T_1 \cap \T_2$ and $\T_2 \cap \T_3$ with winding number zero (marked by $1$ and $2$), and the intersection point in $\T_1 \cap \T_3$ with winding number $1$ (marked as $3$). 

Now we are ready to count disk. Consider a disk with three boundary marked points and one interior marked point, denoted as $\D_{3,1}$. Such a disk has a real 2 dimensional moduli space $\mcal_{\D_{3,1}}$.  

Recall from Theorem \ref{cylindrical model intro}, for the base direction map $\D \to \C^*_y$,   the interior marked point in $\D$  maps to the marked point $a \in \C^*_y$; for the fiber direction $\D \to \P^1_u$, the interior marked point maps to $u=0$. 

$$
\begin{tikzcd}
    & \mcal_{\D_{3,1} \to \C^*_y \times \P^1_u} (\dim_\R = 0)\ar[ld] \ar[rd] & \\
    \mcal_{\D_{3,1} \to \C^*_y} (\dim_\R = 0) \ar[rd] & & \mcal_{\D_{3,1} \to  \P^1_u} (\dim_\R = 2)\ar[ld] \\
      & \mcal_{\D_{3,1}}  (\dim_\R = 2)& 
\end{tikzcd}
$$ 

The base direction map has no moduli and is consist of one point. The fiber direction map has two degrees of freedom from the 'obtuse angle' at points $2$ and $3$. A more careful analysis shows that the map $\mcal_{\D_{3,1} \to  \P^1_u} \to \mcal_{\D_{3,1}}$ is a bijection. Hence the fiber product $\mcal_{\D_{3,1} \to \C^*_y \times \P^1_u}$ consist of a single point. 

\subsubsection{Bigon of black strands of adjacent nodes circle nodes}

\begin{figure}
    \centering
    \includegraphics[width=1\linewidth]{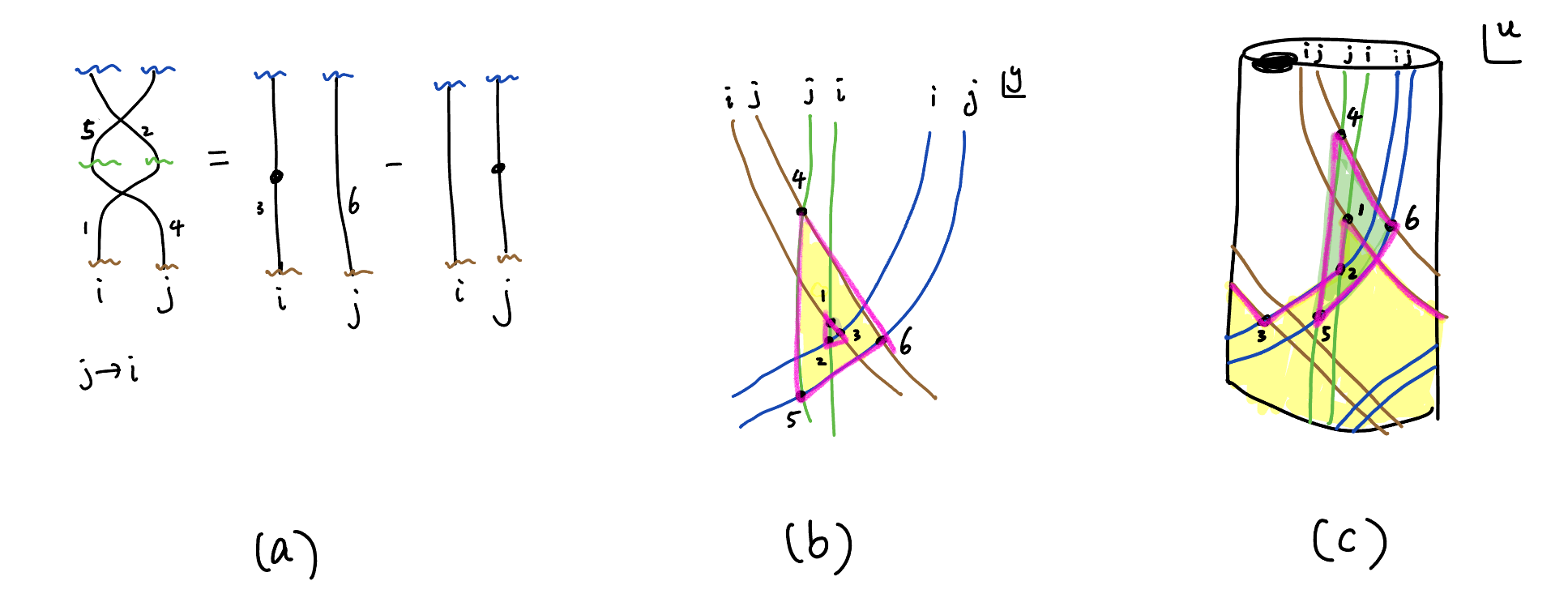}
   \caption{KLRW relation for bigon between a black strand and a red strand. Figure (c) shows the first composition output term in (a).}
         \label{fig:bigon-ij}
\end{figure}

Consider two adjacent nodes in the quiver $(j) \to (i)$.  We draw the Lagrangians base and fiber picture as in Figure \ref{fig:bigon-ij}. 

We show that the first output (dot on the (i) strand) can be realized, the second output can be similarly analysed which we omit. We are looking for two disks $\S_i, \S_j$ with three marked points (i.e. triangle), with maps $\S_i \sqcup \S_j \to \D \times \C^*_y \times \C_u$, such that 
\begin{enumerate}
    \item In $\C^*_y$, $\S_j$ (resp $\S_i$) maps to the triangle with vertex $1,2,3$ (resp $4,5,6$). 
    \item In $\C_u$, $\S_j$  maps to the triangle with vertex $1,2,3$, and $S_i$ maps into the region with corner $4,5,6$.  
    \item There is a $s_i \in \S_i$ and $s_j \in \S_j$ that maps to the same point $(z_0, y_0)$ in $\D \times \C^*_y$, and such that under $S_j \to \C_u$, we have $s_j \mapsto 0$. 
\end{enumerate}
We now show that there is a unique such disk. 
\begin{enumerate}
    \item Let $\mcal_{\S}$ be the moduli space of $\{(\S_i, s_i), (\S_j, s_j)\}$ of two disks, each with 3 boundary marked points and one interior marked points, $\mcal_\S \cong \mcal_{\D_{3,1}}^2$.
    \item Let $\mcal_{\S \to \D}$ be the moduli space of disk $\D$ with 3 boundary marked points and 1 interior marked point,  and isomorphism  $\S_i \to \D, \S_j \to \D$, hence $\mcal_{\S \to \D}$ is the diagonal in $\mcal_\S$. 
    \item Let $\mcal_{\S \to \C_y^*}$ be the moduli space of $\S_j$ (resp $\S_i$) maps to the triangle with vertex $1,2,3$ (resp $4,5,6$), and such that the marked points $s_i \in \S_i, s_j \in \S_j$ maps to the same point. Then the moduli space is determined by the image $y_0$ of $s_i$, hence $\mcal_{\S \to \C_y^*}$ is isomorphic to choice of $y_0$, namely the triangle with vertex $123$ in $\C^*_y$. 
    \item Let $\mcal_{\S \to \C_u}$ be the moduli space where $\S_j$ is mapped to the triangle $123$ in $\C_u$, and $\S_i$ is mapped to a degenerated 'triangle' where the three boundary marked points are mapped to $4,5,6$ and the interior point is mapped to $0 \in \C_u$ (see the green shaded region in Figure \ref{fig:bigon-ij} (c)). The obtuse corner $5,6$ gives two real degrees of freedom for $\S_j$ with marked points, and a more careful analysis shows that the forgetful map $\mcal_{\S \to \C_u} \to \mcal_\S$ is an isomorphism. 
\end{enumerate}  
By contraction mapping argument, the following fiber product $\mcal_{\S \to \D \times \C_y^* \times \C_u}$ is one point. 
$$
\begin{tikzcd}
    & \mcal_{\S \to \D \times \C_y^* \times \C_u} (0 \dim_\R) \ar[ld] \ar[d] \ar[rd] & \\
   \mcal_{\S \to \D} (2 \dim_\R)  \ar[rd] &  \mcal_{\S \to  \C_y^* }  (2 \dim_\R) \ar[d]  & \mcal_{\S \to \C_u} (4 \dim_\R) \ar[ld]  \\
    & \mcal_\S  (4 \dim_\R) & 
\end{tikzcd}
$$


 
 

         

         


\subsection{Proof of Theorem \ref{thm:stage2}}
From Proposition \ref{stage2:KRLW has basis} and Corollary \ref{c: stage2, Fuk has basis}, we see for any $c_1, c_2 \in \ccal$, there is a vector space isomorphism
$$ \Hom_{KLRW, \hbar=0}(c_1, c_2) \cong \Hom_{Fuk}(\T_{c_1}, \T_{c_2}) \cong \C\pcal(c_1,c_2). $$

Let $\circ_{KLRW}$ and $\circ_{Fuk}$ denote composition in the two categories, and let $\circ$ denote concatenation of diagram.  Then, for any weighted taut diagrams $c_1 \xto{D_{12}} c_2 \xto{D_{23}} c_3$, we need to verify
\begin{equation} \label{functoriality with matter}
\Phi(D_{23} \circ_{KLRW} D_{12}) \overset{?}{=} D_{23} \circ_{Fuk} D_{12}. \end{equation}

Proposition \ref{p:stage2, taut composition} asserts that \eqref{functoriality with matter} holds when $concat(D_{23}, D_{12})$ is itself taut (Prop \ref{p:stage2, taut composition}).

Let us show that if $D_{12}$ and $D_{23}$ are elementary weighted strands diagram, which can also be viewed as KLRW diagrams, then 
$$ \Phi (D_{23} \circ_{KLRW} D_{12}) = D_{23} \circ_{Fuk} D_{12}. $$
That is,  that \eqref{functoriality with matter} holds for compositions of elementary diagrams. 

Indeed, if $D_{12}$ or $D_{23}$ has all vertical strands, then the concatenation $concat(D_{23}, D_{12})$ is taut thus Prop \ref{p:stage2, taut composition} implies the result holds. 

If $concat(D_{23}, D_{12})$ is not taut, then we have exactly one bigon. If the bigon is of type $(i)-(i)$, $(i)-[i]$ or $(i)-(j)$ where $(i)$ and $(j)$ are adjacent in $\Gamma$, then 
section \ref{s: bigon relation} shows that the KLRW composition coincide with the Fukaya composition. If the bigon is of type $(i), (j)$ where $(i)$ $(j)$ are not adjacent, then KLRW and Fukaya composition (Lemma \ref{lm:baby-comparison-1.5}) will cancel the bigon.

Finally, by applying Prop \ref{p: KLRW h=0 recog} to the full subcategory $\{\T_c, c \in \ccal\}$, we finishes the proof of Theorem \ref{thm:stage2}.  $\qed$

\section{Floer homologies in the multiplicative Coulomb branch}
In this section we consider the Fukaya category of the original space $\MCB(\Gamma, \vec d)$ with no divisor removed. 

Two new difficulties are that (1) when $\hbar \ne 0$ there is no longer an isomorphism $\Hom_{KLRW}(c_1, c_2) \cong \C \pcal(c_1, c_2)$, and, in fact correspondingly, (2) the continuation maps determining wrapped Floer homology no longer preserve the diagram basis of Proposition \ref{p:intersection as diagram}.  That is, both isomorphisms of Equation \eqref{punctured functoriality via taut strands} fail.  

Instead we proceed by implementing the recognition principle of Proposition \ref{p: KLRW recog}.  

Recall we consider a Fukaya category where we count disks $\D \to \MCB(\Gamma, \vec d)$ with weight $\eta^{\D \cdot H_{matter}} \hbar^{\D \cdot H_{root}}$. Let $\acal$ be a full subcategory of the Fukaya category with objects $\{\T_c \mid c \in \ccal\}$. Our result is that $KLRW \cong \acal$, as categories defined over $\C[\eta, \hbar]$. 

\subsection{Crossing number filtration}

Let $c_1, c_2 \in \ccal$ be two black points configuration, for any diagram $D \in \pcal(c_1, c_2)$, define "crossing number of type (i)" as
$$ cross_i(D) = \text{number of crossings of black strands of type $(i)$},$$ 
and define the "total crossing number" as
$$ cross(D) = \sum_{i \in \Gamma} cross_i(D). $$

We may define a (multi-index) increasing filtration on $\C \pcal(c_1,c_2)$. 
\begin{definition}
For any $(i) \in \Gamma$, let $k_i \geq 0$ be an integer and let $\vec k = (k_i)$. Define an increasing filtration on $\C \pcal(c_1, c_2)$ called `crossing number filtration', 
$$ \fcal_{\vec k} \C \pcal(c_1, c_2) = \C \cdot \{D \in \pcal(c_1, c_2) \mid cross_i(D) \leq k_i , \forall i\}. $$
For any $k \in \Z_{\geq 0}$, we may define
$$ \fcal_{k} \C \pcal(c_1, c_2) = \C \cdot \{D \in \pcal(c_1, c_2) \mid cross(D) \leq k\} = \sum_{\vec k, |\vec k| = k} \fcal_{\vec k} \C \pcal(c_1, c_2). $$
\end{definition}
Let $Gr^\fcal_k$ denote the associated graded for the total crossing number filtration, and let $Gr^\fcal$ be the direct sum over $k$.

\begin{proposition} \label{p:resolve-crossing}
For any $c_1,c_2,c_3 \in \ccal$, and Lagrangians $\tcal^w_{c_1} >  \tcal^w_{c_2} > \tcal^w_{c_3}$, the composition 
$$ \mu_2:  CF(\tcal^w_{c_1}, \tcal^w_{c_2}) \otimes CF(\tcal^w_{c_2}, \tcal^w_{c_3})  \to CF(\tcal^w_{c_1}, \tcal^w_{c_3}) $$
preserves the crossing number filtration. 
Namely, if diagram $D_{13} \in CF(\tcal_1, \tcal_3)$ appears in $\mu_2(D_{12}, D_{23})$, then for any $i \in \Gamma$, 
$$ cross_i(D_{13}) \leq cross_i(D_{12}) + cross_i(D_{23}). $$
\end{proposition}

\begin{proof}
If $\varphi: \D \to \ycal$ is a holomorphic triangle interpolating $D_{12}, D_{23}, D_{13}$, then under the projection $\ycal \to \Sym^{d_i}(\C^*)$, we have $\D \to \Sym^{d_i}(\C^*)$. Let $\Delta$ be the big diagonal divisor in $\Sym^{d_i}(\C^*)$, then we have 
$$0 \leq \#(\D \cap \Delta) = cross_i(D_{12}) + cross_i(D_{23}) - cross_i(D_{13}), $$
hence the desired inequality. 
\end{proof}





\begin{corollary} \label{c:GrHW}
There is a natural total crossing number filtration on  $HW(\T_{c_1}, \T_{c_2})$, and
$$ Gr_k(HW(\T_{c_1}, \T_{c_2})) \cong \C[\eta, \hbar]\cdot \pcal_k(c_1, c_2), \forall k \in \Z_{\geq 0}. $$
\end{corollary}
\begin{proof}
Recall that 
$$ HW(\T_{c_1}, \T_{c_2}) = co\lim_{\T_{c_1}^w} CF(\T^w_{c_1}, \T_{c_2}) $$
where various $CF(\T^w_{c_1}, \T_{c_2})$ are related by continuation maps. The continuation map is realized by composition with an intersection that is labelled by the 'identity diagram' in $\pcal(c_1,c_2)$, hence is in $\fcal_0$. By Prop \ref{p:resolve-crossing}, continuation map does not increase the filtration level, i.e., respects filtration. Hence the colimit $HW(\T_{c_1}, \T_{c_2})$ is filtered. On the associated graded level, we have continuation map is identity on the basis of weighted taut strands diagrams, hence $Gr(HW(\T_{c_1}, \T_{c_2}))$ has a basis which are in bijection with $\pcal(c_1,c_2)$ (Prop \ref{p:intersection as diagram}). 
\end{proof}

Since $\acal$ is a filtered category, i.e., morphism space is filtered and composition respects filtration, we may consider its assoicated graded category $Gr(\acal)$ where we replace the hom space by the assoicated graded. 
\begin{corollary} \label{stage 2: graded level}
We have equivalence of categories
$$  Gr(KLRW) \cong Gr(\acal). $$
\end{corollary}
\begin{proof}
By Corollary \ref{c:GrHW} and Lemma \ref{lm:GrKLRW}, we have
$$ Gr(HW(\T_{c_1}, \T_{c_2})) \cong \C[\eta, \hbar]\cdot \pcal(c_1, c_2) \cong Gr(Hom_{KLRW}(c_1, c_2). $$
On the associated graded level, the composition in $Gr(\acal)$ uses disks that avoid the root divisors, hence by Theorem \ref{thm:stage2}, we have the isomorphism of hom space is compatible with composition, thus we have the desired equivalence of categories. 
\end{proof}


\subsection{Continuation maps and Robust diagrams}

Given two black points configurations $c_1, c_2$, the continuation map for computing wrapped Floer homology may not be 
the identity matrix in the basis given $ \pcal(c_1, c_2)$. 
Correspondingly, the wrapped Floer cohomology $HW(\T_{c_1}, \T_{c_2})$ may not be canonically  isomorphic to $\cp(c_1, c_2)$.

\begin{definition}
Let $c_1, c_2 \in \ccal$ be two black points configurations and let $D \in \pcal(c_1, c_2)$ be any weighted taut strands diagrams. We say $D$ is a {\bf robust diagram} (under the continuation map) if for any morphism $D_{12} \in CF(\T^w_{c_1}, \T^w_{c_2})$ labelled by $D$, and any continuation maps $e_{1}: \T^{w'}_{c_1} \to \T^w_{c_1}$ and $e_2: \T^w_{c_2} \to \T^{w'}_{c_2}$, the compositions $D_{12} \circ e_1$ and $e_2 \circ D_{12}$ consist of a single output intersection labelled by diagram $D$. 
\end{definition}

\begin{lemma} \label{l:lowest-filtrant}
    Let $D \in \cp(c_1, c_2)$ be a weighted taut strand diagram. If there exists no diagram $D' \in \cp(c_1, c_2)$, such that the following two conditions are satisfied: (1) $D'$ and $D$ has the same $q_i$-grading and $q$-grading, and (2) there exists $i \in \Gamma$ with $cross_i(D') < cross_i(D)$, then $D$ is a robust diagram.
\end{lemma}
\begin{proof}
    Consider the direct sum decomposition of $\cp(c_1,c_2)$ given by the equivariant and $q$ gradings, and consider the crossing number filtration within each direct summand. By assumption, $D$ is in the lowest filtrand of one of the summand, hence by Proposition \ref{p:resolve-crossing}, is preserved by the continuation map. 
\end{proof}

\begin{proposition}
    Elementary weighted taut diagrams are robust diagrams. 
\end{proposition}
\begin{proof}
Let $D$ be an elementary weighted strand diagram. 
    If $D$ is vertical, then its crossing numbers are zero, then by Lemma \ref{l:lowest-filtrant}, it is a robust diagram. 
    
    If $D$ has one crossing, then by $q$-degree and crossing number filtration, $D$ is robust. 
\end{proof}

Since robust diagrams are preserved by continuation maps, they give well-defined element in the wrapped Floer cohomology, thus it makes sense to compose robust diagrams.

\subsection{KLRW relations from disk counts II: root divisor} \label{s: klrw from disk 2}
In this subsection we prove the following KLRW relation, where each KLRW diagram on the left is viewed as a composition of two elementary diagrams, which in turn are viewed as compositions in Fukaya categories. 
$$ \begin{tikzpicture}[very thick]
        \def\br{40}
        \begin{scope}
        \draw (0,0) -- node[pos=0, below ]{$(i)$}  (1,2);
        \draw (1,0) -- node[pos=0, below ]{$(i)$} node[pos=0.8, inner sep=2pt, circle, fill]{}  (0,2);
        \end{scope}
        \node at (1.5,1) {$-$};
        \begin{scope}[shift=({2,0})]
        \draw (0,0) -- node[pos=0, below ]{$(i)$}  (1,2);
        \draw (1,0) -- node[pos=0, below ]{$(i)$}  (0,2);
         \draw (1,0) -- node[pos=0, below ]{$(i)$} node[pos=0.2, inner sep=2pt, circle, fill]{}  (0,2);
        \end{scope}
        \node at (3.5,1) {$= \;\; \hbar $};
         \begin{scope}[shift=({4.3,0})]
        \draw (0,0) -- node[pos=0, below ]{$(i)$}  (0,2);
        \draw (1,0) -- node[pos=0, below ]{$(i)$}  (1,2);
        \end{scope}
        \end{tikzpicture}
    $$
It can be seen from the cylindrical model that the question reduces to the case where $\T$ is a symmetric product of two Lagrangians of the same label. Let $\T_1 > \T_2 > \T_3$ be wrapped versions of $\T$.  
Let $s$ denote the crossing, and $t_1, t_2$ denote dot on the first or second strand. 
See Figure \ref{fig:ts} for choice of intersections for $t_1 s$, and Figure \ref{fig:st} for that of $s t_2$. Here we draw $\P^1_u$'s open part $\C^*_u$ with $u=\infty$ to the right.  

Both $t_1 s$ and $s t_2$ will have as output the following weighted taut strand diagram$$
\begin{tikzpicture}[very thick]
     \draw (0,0) --   (1,2);
        \draw (1,0) -- node[pos=0.2,right]{$(1)$}  (0,2);
\end{tikzpicture}
$$which that cancels. The only possible output in with the same q-grading and strictly less in filtration is the identity diagram, and the corresponding disk will cross the diagonal once since the crossing number dropped by one. Hence in the cylindrical model,
the covering curve $\S \to \D$ is a branched double cover, and the projection $\S \to \P^1_u$ will cover $u=\infty $ once, sending the branch point $q \in S$ of $\S \to \D$ to $\infty$.

Let us write $\mcal_{\S \to \D \times \C_y \times \P^1_u}(p_1, p_2;q)$
for the moduli space of holomorphic maps $\Phi: \S \to \D \times \C_y \times \P^1_u$ with the given Lagrangian boundary conditions, such that the branching point of $\S \to \D$ coincide with pole of $\Phi:\S \to \P^1_u$.
Per the cylindrical model, the count of points 
in this space is the contribution of $q$ to 
$\mu_2(p_1, p_2)$.  Thus our remaining task is to show
that

$$ \# \mcal_{\S \to \D \times \C_y \times \P^1_u}(s, t_1; id) -  \# \mcal_{\S \to \D \times \C_y \times \P^1_u}(t_2, s; id) = 1. $$

We will describe these moduli spaces as intersections
in other simpler spaces.  
$$
\begin{tikzcd}
    & \mcal_{\S \to \D \times \C^*_y \times \P^1_u} (\dim_\R = 0)\ar[ld] \ar[rd] & \\
    \mcal_{\S \to \D \times \C^*_y} (\dim_\R = 1) \ar[rd] \ar[d] & & \mcal_{\S \to \D \times  \P^1_u} (\dim_\R = 1)\ar[ld] \ar[d] \\
   \mcal_{\S \to \C^*_y} (\dim_\R=4) \ar[rd]   & \mcal_{\S \to \D}  (\dim_\R = 2) \ar[d] & \mcal_{\S \to \P^1_u} (\dim_\R=4) \ar[ld] \\
      & \mcal_{\S} (\dim_\R=5)  & 
\end{tikzcd}
$$

\begin{itemize}
\item $\mcal_\S$ this is the moduli space of a disk with $6$ boundary marked points and one interior marked point. $\dim_\R \mcal_\S = 2+6-3 = 5$.

\item $\mcal_{\S \to \D}$ is the compactified moduli space of the branched cover of disk $\S$ with 6 marked points to disk $\D$ with 3 marked points, where the marked point of $\S$ is the branching point. 
The moduli is real $2$-dimensional, parameterized by the branching loci in $\D$. The boundary corresponds to the limit that the branching loci moves to  $\d\D$. It can either move to a boundary marked point, or a boundary segment.  See Figure \ref{fig:MSD}. 

\item $\mcal_{\S \to \C^*_y}$ is 4 dimensional, where the holomorphic map ($2$ degrees of freedom) determines a disk $\S$ with 6 boundary marked points, and the choice of the interior marked point is $2$  degrees of freedom.

\item $\mcal_{\S \to \P^1_u}$ is 4 dimensional, where the holomorphic map gives $4$ degrees of freedom, and the requirement that the marked point goes to $\infty$ determines the marked point in $\S$. 

\item $\mcal_{\S \to \D \times \C^*_y}$ and $\mcal_{\S \to \D \times \P^1_u}$ are 1-dimensional. 
\end{itemize}

Since the map $\S \to \C^*_y$ has no condition on the interior marked point of $\S$, and the double cover $\S \to \D$ determines the interior marked point to be the branching point, hence we may consider a simplified fiber product to get the map 
$\mcal_{\S \to \D \times \C^*_y} \to \mcal_{\S \to \D}$.
$$
\begin{tikzcd}
    \mcal_{\S \to \D \times \C^*_y} \ar[r] \ar[d] & \mcal_{\S \to \D} \ar[d] \\
    \mcal_{\S_{(6,0)} \to \C^*_y} \ar[r] & \mcal_{\S_{(6,0)}}
\end{tikzcd}
$$
We now explain the notations.

Let $\mcal_{\S_{6,0}}$ denote the (compactified) moduli space of disks with $6$ (labelled) boundary marked points and $0$ interior marked points. Then $\mcal_{\S_{6,0}}$ is the $6-3=3$ (real) dimensional associahedron. See Figure \ref{fig:associahedron}. There are two types of (real) codimension 1 degeneration of $\S_{(6,0)}$, corresponds to split the $6$ (ordered) boundary marked points as $6=2+4$ (6 ways) or $6=3+3$ (3 ways), they correspond to the 6 pentagon faces and 3 square faces in the associahedron. 

\begin{figure}[h]
    \centering
    \includegraphics[width=0.4\linewidth]{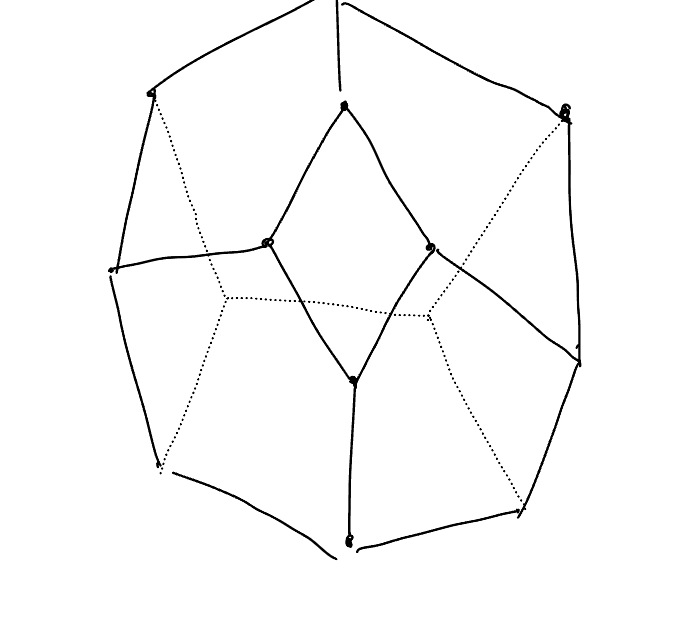}
    \caption{The moduli space $\mcal_{\S_{6,0}}$ is the 3 $\dim_\R$ associahedron.}
    \label{fig:associahedron}
\end{figure}

Let $\mcal_{\S_{6,0} \to \C^*_y}$ be the moduli space of holomorphic maps satisfies the Lagrangian boundary conditions and the marked points goes to Lagrangian intersection.  See Figure \ref{fig:lag-in-cy} for the Lagrangian arrangement and choice of intersections (this is for $\hbar \cdot id \in t_1 \circ s$). 

Let $\mcal_{\S_{6,0} \to \D_z}$ be the moduli space of disk $\D$ with three boundary marked point and a holomorphic map $\S \to \D$ that is a double cover and the boundary marked points of $\S$ maps to boundary marked points in $\D$. Note that this is the same as $\mcal_{\S \to \D}$ before, since the freedom of interior marked point in $\S$ is cancelled by the condition that it is the branching point.

We have studied $\mcal_{\S \to \D}$, (Figure \ref{fig:MSD}) and its embedding image in $\mcal_{\S_{(6,0)}}$ (see Figure \ref{fig: embed to ms } (left), where we can determine the boundary of the embedding by matching the disk degeneration $\mcal_{\S \to \D}$. 

\begin{figure}[h]
    \centering
    \includegraphics[width=0.5\linewidth]{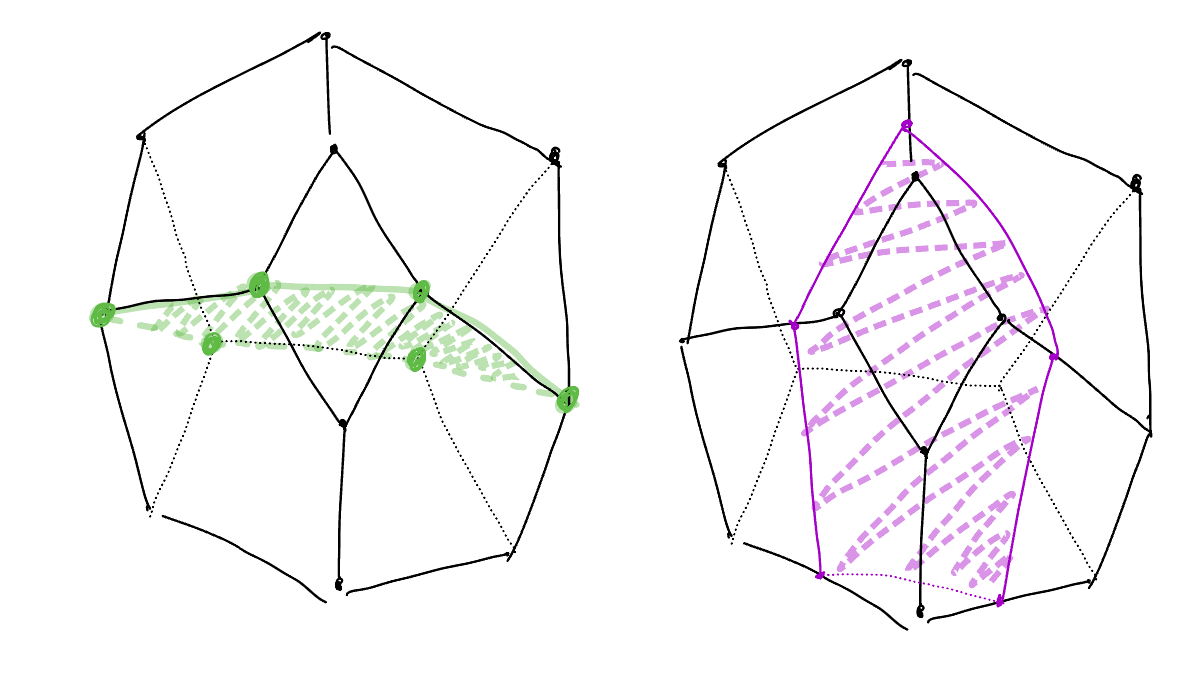}
    \caption{The images of $\mcal_{\S_{(6,0)} \to \D}$ (left, green) and $ \mcal_{\S_{(6,0)} \to \C^*_y}$ (right purple) in $\mcal_{\S_{(6,0)}}$}
    \label{fig: embed to ms }
\end{figure}

\begin{figure}[h]
    \centering
    \includegraphics[width=0.8\linewidth]{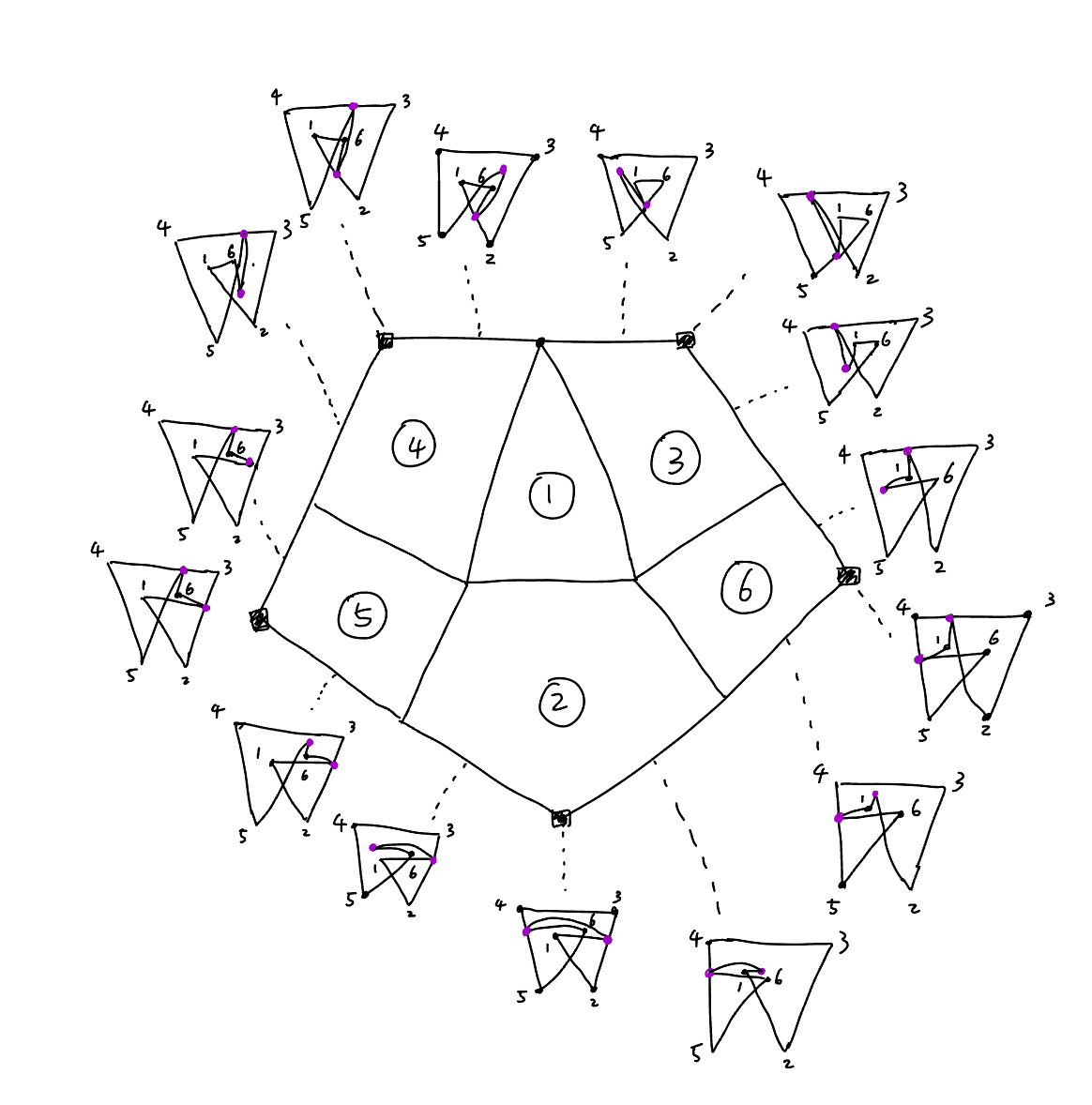}
    \caption{The moduli space $ \mcal_{\S_{(6,0)} \to \C^*_y}$, and we label the boundary of the moduli space by the corresponding disk degeneration type. }
    \label{fig:ms60tocy}
\end{figure}

Now we study $\mcal_{\S_{(6,0)} \to \C^*_y}$ in details. The moduli space is real 2 dimensional, and is the union of 6 regions. 
See Figure \ref{fig:ms60tocy}. We give a brief description of the regions
\begin{enumerate}
    \item Region (1): the image of the disk's boundary is the obvious one (shown in purple in Figure \ref{fig:ts}(b)). The map $\S \to \C^*_y$ has a ramification loci $y_0$, and the choice of $y_0$ is the small triangle with vertices $1,6$. That is the triangle in the moduli space. 
    \item Region (2): as the ramification loci $y_0$ in region (1) approaches the edge $(1)(6)$. The disk's image can develope `boundary double cusp', namely consider the Lagrangian line that contains point $(1), (6)$ (the second red line in Figure \ref{fig:lag-in-cy}). Parameterize the line, so that the intersection with edge (45) is at $x=0$, the intersection with edge (23) is at $x=1$, and the position of point (1) and (6) are $x_1$ and $x_6$. Let $a,b$ denote the end-points of the double-cusp. The boundary $\S$'s image will traverse the points in the order of $(6), a, b, (1)$. Hence $a,b$  satisfies
    $$ a < b, \quad 0 < a < x_6, \quad x_1 < b < 1. $$
    We see the corresponding moduli for $a,b$ is a pentagon, hence the shape of region (2). 
    \item Region (3), (4): as the ramification point $y_0$ reaches the edge (12) (resp edge (56)), the corresponding edge developes a double-cusp. 
    \item Region (5), (6): as the ramification point $y_0$ reaches the corner (6) (resp, corner (1)), the disk boundary can degenerate (see Figure \ref{fig:ms60tocy} near region (5) for the disk boundary degeneration). 
\end{enumerate}
After we analyze the moduli space $\mcal_{\S_{(6,0)} \to \C^*_y}$, we may consider the forgetful map to $\mcal_{\S_{(6,0)}}$. 
By examining the disk degeneration on the edges of of $\mcal_{\S_{(6,0)} \to \C^*_y}$, we see the $6$ marked points all splits as $2+4$ into two disks, hence the boundary of $\mcal_{\S_{(6,0)} \to \C^*_y}$ goes to the pentagon components of the boundary $\mcal_{\S_{(6,0)}}$. (See Figure \ref{fig: embed to ms }, right). Up to symmetry of the associahedron, this is the unique possiblity of image type.

Inside the hexagon $\mcal_{\S \to \D}$, $$ \mcal_{\S \to \D \times \C_y \times \P^1_u} (p_1, p_2;q)= \mcal_{\S \to \D \times \C_y} (p_1, p_2;q)\cap \mcal_{\S \to \D \times \P^1_u}(p_1, p_2; q).$$

For generic choice of base and fiber Lagrangian positions of $\T_1$ and $\T_3$, we will see that the RHS moduli
spaces do not meet along $\d \mcal_{\S \to \D}$. 
Thus, there is a well defined intersection number
on the RHS giving the contribution of the LHS.  
This intersection number is in turn determined by
how the RHS moduli spaces 
intersect the boundary of the hexagon in Figure \ref{fig:MSD}.

\begin{figure}
    \centering
\begin{tikzpicture}
\foreach \x in {0,1,...,6}{
\coordinate (\x) at (\x*60:2); 
\coordinate (m\x) at (\x*60+30:2); 
\fill (\x) circle (2pt);
}

\draw (0) {\foreach \x in {1,...,6} { -- (\x)}};

\node [right] at (m0) {$(25) = (61b)(2a5b)(a34)$};
\node [left] at (m2) {$(14) = (56b)(1a4b)(a23)$};
\node [below] at (m4) {$(36) = (12b)(3a6b)(a45)$};
\node [above] at (m1) {$(234q)(q561)$};
\node [left] at (m3) {$(456q)(q123)$};
\node [right] at (m5) {$(345q)(q612)$};

\begin{scope}[shift=({0,-3.5})]
\draw (0,0) circle (0.5) (1,0) circle (0.5) (-1,0) circle (0.5);
\node [fill, inner sep=2pt, label=-100:$3$] at (-100:0.5) {};
\node [fill, inner sep=2pt, label=80:$6$] at (80:0.5) {};
\node [fill, inner sep=2pt, label=00:$a$] at (0:0.5) {};
\node [fill, inner sep=2pt, label=180:$b$] at (180:0.5) {};

\node [fill, inner sep=2pt, label=-70:$4$, shift=({1,0})] at (-70:0.5) {};
\node [fill, inner sep=2pt, label=30:$5$, shift=({1,0})] at (30:0.5) {};

\node [fill, inner sep=2pt, label=110:$1$, shift=({-1,0})] at (110:0.5) {};
\node [fill, inner sep=2pt, label=210:$2$, shift=({-1,0})] at (210:0.5) {};

\node [fill, inner sep=1pt, red] at (0,0) {};

\end{scope}

\begin{scope}[shift=({-5,-2.2})]
\draw (0,0) circle (0.5) (1,0) circle (0.5);
\node [fill, inner sep=2pt, label=90:$1$] at (90:0.5) {};
\node [fill, inner sep=2pt, label=180:$2$] at (180:0.5) {};
\node [fill, inner sep=2pt, label=270:$3$] at (270:0.5) {};
\node [fill, inner sep=2pt, label=0:$q$,red] at (0:0.5) {};
\node [fill, inner sep=2pt, label=-90:$4$, shift=({1,0})] at (-90:0.5) {};
\node [fill, inner sep=2pt, label=0:$5$, shift=({1,0})] at (0:0.5) {};
\node [fill, inner sep=2pt, label=90:$6$, shift=({1,0})] at (90:0.5) {};
\end{scope}
 
\end{tikzpicture}

    \caption{Moduli space $\mcal_{\S \to \D}$.}
    \label{fig:MSD}
\end{figure}

\begin{figure}
     \centering
     \includegraphics[width=\linewidth]{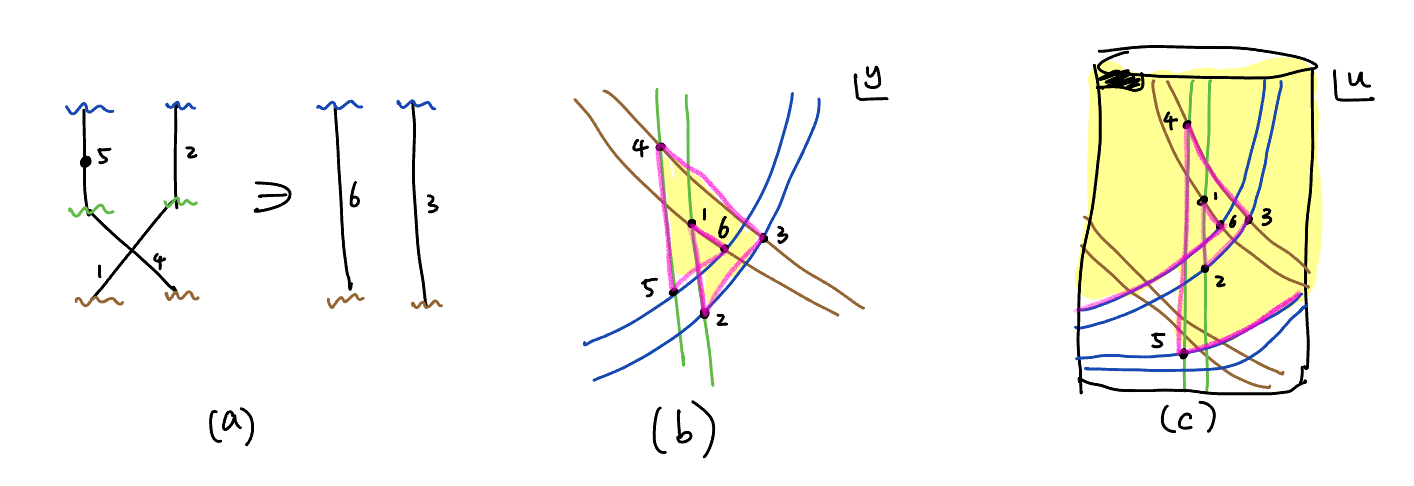}
                \caption{$\C_y$ and $\C^*_u$ projections
        of a cylindrical model 
        presentation a disk 
        contributing $\hbar \cdot id$ to $t_1 \circ s$. }
        \label{fig:ts}
\end{figure}

\begin{figure}
     \centering
\includegraphics[width=\linewidth]{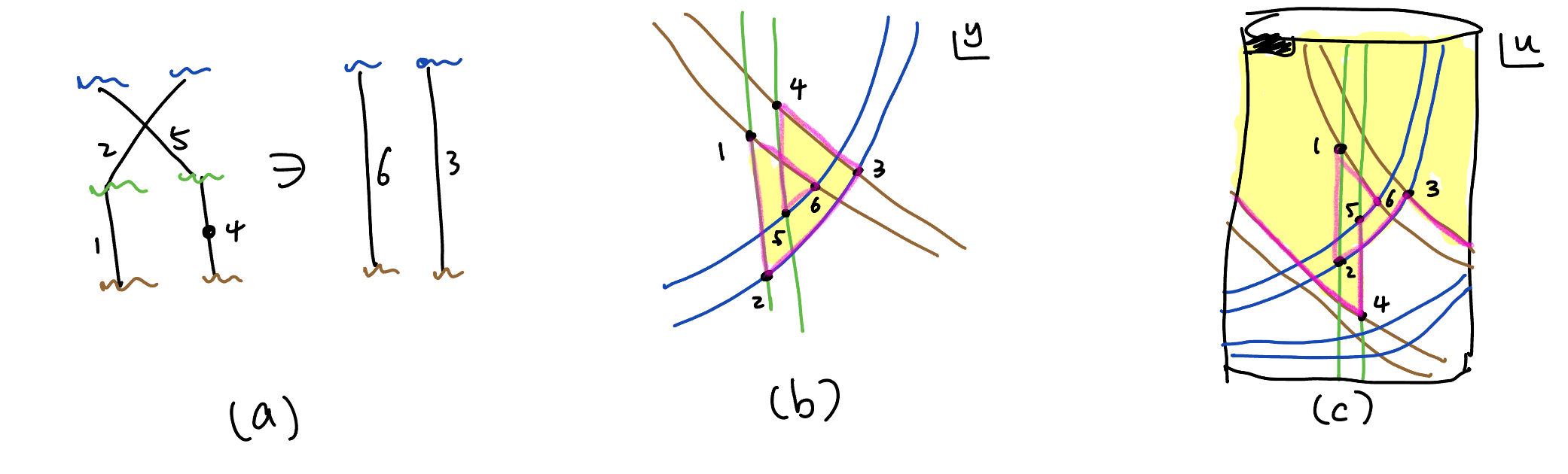}
        \caption{$\C_y$ and $\C^*_u$ projections
        of a cylindrical model 
        presentation of a disk 
        contributing $\hbar \cdot id$ to $s \circ t_2$. 
        }
        \label{fig:st}
\end{figure}

 
\begin{figure}
     \centering
\includegraphics[width=\linewidth]{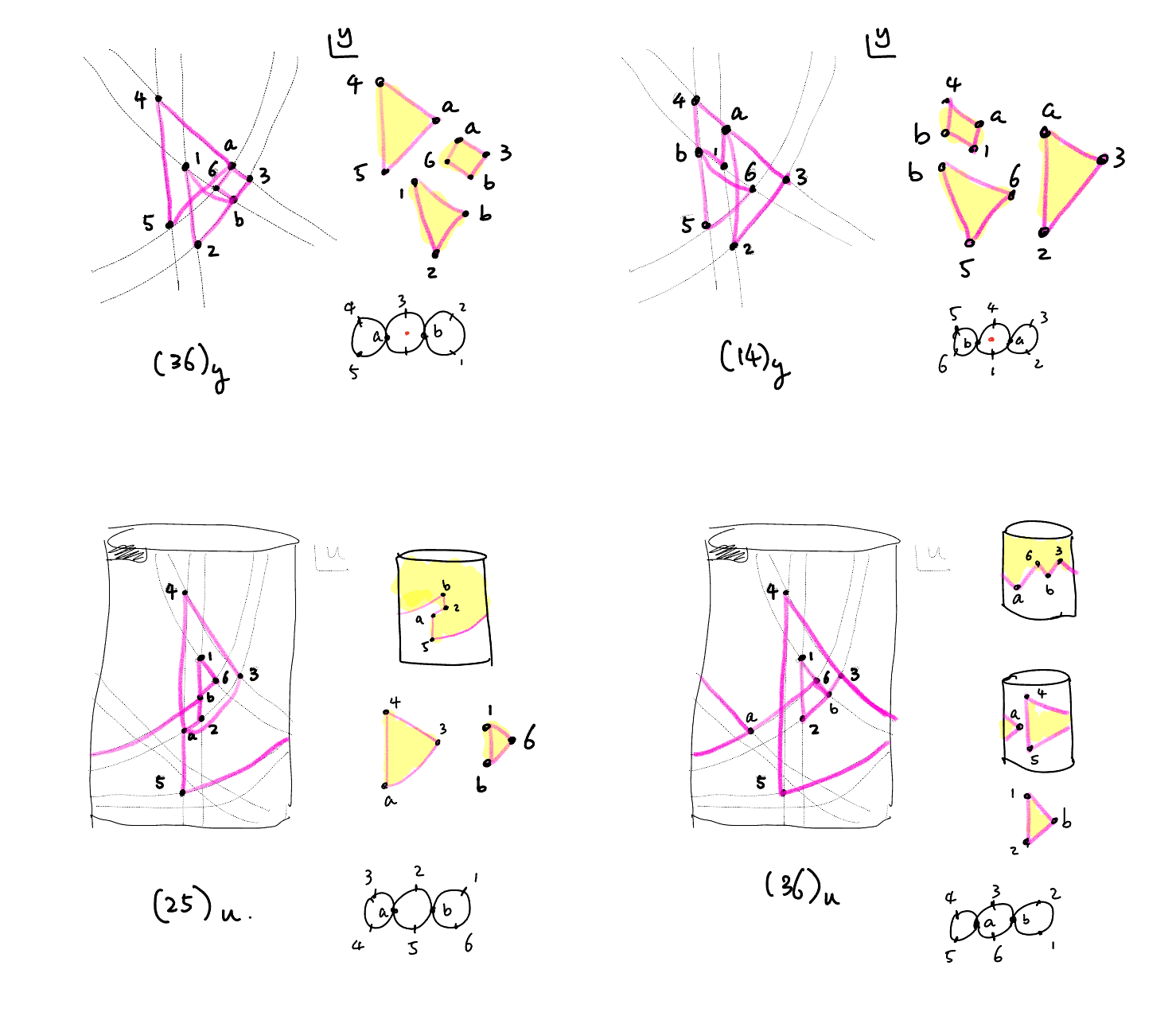}
        \caption{Degeneration in the base and fiber for $t_1 s$ }
        \label{fig:ts-degen}
\end{figure}

\begin{figure}
     \centering
   \includegraphics[width=1\linewidth]{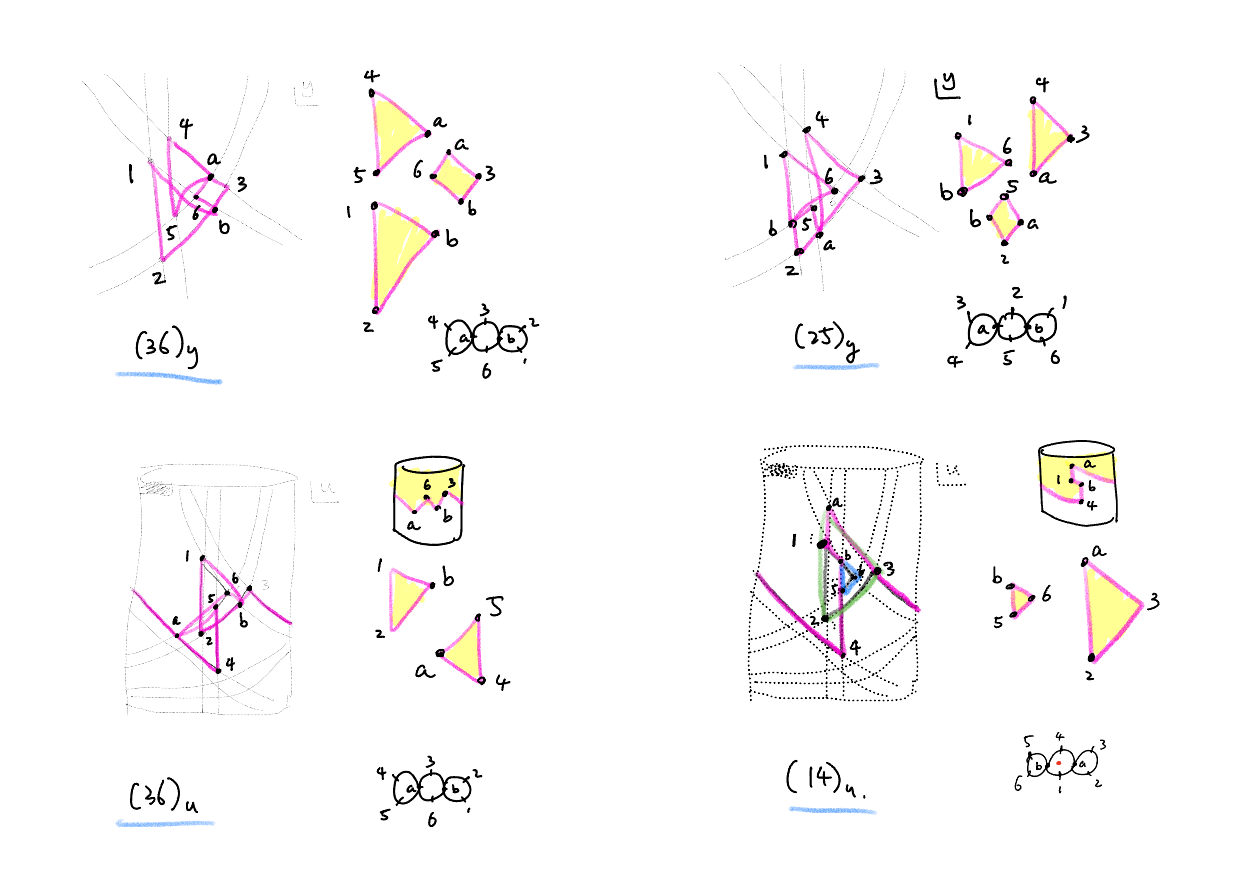}
       \caption{Degeneration in the base and fiber for $s t_2$ }
       \label{fig:st-degen}
\end{figure}

\begin{figure}[hhhh]
    \centering
 \begin{subfigure}[b]{0.4\textwidth}
  \centering
  \begin{tikzpicture}
\foreach \x in {0,1,...,6}{
\coordinate (\x) at (\x*60:2); 
\coordinate (m\x) at (\x*60+30:2); 
\fill (\x) circle (2pt);
}

\draw (0) {\foreach \x in {1,...,6} { -- (\x)}};


\node [fill, inner sep=2pt, label=-90:$(36)_y$, red] (ts-36y) at ($(4)!0.2!(5)$) {};
\node [fill, inner sep=2pt, label=180:$(14)_y$, red ] (ts-14y) at ($(2)!0.3!(3)$) {};
\node [fill, inner sep=2pt, label=-90:$(36)_u$, red ] (ts-36u) at ($(4)!0.7!(5)$) {};
\node [fill, inner sep=2pt, label=0:$(25)_u$, red] (ts-25u) at ($(0)!0.5!(1)$) {};

\draw[red] (ts-36y) to[out=60, in=-10]  node[pos=1, right] {$(t_1 s)_y$}  (ts-14y);
\draw[red] (ts-36u) to[out=100, in=200]  node[pos=1, left] {$(t_1 s)_u$} (ts-25u);

\node [fill, inner sep=1pt, blue] (st-36y) at (ts-36y) {};
\node [fill, inner sep=2pt, label=0:$(25)_y$,  blue ] (st-25y) at ($(0)!0.3!(1)$) {};
\node [fill, inner sep=1pt, blue ] (st-36u) at (ts-36u) {};
\node [fill, inner sep=2pt, label=180:$(14)_u$, blue] (st-14u) at ($(2)!0.7!(3)$) {};

\draw[blue] (st-36y) to[out=60, in=180]  node[pos=0.7, right ] {$(s t_2)_y$}  (st-25y);
\draw[blue] (st-36u) to[out=100, in=0]  node[pos=0.7, left ] {$(s t_2)_u$} (st-14u);
\end{tikzpicture}

        \caption{$id \notin t_1 s, id \in s t_2 $}
 \end{subfigure}
  \begin{subfigure}[b]{0.4\textwidth}
  \centering
    \begin{tikzpicture}
\foreach \x in {0,1,...,6}{
\coordinate (\x) at (\x*60:2); 
\coordinate (m\x) at (\x*60+30:2); 
\fill (\x) circle (2pt);
}

\draw (0) {\foreach \x in {1,...,6} { -- (\x)}};


\node [fill, inner sep=2pt,label=-90:$(36)_y$, red] (ts-36y) at ($(4)!0.8!(5)$) {};
\node [fill, inner sep=2pt,label=180:$(14)_y$, red ] (ts-14y) at ($(2)!0.3!(3)$) {};
\node [fill, inner sep=2pt,label=-90:$(36)_u$, red ] (ts-36u) at ($(4)!0.2!(5)$) {};
\node [fill, inner sep=2pt, label=0:$(25)_u$, red] (ts-25u) at ($(0)!0.5!(1)$) {};

\draw[red] (ts-36y) to[out=70, in=-10]  node[pos=1, right] {$(t_1 s)_y$}  (ts-14y);
\draw[red] (ts-36u) to[out=110, in=200]  node[pos=1, left] {$(t_1 s)_u$} (ts-25u);

\node [fill, inner sep=1pt, blue] (st-36y) at (ts-36y) {};
\node [fill, inner sep=2pt, label=0:$(25)_y$,  blue ] (st-25y) at ($(0)!0.3!(1)$) {};
\node [fill, inner sep=1pt, blue ] (st-36u) at (ts-36u) {};
\node [fill, inner sep=2pt, label=180:$(14)_u$, blue] (st-14u) at ($(2)!0.7!(3)$) {};

\draw[blue] (st-36y) to[out=100, in=180]  node[pos=0.7, right ] {$(s t_2)_y$}  (st-25y);
\draw[blue] (st-36u) to[out=50, in=0]  node[pos=0.7, left ] {$(s t_2)_u$} (st-14u);
\end{tikzpicture}
        \caption{$id \in t_1 s, id \notin s t_2 $}
 \end{subfigure}
    \caption{Intersections of $[\mcal_{\S \to \D \times \C_y} \cap \mcal_{\S \to \D \times \P^1_u}](p_1, p_2; q) $. Red lines for $(p_1, p_2; q) = (s, t_1, id)$, blue lines for $(p_1, p_2; q) = (t_2, s, id)$. Depending on the perturbation choices of $(36)_y$ and $(36)_u$, only one (pure color) intersection is non-empty. }
    \label{fig:2-cases}
\end{figure}
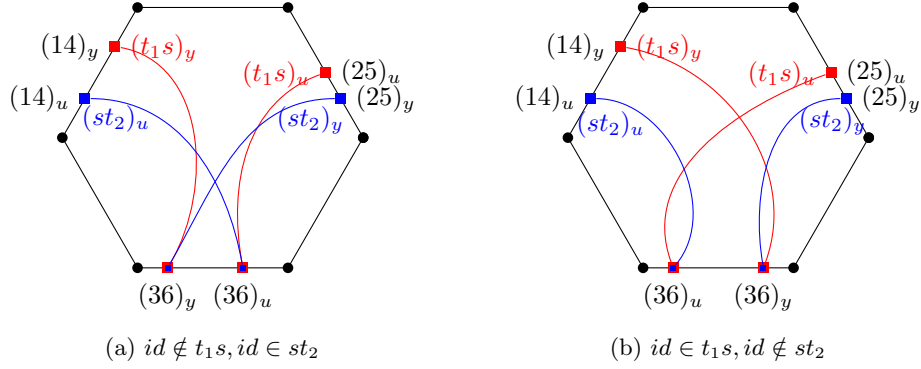

We now analyse in more detail these boundaries, hence 
determining the desired intersection number. 
We first consider  $\mcal_{\S \to \D \times \C_y \times \P^1_u}(s, t_1; id)$. 

Let us analyse the base direction constraint $\mcal_{\S \to \D \times \C_y}(s, t_1; id)$. 
It has two boundary points, which corresponds to Figure\ref{fig:ts-degen} (a),(b), the domain disk $\S$ will have degenerate types $(12b)(3a6b)(a45)$ and $(56b)(1a4b)(a23)$. The corresponding end-points we denote as $(36)_y$ and $(14)_y$ for short. The position of $(36)_y$ on the hexagon edge depends on the moduli of the disk $(3a6b)$ (shown in orange in Figure\ref{fig:ts-degen} (a)) as a disk with 4 marked points. Similarly for $(14)_y$.

Now we treat $\mcal_{\S \to \D \times \P^1_u}(s, t_1; id)$. Since the marked point of $\S$ maps to $u=\infty$, we have only have those degeneration as shown in Figure\ref{fig:ts-degen}(c),(d). The picture only shows the $\C^*_u$ part of $\P^1_u$, with the left end as $u=0$ and right end as $u=\infty$. Consider Figure\ref{fig:ts-degen}(c), the orange shadow indicate the image of the middle component of the degenerated disk $\S$ in the degenerate type $(61b)(2a5b)(a34)$ (See Figure\ref{fig:MSD} for the notation). Note that because of the two out-bending corners at vertex $2$ and $b$, we have freedom to vary the moduli of the domain disk, so that the involution center of $\S$ can map to $u=\infty$. Similarly for Figure \ref{fig:ts-degen}(d), where the two out-bending corners are at $3$ and $6$. 
These are no other boundary degenerations. For example, if we try to get $(14)_u$, the marked point of $\S$ will not be in the disk component containing points $1,4$.

We have identified the boundaries of the moduli
spaces.  The components of the moduli 
with these boundaries are then lines in the hexagon. 
(There may be also components without boundary, i.e.
circles.)
One line ends on sides $(14)$ and $(36)$, the other ends on sides $(25)$ and $(36)$. Depending on the relative position of the two endpoints on edge $(36)$, the two lines (as relative homology class) can have $0$ or $\pm 1$ intersection number (the sign depends on choice of orientations everywhere). This two cases are shown as red lines in the two subfigures in Figure \ref{fig:2-cases}. 

We can run the same analysis for the composition $\mcal_{\S \to \D \times \C_y \times \P^1_u}(t_2, s; id)$, and we get the blue lines in Figure \ref{fig:2-cases}. 
In either perturbation scenario, the difference of the blue intersection number and the red intersection number is $1$.





\subsection{Proof of Main theorem \ref{main theorem}}

We first consider a functor 
$$\psi: FreeElem \to \acal $$
that sends $D_{12} \in Elem(c_1, c_2)$ to the corresponding robust morphism $D_{12}^F$ in $HW(\T_{c_1}, \T_{c_2})$. 
The calculations in Sections \ref{s: bigon relation} and \ref{s: klrw from disk 2} show that  $\psi$ annihilates all the KLRW relations except perhaps the braid relations.   Elementary crossing diagram  of type $[i]-(i)$ and $(i)-(j)$ are not zero divisors in $Gr(KLRW)$, hence in $Gr(\acal)$ by Corollary \ref{stage 2: graded level}, hence are not zero divisors in $\acal$. Thus we may apply 
Lemma \ref{lm: short-cut} to see $\psi$ descends to a functor $KLRW \to \acal$. From Corollary \ref{c:GrHW} and \ref{stage 2: graded level}, we see the hypothesis of Prop \ref{p: KLRW recog} are satisfied, hence we may apply  Prop \ref{p: KLRW recog} to conclude this morphism is an isomorphism.

\clearpage

\bibliographystyle{plain}
\bibliography{ref}

\begin{thebibliography}{10}

\bibitem{abouzaid-smith:khovanov}
Mohammed Abouzaid and Ivan Smith.
\newblock Khovanov homology from {F}loer cohomology.
\newblock {\em Journal of the American Mathematical Society}, 32(1):1--79,
  2019.

\bibitem{aganagic-knot-1}
Mina Aganagic.
\newblock Knot categorification from mirror symmetry, part {I}: Coherent
  sheaves.
\newblock {\em arXiv preprint arXiv:2004.14518}, 2020.

\bibitem{aganagic-knot-2}
Mina Aganagic.
\newblock Knot categorification from mirror symmetry, part {II}: Lagrangians.
\newblock {\em arXiv preprint arXiv:2105.06039}, 2021.

\bibitem{aganagic-icm}
Mina Aganagic.
\newblock Homological knot invariants from mirror symmetry.
\newblock {\em Proc. Int. Cong. Math. 2022, Vol. 3, pp. 2108–2144. EMS Press,
  Berlin, 2023, arXiv:2207.14104}, 2022.

\bibitem{ALR}
Mina Aganagic, Elise LePage, and Miroslav Rapcak.
\newblock Homological link invariants from {F}loer theory.
\newblock {\em arXiv preprint arXiv:2305.13480}, 2023.

\bibitem{BFM}
Roman Bezrukavnikov, Michael Finkelberg, and Ivan Mirkovi{\'c}.
\newblock Equivariant homology and {K}-theory of affine {G}rassmannians and
  {T}oda lattices.
\newblock {\em Compositio Mathematica}, 141(3):746--768, 2005.

\bibitem{Bezrukavnikov-Kaledin}
Roman Bezrukavnikov and Dmitry Kaledin.
\newblock Fedosov quantization in positive characteristic.
\newblock {\em Journal of the American Mathematical Society}, 21(2):409--438,
  2008.

\bibitem{Biran-Cornea-cone}
Paul Biran and Octav Cornea.
\newblock Cone-decompositions of {L}agrangian cobordisms in {L}efschetz
  fibrations.
\newblock {\em Selecta Mathematica}, 23(4):2635--2704, 2017.

\bibitem{braverman2014}
Alexander {Braverman}, Galyna {Dobrovolska}, and Michael {Finkelberg}.
\newblock {Gaiotto-Witten superpotential and Whittaker D-modules on monopoles}.
\newblock {\em arXiv e-prints}, page arXiv:1406.6671, June 2014.

\bibitem{BFN}
Alexander Braverman, Michael Finkelberg, and Hiraku Nakajima.
\newblock Towards a mathematical definition of {C}oulomb branches of
  3-dimensional n=4 gauge theories, {II}.
\newblock {\em Advances in Theoretical and Mathematical Physics},
  22(5):1071--1147, 2018.

\bibitem{BFN-slice}
Alexander Braverman, Michael Finkelberg, and Hiraku Nakajima.
\newblock Coulomb branches of 3d n=4 quiver gauge theories and slices in the
  affine {G}rassmannian.
\newblock {\em Advances in Theoretical and Mathematical Physics},
  23(1):75--166, 2019.

\bibitem{cautis-kamnitzer-2}
Sabin Cautis and Joel Kamnitzer.
\newblock {Knot homology via derived categories of coherent sheaves, I: The
  $\mathfrak{sl}(2)$-case}.
\newblock {\em Duke Mathematical Journal}, 142(3):511 -- 588, 2008.

\bibitem{cautis-kamnitzer-1}
Sabin Cautis and Joel Kamnitzer.
\newblock Knot homology via derived categories of coherent sheaves {II},
  $\mathfrak{sl}(m)$-case.
\newblock {\em Inventiones mathematicae}, 174(1):165--232, 2008.

\bibitem{Colin-Honda-Tian}
Vincent Colin, Ko~Honda, and Yin Tian.
\newblock Applications of higher-dimensional {H}eegaard {F}loer homology to
  contact topology.
\newblock {\em arXiv preprint arXiv:2006.05701}, 2020.

\bibitem{Finkelberg-Tsymbaliuk}
Michael Finkelberg and Alexander Tsymbaliuk.
\newblock Multiplicative slices, relativistic {T}oda and shifted quantum affine
  algebras.
\newblock In {\em Representations and nilpotent orbits of Lie algebraic
  systems}, pages 133--304. Springer, 2019.

\bibitem{gaiottowitten2011}
Davide Gaiotto and Edward Witten.
\newblock {Knot Invariants from Four-Dimensional Gauge Theory}.
\newblock {\em Adv. Theor. Math. Phys.}, 16(3):935--1086, 2012.

\bibitem{gammage2019homological}
Benjamin Gammage, Michael McBreen, and Ben Webster.
\newblock Homological mirror symmetry for hypertoric varieties, {II}.
\newblock {\em arXiv preprint arXiv:1903.07928}, 2019.

\bibitem{GPS1}
Sheel Ganatra, John Pardon, and Vivek Shende.
\newblock Covariantly functorial wrapped {F}loer theory on {L}iouville sectors.
\newblock {\em Publications math{\'e}matiques de l'IH{\'E}S}, pages 1--128,
  2019.

\bibitem{GPS3}
Sheel Ganatra, John Pardon, and Vivek Shende.
\newblock Microlocal {M}orse theory of wrapped {F}ukaya categories.
\newblock {\em Annals of Mathematics}, 199(3):943--1042, 2024.

\bibitem{GPS2}
Sheel Ganatra, John Pardon, and Vivek Shende.
\newblock Sectorial descent for wrapped {F}ukaya categories.
\newblock {\em Journal of the American Mathematical Society}, 37(2):499--635,
  2024.

\bibitem{Givental-mirror}
Alexander Givental.
\newblock A mirror theorem for toric complete intersections.
\newblock In {\em Topological field theory, primitive forms and related
  topics}, pages 141--175. Springer, 1998.

\bibitem{Honda-Tian-Yuan}
Ko~Honda, Yin Tian, and Tianyu Yuan.
\newblock Higher-dimensional {H}eegaard {F}loer homology and {H}ecke algebras.
\newblock {\em arXiv preprint arXiv:2202.05593}, 2022.

\bibitem{jin2022homological}
Xin Jin.
\newblock Homological mirror symmetry for the universal centralizers.
\newblock {\em arXiv preprint arXiv:2206.09035}, 2022.

\bibitem{kaledin2008derived}
Dmitry Kaledin.
\newblock Derived equivalences by quantization.
\newblock {\em Geometric and Functional Analysis}, 6(17):1968--2004, 2008.

\bibitem{KWWY22}
Joel Kamnitzer, Ben Webster, Alex Weekes, and Oded Yacobi.
\newblock Lie algebra actions on module categories for truncated shifted
  {Y}angians.
\newblock In {\em Forum of Mathematics, Sigma}, volume~12, page e18. Cambridge
  University Press, 2024.

\bibitem{Kang-Kashiwara}
Seok-Jin Kang and Masaki Kashiwara.
\newblock Categorification of highest weight modules via
  {K}hovanov-{L}auda-{R}ouquier algebras.
\newblock {\em Inventiones mathematicae}, 190(3):699--742, 2012.

\bibitem{Khovanov-Lauda-diagrammatics-1}
Mikhail Khovanov and Aaron Lauda.
\newblock A diagrammatic approach to categorification of quantum groups {I}.
\newblock {\em Representation Theory of the American Mathematical Society},
  13(14):309--347, 2009.

\bibitem{Khovanov-Lauda-diagrammatics-2}
Mikhail Khovanov and Aaron Lauda.
\newblock A diagrammatic approach to categorification of quantum groups {II}.
\newblock {\em Transactions of the American Mathematical Society}, pages
  2685--2700, 2011.

\bibitem{lipshitz2006cylindrical}
Robert Lipshitz.
\newblock A cylindrical reformulation of {H}eegaard {F}loer homology.
\newblock {\em Geometry \& Topology}, 10(2):955--1096, 2006.

\bibitem{Mak-Smith}
Cheuk~Yu Mak and Ivan Smith.
\newblock Fukaya--{S}eidel categories of {H}ilbert schemes and parabolic
  category {O}.
\newblock {\em Journal of the European Mathematical Society}, 24(9):3215--3332,
  2021.

\bibitem{manolescu}
Ciprian Manolescu.
\newblock {Nilpotent slices, Hilbert schemes, and the Jones polynomial}.
\newblock {\em Duke Mathematical Journal}, 132(2):311 -- 369, 2006.

\bibitem{manolescu2}
Ciprian Manolescu.
\newblock Link homology theories from symplectic geometry.
\newblock {\em Advances in Mathematics}, 211(1):363--416, 2007.

\bibitem{McBreen-Shende-Zhou}
Michael McBreen, Vivek Shende, and Peng Zhou.
\newblock The {H}amiltonian reduction of hypertoric mirror symmetry.
\newblock {\em arXiv preprint arXiv:2405.07955}, 2024.

\bibitem{mcbreen2018homological}
Michael McBreen and Ben Webster.
\newblock Homological mirror symmetry for hypertoric varieties, i: Conic
  equivariant sheaves.
\newblock {\em Geometry \& Topology}, 28(3):1005--1063, 2024.

\bibitem{nadler-shende}
David Nadler and Vivek Shende.
\newblock Sheaf quantization in {W}einstein symplectic manifolds.
\newblock {\em arXiv preprint arXiv:2007.10154}, 2020.

\bibitem{Rouquier-2kac}
Rapha{\"e}l Rouquier.
\newblock 2-{K}ac-{M}oody algebras.
\newblock {\em arXiv preprint arXiv:0812.5023}, 2008.

\bibitem{Seidel-book}
Paul Seidel.
\newblock {\em Fukaya categories and Picard-Lefschetz theory}, volume~10.
\newblock European Mathematical Society, 2008.

\bibitem{SS}
Paul {Seidel} and Ivan {Smith}.
\newblock {A link invariant from the symplectic geometry of nilpotent slices}.
\newblock {\em arXiv Mathematics e-prints}, page math/0405089, May 2004.

\bibitem{shende2021toric}
Vivek Shende.
\newblock Toric mirror symmetry revisited.
\newblock {\em arXiv preprint arXiv:2103.05386}, 2021.

\bibitem{teleman2014gauge}
Constantin Teleman.
\newblock Gauge theory and mirror symmetry.
\newblock {\em arXiv preprint arXiv:1404.6305}, 2014.

\bibitem{Webster-weighted}
Ben Webster.
\newblock Weighted {K}hovanov-{L}auda-{R}ouquier algebras.
\newblock {\em arXiv preprint arXiv:1209.2463}, 2012.

\bibitem{Webster}
Ben {Webster}.
\newblock {Knot invariants and higher representation theory}.
\newblock {\em arXiv e-prints}, page arXiv:1309.3796, September 2013.

\bibitem{webster2019coherent}
Ben Webster.
\newblock Coherent sheaves and quantum {C}oulomb branches {I}: tilting bundles
  from integrable systems.
\newblock {\em arXiv preprint arXiv:1905.04623}, 2019.

\bibitem{webster2022coherent}
Ben Webster.
\newblock Coherent sheaves and quantum {C}oulomb branches {II}: quiver gauge
  theories and knot homology.
\newblock {\em arXiv preprint arXiv:2211.02099}, 2022.

\end{thebibliography}

\end{document}